\documentclass{article}
\usepackage{amsthm}
\usepackage{amsfonts}
\usepackage{amsmath,amssymb}
\usepackage[a4paper,margin=26mm]{geometry}
\usepackage{hyperref}

\renewcommand{\P}{\mathbb{P}}

\renewcommand{\ge}{\geqslant}

\renewcommand{\le}{\leqslant}
\renewcommand{\leq}{\leqslant}

\theoremstyle{plain}
\newtheorem{proposition}{Proposition}
\newtheorem{theorem}{Theorem}
\newtheorem{corollary}{Corollary}
\newtheorem{lemma}{Lemma}
\theoremstyle{definition}
\newtheorem{definition}{Definition}
\newtheorem{example}{Example}
\theoremstyle{remark}
\newtheorem{remark}{Remark}
\begin{document}
\title{Distance-layer structure of the De Bruijn and Kautz digraphs: analysis and application to deflection routing\\
(with examples and remarks)}
\author{J. F\`abrega, J. Mart\'{\i}-Farr\'e and X. Mu\~noz
\thanks{Partially supported by the Ministerio de Ciencia e Innovaci\'on\mbox{$/$}Agencia Estatal de Investigaci\'on, Spain, and the European Regional Development Fund under project  PGC2018-095471-B-I00; and by AGAUR from the Catalan Government under project 2017SGR--1087.} \\
Departament de Matem\`{a}tiques, Universitat Polit\`ecnica de Catalunya\\
		Barcelona, Spain}
	
\maketitle 
\begin{abstract} 
In this paper, we present a detailed study of the reach distance-layer structure of the De Bruijn and Kautz digraphs, and we apply our analysis to the performance evaluation of deflection routing in De Bruijn and Kautz networks. Concerning the distance-layer structure, we provide explicit polynomial  expressions, in terms of the degree of the digraph, for the cardinalities of  some relevant sets of this structure. Regarding the application to defection routing, and as a consequence of our polynomial description of the distance-layer structure, we formulate explicit rational expressions, in terms of the degree of the digraph, for some probabilities of interest in the analysis of this type of routing.

De Bruijn and Kautz digraphs are fundamental examples of digraphs on alphabet and iterated line digraphs. If the topology of the network under consideration corresponds to a digraph of this type, we can perform, in principle, a similar vertex layer description.
\end{abstract}

\section{Introduction}
\label{intro}

Deflection routing \cite{origindeflection} is a routing scheme for bufferless networks based on the fact that if a packet cannot be sent through a given link due to congestion, it is deflected through any other available link (instead of being buffered in the node queue), and the packet is then rerouted to destination. This kind of routing is nowadays interesting in the context of optical networks \cite{haeri2015,NlMu, zheng2015} and on-chip networks \cite{CheTaMi, KuJaSl}. However, its efficiency depends highly on the network topology (as well as on the decision criteria used to deflect packets when collisions appear \cite{fm}). More precisely, the routing efficiency will be determined by how much the distance to the destination increases when a deflection occurs. This question is addressed by considering some probabilities, as studied in Subsection~\ref{application}. Because of this reason, the efficiency in networks with unidirectional links may be worse than in the bidirectional case. Nevertheless, in many cases, directed networks are convenient \cite{KuJaSl, SheZhaGu}.

Despite being known for a long time, active research is still going on on De Bruijn and Kautz digraphs $B(d,D)$ and $K(d,D)$ \cite{kautz,debruijn,kautz-2,kautz-3}, both in graph theory \cite{BoDaHu, DoShMi, GriKaSte, Lichi} and in network engineering \cite{FaMo, MaEtYa, SheLi, ZeGyBl}. This paper is concerned with deflection routing in these kind of networks. 

In order to study the topological properties of $B(d,D)$ and $K(d,D)$ that we need to evaluate the performance of deflection routing, we provide a detailed study of its reach distance-layer structure. We give explicit polynomial  expressions, in terms of the degree of the digraph, for the cardinalities of  some relevant sets of this structure. For instance, if $S_{i}^\star(v)$ denotes the set of vertices at distance $i$ from a given vertex $v$, we show that $|S_{i}^\star(v)|=d^i-a_{i-1}d^{i-1}-\cdots -a_{1} d-a_{0}$,  where the coefficients $a_{k}$ are $0$ or $1$, and are explicitly determined from the sequence representation of $v$. Moreover, if $w$ is a vertex adjacent from $v$, we demonstrate that there are at most two integers $j$ such that the intersection $S_{i}^\star(v) \cap S_j^{\star}(w)$ is nonempty; we show how to determine such values of $j$; and we relate the polynomial description of $|S_{i}^\star(v) \cap S_j^{\star}(w)|$ with that of $|S_{i}^\star(v)|$. 

We apply our results on the distance-layer structure to provide explicit rational expressions, in terms of the degree $d$, of some probabilities of interest in the performance evaluation of deflection routing in $B(d,D)$ and $K(d,D)$.  Moreover, the polynomial description of the distance-layer structure is interesting by itself from a graph theoretical approach, and it can be helpful in other applications of De Bruijn and Kautz digraphs to networks or other engineering fields.

The paper is organized as follows. In Section~\ref{results} we present our results on the distance-layer structure of the set of vertices of $B(d,D)$ and $K(d,D)$ (Subsections~\ref{distance-layer} and \ref{sub: layer}), and on deflection routing (Subsection~\ref{application}). The proofs concerning the distance-layer structure are given in Section~\ref{proofs} (from Subsection~\ref{proof-2.3} to Subsection~\ref{proof propo 2.9-v2-cas 3}). In order to develop all these proofs, we need a collection of technical lemmas (Subsection~\ref{technical-lemmas}) that allow us to understand the distance-layer structure deeply. Finally, in Section~\ref{proofs-2} we prove the results presented in Subsection~\ref{application} on input and transition probabilities in deflection routing (Subsections~\ref{proof p_in}, \ref{proof p_ij}, and \ref{proof p_iD}). To do this, we need several additional technical lemmas on the distance-layer structure (Subsection~\ref{technical-lemmas-defl}). 

An extended abstract of a preliminary version of our work appeared in \cite{endam}.

\setcounter{subsection}{0}
\section{Our results}
\label{results}
Concerning the distance-layer structure of the set of vertices of the De Bruijn and Kautz digraphs we formulate some polynomial expressions (in terms of the degree $d$ of the digraph) for the cardinalities of some relevant sets of this structure. More precisely, let $S_{i}^\star(v)$ be the set of vertices at distance $i$ from a given vertex $v$. We show that $|S_{i}^\star(v)|=d^i-a_{i-1}d^{i-1}-\cdots -a_{1} d-a_{0}$, and  the coefficients $a_{k}\in\{0,1\}$ are explicitly  calculated. Moreover, given $v$, we show that for all vertex $w$ there exists at most one integer $j\ge i$ such that the intersection $S_{i}^\star(v)\cap S_{j}^{\ast}(w)$ is nonempty; and in the case that $w$ is adjacent from $v$, we provide a precise characterization of when $S_{i}^\star(v)\cap S_{j}^{\ast}(w)\ne\emptyset$. Furthermore, if $w$ is adjacent from $v$, we prove that if $S_{i}^\star(v) \cap S_{i-1}^{\ast}(w)\ne\emptyset$, then $|S_{i}^\star(v) \cap S_{i-1}^{\ast}(w)|= d^{i-1}-b_{i-2}d^{i-2}-\ldots -b_{1} d-b_{0}$, and that if $S_{i}^\star(v)\cap S_{j}^{\ast}(w)\neq \emptyset$, then $|S_{i}^\star(v) \cap S_{j}^{\ast}(w)|=d^i-\alpha_{i-1}d^{i-1}-\ldots -\alpha_{1} d-\alpha_{0}$, where the coefficients of these polynomial expressions, $b_{k},\alpha_k\in\{0,1\}$, $0\le k\le i-2$, and $\alpha_{i-1}\in\{0,1,2\}$, are determined from the coefficients $a_k$ of the polynomial expression of $|S_{i}^\star(v)|$.


\subsection{The distance-layer structure of $B(d,D)$ and $K(d,D)$}
\label{distance-layer}
This subsection and the following one are devoted to presenting our results on the characterization of the distance-layer structure of $B(d,D)$ and $K(d,D)$. We will prove our propositions and theorems in Section~\ref{proofs}. In order to elaborate these proofs, we present in Subsection~\ref{technical-lemmas}  several lemmas and remarks that allow us to understand the distance-layer structure comprehensively.

We make use of the well-known sequence representation of the vertices of $B(d,D)$ and $K(d,D)$. Each vertex of the De Bruijn digraph $B(d,D)$ corresponds to a sequence $v=v_1 v_2\cdots v_D$ such that each  element $v_k$ belongs to a base alphabet $A$ of $d$ symbols, and vertex $v$ is adjacent to the $d$ vertices $w=v_2\cdots v_D v_{D+1}$, where $v_{D+1}\in A$. Analogously, each vertex of the Kautz digraph $K(d,D)$ corresponds to a sequence $v=v_1 v_2\cdots v_D$, where now $v_{k}\ne v_{k+1}$, $1\le k< D$, and the base alphabet $A$ has $d+1$ symbols. In $K(d,D)$, vertex $v$ is adjacent to the $d$ vertices $w=v_2\cdots v_D v_{D+1}$, where $v_{D+1}\in A$ and $v_{D+1}\ne v_D$.  The digraphs $B(d,D)$ and $K(d,D)$ are $d$-regular, $d\ge 2$, have diameter $D$, and number of vertices $d^D$ and $d^D+d^{D-1}$, respectively. 

Notice that if $v=v_1v_2\cdots v_iv_{i+1}\cdots v_D$ is the sequence representation of a vertex $v$, then the sequence representation of a generic vertex $u$ for which there exists a walk from $v$ to $u$ of length $i$, $0\le i\le D-1$, is $u=v_{i+1}\cdots v_D \ast\cdots\ast$, where the subsequence  $\ast\cdots\ast$ means that the last $i$ symbols of $u$ can be arbitrarily chosen (in the case $G=K(d,D)$, two consecutive symbols must be different). It is easily checked that between any pair of vertices there exists a walk of length $D$ in $B(d,D)$ and of length $D+1$ in $K(d,D)$. It is also a well-known fact that in $B(d,D)$ and $K(d,D)$ the shortest path between any two vertices is unique. Indeed,  let $v$ and $z$ be distinct vertices  with a sequence representation $v=v_1v_2\cdots v_D$ and $z=z_1z_2\cdots z_D$, respectively. Then, the distance from $v$ to $z$ is $k$ if and only if  $k$ is the smallest integer such that $v=v_1\cdots v_{k}z_1\cdots z_{D-k}$; that is to say, $k$ is the smallest integer such that the last $D-k$ symbols of the sequence representation of $v$ coincide with the first $D-k$ symbols of the sequence representation of $z$. Moreover, if $k\ge 2$, then the shortest path from $v$ to $z$ is $v, u_1,\ldots, u_{k-1}, z$, where the sequence representation of the intermediate vertex $u_i$ is $u_i=v_{i+1}\cdots v_{k}z_1\cdots z_{D-k+i}$, $1\le i\le k-1$.

From now on let $G$ be the digraph under consideration (either $G=B(d,D)$ or $G=K(d,D)$) and  let $V$ denote its vertex set. 

Given $v\in V$,  for $i\ge 0$, let $S_i(v)$ be the set of vertices for which there exists a walk from $v$ of length $i$,  and  let $S_{i}^\star(v)$ denote the set of vertices  at distance $i$ from $v$. From the definition it is clear that $S_0(v)=\{v\}$;  $S_1(v)$ is the set of vertices adjacent from $v$, usually denoted as $\Gamma^+(v)$; $S_{i}^\star(v)=\emptyset$ for $i\ge D+1$; and 
\begin{equation}
\label{layer-previ}
S_{i}^\star(v)=S_i(v)\setminus \left (\bigcup_{k=0}^{i-1}S_{k}(v) \right ) \textrm{ for } 0\le i\le D.
\end{equation}
Moreover, since in $B(d,D)$ there exists a walk of length $D$ between any pair of vertices, if $G=B(d,D)$ and $i\ge D$, then $S_i(v)=V$, and so $|S_i(v)|=d^D$. Analogously, if $G=K(d,D)$ and  $i\ge D+1$, then $S_i(v)=V$ and $|S_i(v)|=d^D+d^{D-1}$, because in $K(d,D)$ there is a walk of length $D+1$ between any pair of vertices. 

To illustrate the sets of vertices $S_i(v)$ and $S_i^\star(v)$ we next consider examples of both De Bruijn and Kautz digraphs. In these examples the cardinality of $S_i^\star(v)$ for some value of $i$ and some vertex $v$ is computed. We observe that these cardinalities have a polynomial expression in terms of the degree $d$ of $G$. One of the purposes of the next section is to point out that these polynomial expressions always exist and to provide a method to compute them.

\begin{example}
\label{example B(d,7)}
Consider the De Bruijn digraph $G=B(d,7)$ and let $v\in V$ be a vertex with sequence representation $v=\alpha\beta\beta\alpha\beta\alpha\beta$, where $\alpha$ and $\beta$ are distinct elements of the symbol alphabet $A$. Let us determine the number of vertices, $|S_6^\star(v)|$, at distance $6$ from such a vertex $v$.

The sets $S_i(v)$ can be described as follows:
\begin{equation}
\label{layer-1}
\begin{array}{l}
S_0(v)=\{u\in V:\, u=\alpha\beta\beta\alpha\beta\alpha\beta\},\\[1mm]
S_1(v)=\{u\in V:\, u=\beta\beta\alpha\beta\alpha\beta\ast\},\\[1mm]
S_2(v)=\{u\in V:\, u=\beta\alpha\beta\alpha\beta\ast\ast\},\\[1mm]
S_3(v)=\{u\in V:\, u=\alpha\beta\alpha\beta\ast\ast\ast\},\\[1mm]
S_4(v)=\{u\in V:\, u=\beta\alpha\beta\ast\ast\ast\ast\},\\[1mm]
S_5(v)=\{u\in V:\, u=\alpha\beta\ast\ast\ast\ast\ast\},\\[1mm]
S_6(v)=\{u\in V:\, u=\beta\ast\ast\ast\ast\ast\ast\},
\end{array}
\end{equation}
where the symbol $\ast$ stands for an arbitrary element of the symbol alphabet $A$. We realize from the sequence representation of the vertices in (\ref{layer-1}) that if $k=1,2,4$, then $S_k(v)\subseteq S_6(v)$; while if $k=0,3,5$, then $S_k(v)\cap S_6(v)=\emptyset$. Therefore we deduce from (\ref{layer-previ}) that the  set of vertices at distance $6$ from $v$ is $S_6^\star(v)=S_6(v)\setminus\left(S_1(v)\cup S_2(v)\cup S_4(v)\right)$. Furthermore, observe that $S_2(v)\subseteq S_4(v)$ and that $S_1(v)\cap S_4(v)=\emptyset$; so we have $S_6^\star(v)=S_6(v)\setminus\left(S_1(v)\cup S_4(v)\right)$. Moreover, $|S_i(v)|=d^i$ if $i\le 6$, because we have $d$ possible choices for each symbol $\ast$. So we conclude that $|S_6^\star(v)|=d^6-d^4-d$. 
\end{example}

\begin{example}
\label{example K(d,10)}
In this second example we consider the Kautz digraph $G=K(d,10)$ and let us calculate $|S_8^\star(v)|$, being $v$  a vertex with sequence representation  $v=\alpha\beta\gamma\alpha\beta\gamma\alpha\beta\alpha\beta$, where $\alpha$, $\beta$ and $\gamma$ stand for different elements of the alphabet $A$. As in Example~\ref{layer-1}, the sets $S_i(v)$ can be described as follows:
\begin{equation}
\label{layer-2}
\begin{array}{l}
S_0(v)=\{u\in V:\, u=\alpha\beta\gamma\alpha\beta\gamma\alpha\beta\alpha\beta\},\\[1mm]
S_1(v)=\{u\in V:\, u=\beta\gamma\alpha\beta\gamma\alpha\beta\alpha\beta\ast\},\\[1mm]
S_2(v)=\{u\in V:\, u=\gamma\alpha\beta\gamma\alpha\beta\alpha\beta\ast\ast\},\\[1mm]
S_3(v)=\{u\in V:\, u=\alpha\beta\gamma\alpha\beta\alpha\beta\ast\ast\ast\},\\[1mm]
S_4(v)=\{u\in V:\, u=\beta\gamma\alpha\beta\alpha\beta\ast\ast\ast\ast\},\\[1mm]
S_5(v)=\{u\in V:\, u=\gamma\alpha\beta\alpha\beta\ast\ast\ast\ast\ast\},\\[1mm]
S_6(v)=\{u\in V:\, u=\alpha\beta\alpha\beta\ast\ast\ast\ast\ast\ast\},\\[1mm]
S_7(v)=\{u\in V:\, u=\beta\alpha\beta\ast\ast\ast\ast\ast\ast\ast\},\\[1mm]
S_8(v)=\{u\in V:\, u=\alpha\beta\ast\ast\ast\ast\ast\ast\ast\ast\}.
\end{array}
\end{equation}
(Remember that, since $G$ is a Kautz digraph, two successive symbols in the above sequence representations must be different.) We can verify that if $k=0,3,6$, then $S_k(v)\subseteq S_8(v)$; while if $k=1,2,4,5,7$, then $S_k(v)\cap S_8(v)=\emptyset$. In consequence, from (\ref{layer-previ}) we get that the set of vertices at distance $8$ from $v$ is $S_8^\star(v)=S_8(v)\setminus\left(S_0(v)\cup S_3(v)\cup S_6(v)\right)$; and, since  $S_0(v)\cap S_3(v)= S_0(v)\cap S_6(v)= S_3(v)\cap S_6(v)=\emptyset$, the expression of $S_8^\star(v)$ cannot be further simplified. Besides, as in the previous example, we have  $|S_i(v)|=d^i$ if $i\le 8$ (notice that, despite the fact that two consecutive symbols in the sequence representation of a vertex must be different, in (\ref{layer-2}) we have again $d$ possible choices for each symbol $\ast$, because the symbol alphabet $A$ has cardinality $d+1$). Therefore the number of vertices at distance $8$ from  $v$ has the polynomial expression  $|S_8^\star(v)|=d^8-d^6-d^3-1$. 
\end{example}

%

\subsection{Polynomial description of the distance-layer structure}
\label{sub: layer}
The first goal of this subsection is to present a  polynomial description of the cardinality of the set $S_{i}^\star(v)$, where the polynomial has degree $i$,   variable $d$ (the degree of the digraph), and coefficients $0$ or $1$.  In order to obtain this description, we introduce the following definition. 
\begin{definition}
\label{def_S_kj}
Given $v\in V$ and two integers $k, i$ such that $0\le i\le D$ and $0\leq k\leq i$, let 
$$
S_{k,i}(v)= \left\{\begin{array}{ll}
S_{k}(v), & \textrm{if } S_{k}(v)\subseteq S_{i}(v) \textrm{  and for all } j,\ k<j<i,\\[1mm]
&\, \textrm{such that }S_{j}(v)\subseteq S_{i}(v)  \textrm{ we have } S_{k}(v)\not \subseteq S_{j}(v);\\[2mm]
\emptyset,&   \textrm{otherwise. }
\end{array}\right.
$$
\end{definition}

\begin{remark}
\label{rem_def_S_kj}
It follows from Definition~\ref{def_S_kj} that $S_{i,i}(v)=S_i(v)$ for any $v\in V$. Moreover, if $G=K(d,D)$, then $S_{i-1,i}(v)=\emptyset$ for any $v\in V$, because $S_{i-1}(v)\not\subseteq S_{i}(v)$. In the case $G=B(d,D)$ there are vertices $v$ such that $S_{i-1,i}(v)=S_{i-1}(v)$ and vertices $v$ for which $S_{i-1,i}(v)=\emptyset$.
\end{remark}

\begin{proposition}
\label{propo 2.3}
Let $v\in V$. Then, 
$$|S_{i}^\star(v)|=d^i-a_{i-1}d^{i-1}-\cdots -a_{1} d-a_{0},$$  
where the coefficients $a_{k}$ are $0$ or $1$, and $a_{k}=1$ if and only if $S_{k,i}(v)\neq\emptyset$. In particular, if $v=v_1v_2\cdots v_D$, then $a_{i-1}=1$ if and only if $G=B(d,D)$ and $v_i=v_{i+1}=\cdots=v_D$.
\end{proposition}

The proof of this result is given in Subsection~\ref{proof-2.3}. To illustrate the use of Proposition~\ref{propo 2.3} let us recalculate the cardinalities computed in Examples~\ref{example B(d,7)} and \ref{example K(d,10)}.

 \begin{example}
 \label{both} 
 On one hand, if $G=B(d,7)$,  $v=\alpha\beta\beta\alpha\beta\alpha\beta$, and $k<6$, then $S_k(v)\subseteq S_6(v)$ if and only if $k=1,2,4$. Hence $S_{1,6}(v)=S_1(v)$, because $S_1(v)\not\subseteq S_j(v)$ if $1<j<6$; $S_{2,6}(v)=\emptyset$, because $S_2(v)\subseteq S_4(v)$; and $S_{4,6}(v)=S_4(v)$, because $S_4(v)\not\subseteq S_5(v)$. Therefore we have $a_1=a_4=1$ and $a_2=a_3=a_5=0$, and hence $|S_6^\star(v)|=d^6-d^4-d$. 

 On the other hand, if $G=K(d,10)$,  $v=\alpha\beta\gamma\alpha\beta\gamma\alpha\beta\alpha\beta$, and $k<8$, then the only subsets $S_{k,8}(v)$ which are nonempty are $S_{0,8}(v)=S_0(v)$, $S_{3,8}(v)=S_3(v)$, and $S_{6,8}(v)=S_6(v)$. So we conclude that $|S_8^\star(v)|=d^8-d^6-d^3-1$. 
 \end{example}

 In the following example we compute all the cardinalities $|S_i^\star(v)|$ in $K(d,4)$.

 \begin{example}
 \label{example K(d,4)} 
 In $K(d,4)$ we can distinguish five different types of vertices according to the structure of its sequence representation. Indeed,  let us consider the partition $V=\bigcup_{r=1}^5 \mathcal{V}_r$ defined by $\mathcal{V}_1=\{v\in V: v=\alpha\beta\alpha\beta\}$, $\mathcal{V}_2=\{v\in V: v=\alpha\beta\alpha\gamma\}$, $\mathcal{V}_3=\{v\in V: v=\alpha\beta\gamma\alpha\}$, $\mathcal{V}_4=\{v\in V: v=\alpha\beta\gamma\beta\}$, and $\mathcal{V}_5=\{v\in V: v=\alpha\beta\gamma\delta\}$, where the symbols $\alpha$, $\beta$, $\gamma$, and $\delta$ are different (so, we assume $d\ge 3$). By Proposition~\ref{propo 2.3}, if $v\in\mathcal{V}_r$, then the number of vertices at distance $i$ from $v$ has a polynomial expression of the form $|S_{i}^\star(v)|=d^i-a^{(r,i)}_{i-1}d^{i-1}-\cdots -a^{(r,i)}_{1} d-a^{(r,i)}_{0}$, where $a^{(r,i)}_k$ is either $0$ or $1$, and $a^{(r,i)}_{k}=1$ if and only if $S_{k,i}(v)\neq\emptyset$. For instance, if $v=\alpha\beta\alpha\beta\in \mathcal{V}_1$, then $S_2(v)=\{u\in V: u=\alpha\beta\ast\ast\}$ and hence  the subset $S_{k,2}(v)$  is nonempty only for $k=0$. Therefore, if $v\in\mathcal{V}_1$, then $|S^\star_2(v)|= d^2-1$. Or, if $v=\alpha\beta\gamma\beta\in\mathcal{V}_4$, then $S_3(v)=\{u\in V: u=\beta\ast\ast\ast\}$ and $S_{k,3}(v)$ is nonempty only for $k=1$. Hence  $|S^\star_3(v)|=d^3-d$. In Table~\ref{table 1} we summarize the value of all these coefficients $a^{(r,i)}_k$, as well as the polynomial expressions of $|S_i^\star(v)|$, $1\le i\le 4$.
 \begin{table}[ht]
 \caption{Coefficients and cardinalities $|S_i^\star(v)|$ in $K(d,4)$}
 \label{table 1}
 {\small
 \begin{center}
 \begin{tabular}{|c|c|c|c|c|c|}
 \hline
 &$v\in\mathcal{V}_1$&$v\in\mathcal{V}_2$&$v\in\mathcal{V}_3$&$v\in\mathcal{V}_4$&$v\in\mathcal{V}_5$\\ [1mm]\hline
 &$a_0^{(1,1)}=0$&$a_0^{(2,1)}=0$&$a_0^{(3,1)}=0$&$a_0^{(4,1)}=0$&$a_0^{(5,1)}=0$ \\[2mm] 
 $|S_1^\star(v)|$&$d$&$d$&$d$&$d$&$d$\\[2mm]  \hline
  &$a_1^{(1,2)}=0$&$a_1^{(2,2)}=0$&$a_1^{(3,2)}=0$&$a_1^{(4,2)}=0$&$a_1^{(5,2)}=0$ \\[1mm] 
   &$a_0^{(1,2)}=1$&$a_0^{(2,2)}=0$&$a_0^{(3,2)}=0$&$a_0^{(4,2)}=0$&$a_0^{(5,2)}=0$ \\[2mm]  
 $|S_2^\star(v)|$&$d^2-1$&$d^2$&$d^2$&$d^2$&$d^2$  \\[2mm]  \hline
  &$a_2^{(1,3)}=0$&$a_2^{(2,3)}=0$&$a_2^{(3,3)}=0$&$a_2^{(4,3)}=0$&$a_2^{(5,3)}=0$ \\[1mm] 
   &$a_1^{(1,3)}=1$&$a_1^{(2,3)}=0$&$a_1^{(3,3)}=0$&$a_1^{(4,3)}=1$&$a_1^{(5,3)}=0$  \\[1mm] 
   &$a_0^{(1,3)}=0$&$a_0^{(2,3)}=0$&$a_0^{(3,3)}=1$&$a_0^{(4,3)}=0$&$a_0^{(5,3)}=0$ \\[2mm] 
 $|S_3^\star(v)|$&$d^3-d$&$d^3$&$d^3-1$&$d^3-d$&$d^3$ \\[2mm]  \hline
  &$a_3^{(1,4)}=0$&$a_3^{(2,4)}=0$&$a_3^{(3,4)}=0$&$a_3^{(4,4)}=0$&$a_3^{(5,4)}=0$ \\[1mm] 
   &$a_2^{(1,4)}=1$&$a_2^{(2,4)}=1$&$a_2^{(3,4)}=1$&$a_2^{(4,4)}=1$&$a_2^{(5,4)}=1$ \\[1mm] 
   &$a_1^{(1,4)}=0$&$a_1^{(2,4)}=1$&$a_1^{(3,4)}=1$&$a_1^{(4,4)}=0$&$a_1^{(5,4)}=1$  \\[1mm]
   &$a_0^{(1,4)}=0$&$a_0^{(2,4)}=1$&$a_0^{(3,4)}=0$&$a_0^{(4,4)}=1$&$a_0^{(5,4)}=1$  \\[2mm] 
 $|S_4^\star(v)|$&$d^4-d^2$&$d^4-d^2-d-1$&$d^4-d^2-d$&$d^4-d^2-1$&$d^4-d^2-d-1$  \\  \hline
 \end{tabular}
 \end{center}
 }
 \end{table}
 \end{example}

For the application to deflection routing, in addition to the polynomial description of $|S_{i}^\star(v)|$, we are also interested in the polynomial description of $|S_{i}^\star(v) \cap S_j^{\ast}(w)|$ when $w$ is a vertex adjacent from $v$. The following three propositions deal with this issue. 

Let $v\in V$ and $w\in S_1(v)$, and let $i\ge 0$. By the triangular inequality we have $S_i^\star(v)\cap S_j^\star(w)=\emptyset$ if $j<i-1$. Therefore, since $V=\bigcup_{j=0}^D S_j^\star(w)$, we conclude that
$$S_i^\star(v)=\bigcup_{j=0}^D \left(S_i^\star(v)\cap S_j^\star(w)\right)=\bigcup_{j=i-1}^D \left(S_i^\star(v)\cap S_j^\star(w)\right).$$

First, we demonstrate in Proposition~\ref{v2-first} that there are at most two integers $j\ge i-1$ such that the intersection $S_{i}^\star(v) \cap S_j^{\star}(w)$ is nonempty. After this, in Propositions~\ref{v2-second} and \ref{v2-second-cas 2} we show how to determine such values of $j$. Finally, in Theorems~\ref{propo 2.9-v2-cas 1}, \ref{propo 2.9-v2-cas 2} and \ref{propo 2.9-v2-cas 3} we relate the polynomial description of $|S_{i}^\star(v)|$ with that of $|S_{i}^\star(v) \cap S_j^{\star}(w)|$. 

We will use the following notation. If $v\in V$, then $v_{[i,j]}$ denotes the subsequence $v_iv_{i+1}\cdots v_j$ of the sequence representation $v=v_1v_2\cdots v_D$. In particular, $v_{[i,i]}=v_i$ is the $i$-th element of this sequence.

\begin{proposition}
\label{v2-first}
Let $v\in V$ and let $i\le D$. Then for all vertex $w\in V$  there exists at most one integer $j_0$, $i\le j_0\le D$, such that $S_{i}^\star(v) \cap S_{j_0}^{\star}(w)\ne \emptyset$. In particular, if $w\in S_1(v)$, then
\begin{enumerate}
\item either there exists a unique integer $j_0$, $i\le j_0\le D$, such that the intersection $S_{i}^\star(v) \cap S_{j_0}^{\star}(w)$  is nonempty, and so
$$
\left\{\begin{array}{l}
S_i^\star(v)=\left(S_{i}^\star(v) \cap S_{i-1}^{\star}(w)\right)\cup \left(S_{i}^\star(v) \cap S_{j_0}^{\star}(w)\right) \textrm{ if } S_{i}^\star(v) \cap S_{i-1}^{\star}(w)\ne\emptyset,\\[2mm]
S_i^\star(v)=S_{i}^\star(v) \cap S_{j_0}^{\star}(w) \textrm{ if } S_{i}^\star(v) \cap S_{i-1}^{\star}(w)=\emptyset;
\end{array}
\right.
$$
\item or, for all integer $j$, $i\le j\le D$, the intersection $S_{i}^\star(v) \cap S_{j}^{\star}(w)$  is empty, and so $S_i^\star(v)= S_{i}^\star(v) \cap S_{i-1}^{\star}(w)$.
\end{enumerate}
\end{proposition}

\begin{proposition}
\label{v2-second}
Assume $d\ge 3$. Let $v\in V$, $w\in S_1(v)$, and let $i\le D$. Then,
\begin{enumerate}
\item If $i=j=D$, then  $S_{i}^\star(v) \cap S_{j}^{\star}(w)\ne\emptyset$.
\item If $j\ge i\ne D$, then $S_{i}^\star(v) \cap S_{j}^{\star}(w)\ne\emptyset$ if and only if $S_i(v)\cap S_j(w)\neq \emptyset$ and  $S_i(v)\not\subseteq S_{k,j}(w)$ for  $i \le  k < j$.
\item The intersection   $S_{i}^\star(v) \cap S_{i-1}^{\star}(w)$ is empty if and only if $G=B(d,D)$ and $v_i=v_{i+1}=\cdots=v_D=w_D$. Furthermore,  if $S_{i}^\star(v)\cap S_{i-1}^{\ast}(w)=\emptyset$, then $S_{i}^\star(v)\cap S_{i}^{\ast}(w)\ne\emptyset$.
\item There exists a unique integer $j$, $i\le j\le D$, such that the intersection $S_{i}^\star(v) \cap S_{j}^{\star}(w)$  is non-empty.
\end{enumerate}
\end{proposition}

The condition $d\ge 3$ cannot be removed from the hypothesis of Proposition~\ref{v2-second}. Namely, in Remarks~\ref{rem1-v2-second} and \ref{rem2-v2-second} we show that if $d=2$ and $G=B(d,D)$, then statements (1) and (2) of Proposition~\ref{v2-second} do  not necessarily hold. The case $d=2$ is completely studied in Proposition~\ref{v2-second-cas 2}.

 \begin{remark}
 \label{rem1-v2-second}
 Let us show that if $d=2$, then statement (1) of Proposition~\ref{v2-second} does not necessarily hold; that is, we can have $S_{D}^\star(v) \cap S_{D}^{\star}(w)=\emptyset$. Consider for example the digraph $B(2,4)$ with symbol alphabet $A=\{\alpha,\beta\}$, and let $v$ and $w\in S_1(v)$ be vertices with sequence representation $v=\alpha\beta\alpha\alpha$ and $w=\beta\alpha\alpha\beta$, respectively. It is easily checked that $S_4^\star(v)=V\setminus\left(S_3(v)\cup S_1(v)\right)$ and that $S_4^\star(w)=V\setminus\left(S_3(w)\cup S_2(w)\cup S_1(w)\right)$. Moreover we have $S_1(v)\subseteq S_3(w)$, $S_2(w)\subseteq S_3(v)$, and $S_1(w)\subseteq S_3(v)$.
 Therefore 
 \begin{align*}
 S_{4}^\star(v)\cap S_{4}^{\ast}(w) &=V\setminus\left(S_3(v)\cup S_1(v)\cup S_3(w)\cup S_2(w)\cup S_1(w)\right)\\
 &=V\setminus\left(S_3(v)\cup  S_3(w)\right).
 \end{align*}
 But $S_3(v)=\{u\in V:\ u=\alpha\ast\ast\ast\}$ and $S_3(w)=\{u\in V:\ u=\beta\ast\ast\ast\}$, and thus we have $S_3(v)\cup S_3(w)=V$.  Therefore $S_{4}^\star(v)\cap S_{4}^{\ast}(w)=\emptyset$.
 \end{remark}

 \begin{remark}
 \label{rem2-v2-second}
 Now let us see that if $d=2$, then statement (2) of Proposition~\ref{v2-second} does not necessarily hold; that is,
 there exist $j\ge i\ne D$ such that $S_{i}^\star(v) \cap S_{j}^{\star}(w)=\emptyset$,  $S_i(v)\cap S_j(w)\neq \emptyset$ and $S_i(v)\not\subseteq S_{k,j}(w)$ for  $i \le  k < j$. Indeed, consider for instance $G=B(2,10)$ with symbol alphabet $A=\{\alpha,\beta\}$, and let $v$ and $w\in S_1(v)$ be the vertices which sequence representation  is $v=\alpha\beta\alpha\alpha\alpha\alpha\alpha\alpha\alpha\alpha$ and $w=\beta\alpha\alpha\alpha\alpha\alpha\alpha\alpha\alpha\beta$, respectively. Set $j=D=10$ and $i=3$. Since $G=B(2,10)$ we have $S_{10}(w)=V$ and hence $S_i(v)\cap S_j(w)\neq \emptyset$. We can easily check that for $3\le k< 10$ the only nonempty sets $S_{k,10}(w)$ are $S_8(w)$ and $S_9(w)$. Moreover, since $S_3(v)=\{u\in V:\ u=\alpha\alpha\alpha\alpha\alpha\alpha\alpha\ast\ast\ast\}$, $S_8(w)=\{u\in V:\ u=\alpha\beta\ast\ast\ast\ast\ast\ast\ast\ast\}$, and $S_9(w)=\{u\in V:\ u=\beta\ast\ast\ast\ast\ast\ast\ast\ast\ast\}$, we get that $S_3(v)\not\subseteq S_8(w)$ and $S_3(v)\not\subseteq S_9(w)$. Hence we have $S_3(v)\not\subseteq S_{k,10}(w)$ for $3\le k< 10$. However, let us see that $S_{i}^\star(v) \cap S_{j}^{\star}(w)=\emptyset$. Indeed, from (\ref{layer-previ}) we deduce that
 $$
 S_{3}^\star(v)\cap S_{10}^{\ast}(w)\subseteq \big ( S_{3}(v)\cap S_{10}(w)\big)\setminus \big ( S_{2}(v)\cup S_{2}(w)\big);
 $$
 and thus, since $S_{10}(w)=V$, $S_3(v)=\{u\in V:\ u=\alpha\alpha\alpha\alpha\alpha\alpha\alpha\ast\ast\ast\}$, $S_2(v)=\{u\in V:\ u=\alpha\alpha\alpha\alpha\alpha\alpha\alpha\alpha\ast\ast\}$, and $S_2(w)=\{u\in V:\ u=\alpha\alpha\alpha\alpha\alpha\alpha\alpha\beta\ast\ast\}$, we conclude that
 $$S_{3}^\star(v)\cap S_{10}^{\ast}(w)\subseteq S_3(v)\setminus\left(S_2(v)\cup S_2(w)\right)=\emptyset,$$
 because $S_2(v)\cup S_2(w)=S_3(v)$. In particular, $S_{3}^\star(v)\cap S_{10}^{\ast}(w)=\emptyset$, as claimed.
 \end{remark}

\begin{proposition}
\label{v2-second-cas 2}
Assume $d=2$. Let $v\in V$, $w\in S_1(v)$, and let $i\le D$. Then,
\begin{enumerate}
\item  If $i=j=D$, then $S_{i}^\star(v) \cap S_{j}^{\star}(w)\ne\emptyset$ if and only if $G=K(d,D)$ or $v_D=w_D$.
\item If $j\ge i\ne D$, then $S_{i}^\star(v) \cap S_{j}^{\star}(w)\ne\emptyset$ if and only if $S_i(v)\cap S_j(w)\neq \emptyset$,  $S_i(v)\not\subseteq S_{k,j}(w)$ for  $i \le  k < j$, and one of the following conditions holds:
\begin{enumerate}
\item $j<D$;
\item $j=D$, and $v_{[i,D-1]}\ne v_{[i+1,D]}$ or $S_{i-1,j}(w)=\emptyset$.
\end{enumerate} 
\item The intersection   $S_{i}^\star(v) \cap S_{i-1}^{\star}(w)$ is empty if and only if $G=B(d,D)$ and $v_i=v_{i+1}=\cdots=v_D=w_D$. Furthermore,  if $S_{i}^\star(v)\cap S_{i-1}^{\ast}(w)=\emptyset$, then $S_{i}^\star(v)\cap S_{i}^{\ast}(w)\ne\emptyset$.
\item The intersection $S_{i}^\star(v) \cap S_{j}^{\star}(w)$  is empty for all integer $j$, $i\le j\le D$, if and only if $G=B(d,D)$ and $v_i=v_{i+1}=\cdots=v_D\ne w_D$. 
\end{enumerate}
\end{proposition}

Next we show that, whenever the intersection $S_{i}^\star(v) \cap S_j^{\ast}(w)$ is nonempty, its cardinality has a polynomial expression of the form 
$$
\big |S_{i}^\star(v) \cap S_j^{\ast}(w) \big |=d^i-b_{i-1}d^{i-1}-\ldots -b_{1} d-b_{0},
$$ 
where the coefficients $b_{k}$ are determined from the coefficients $a_k$ of the polynomial expression of $|S_{i}^\star(v)|$. The coefficients $b_{k}$ are computed in the following theorems, that will be  proved in Section~\ref{proofs}. 

\begin{theorem}
\label{propo 2.9-v2-cas 1}
Let $v\in V$, $w\in S_1(v)$, and let $1\le i\le D$. Assume that $S_{i}^\star(v) \cap S_{i-1}^{\ast}(w)\ne\emptyset$ and that 
$S_{i}^\star(v) \cap S_{j_0}^{\ast}(w)\ne\emptyset$ for some $i\le j_0\le D$. Then, 
$$
S_i^\star(v)=\left(S_{i}^\star(v) \cap S_{i-1}^{\star}(w)\right)\cup \left(S_{i}^\star(v) \cap S_{j_0}^{\star}(w)\right),
$$
and so 
$$
\big |S_i^\star(v)\big |=\big |S_{i}^\star(v) \cap S_{i-1}^{\star}(w)\big |+ \big |S_{i}^\star(v) \cap S_{j_0}^{\star}(w)\big |.
$$
Moreover, if 
$$\big |S_{i}^\star(v) \big |=d^i-a_{i-1}d^{i-1}-\ldots -a_{1} d-a_{0}$$ 
is the polynomial expression of $|S_{i}^\star(v)|$ given in Proposition~\ref{propo 2.3}, then  $|S_{i}^\star(v) \cap S_{i-1}^{\ast}(w) \big |$ and $|S_{i}^\star(v) \cap S_{j_0}^{\ast}(w) \big |$ have polynomial expressions
$$
\big |S_{i}^\star(v) \cap S_{i-1}^{\ast}(w) \big | = d^{i-1}-b_{i-2}d^{i-2}-\ldots -b_{1} d-b_{0},
$$
$$
\big |S_{i}^\star(v) \cap S_{j_0}^{\ast}(w) \big | = d^i-(a_{i-1}+1)d^{i-1}-(a_{i-2}-b_{i-2})d^{i-2} -\ldots -(a_1-b_{1}) d-(a_0-b_{0}),
$$
where $b_k\in\{0,1\}$ and $b_k=1$  if and only if $a_k=1$ and $v_{D-i+k+1}=w_D$.
\end{theorem}

\begin{remark}
We know from Proposition~\ref{propo 2.3} that $a_k\in\{0,1\}$. Since $b_k\in\{0,1\}$ and $b_k=1$ only if $a_k=1$, we conclude that, for $0\le k\le i-2$, the coefficient $a_k-b_k$ in the polynomial expression of $|S_{i}^\star(v) \cap S_{j_0}^{\ast}(w) \big |$ is also either $0$ or $1$. Moreover, the coefficient $a_{i-1}+1$ in this polynomial expression is $1$ or $2$. More precisely, from Proposition~\ref{propo 2.3}, we have $a_{i-1}+1=2$ if and only if $G=B(d,D)$ and $v_i=v_{i+1}=\cdots=v_D$.
\end{remark}

\begin{theorem}
\label{propo 2.9-v2-cas 2}
Let $v\in V$, $w\in S_1(v)$, and let $1\le i\le D$. Assume that $S_{i}^\star(v) \cap S_{i-1}^{\ast}(w)=\emptyset$. Then 
$S_{i}^\star(v)=S_{i}^\star(v)\cap S_{i}^{\ast}(w),$ 
and so  
$$
\big |S_{i}^\star(v) \cap S_{i}^{\ast}(w) \big |=\big |S_{i}^\star(v) \big |=d^i-a_{i-1}d^{i-1}-\ldots -a_{1} d-a_{0}.$$
\end{theorem}

\begin{theorem}
\label{propo 2.9-v2-cas 3}
Let $v\in V$, $w\in S_1(v)$, and let $1\le i\le D$. Assume that $S_{i}^\star(v) \cap S_{j}^{\star}(w)=\emptyset$ for all $i\le j\le D$. Then $S_i^\star(v)= S_{i}^\star(v) \cap S_{i-1}^{\star}(w)$ and 
$$
\big |S_{i}^\star(v) \cap S_{i-1}^{\ast}(w) \big |= \big |S_{i}^\star(v) \big |=d^i-a_{i-1}d^{i-1}-\ldots -a_{1} d-a_{0},
$$
where, in this case, we have $d=2$ and $a_{i-1}=1$. Therefore $\left|S_{i}^\star(v)\cap S_{i-1}^{\ast}(w)\right|$ can be equivalently expressed as
$$
\big |S_{i}^\star(v) \cap S_{i-1}^{\ast}(w) \big |=d^{i-1}-a_{i-2}d^{i-2}\ldots -a_{1} d-a_{0}.
$$
\end{theorem}

 \begin{remark}
 Observe that under the hypothesis of Theorem~\ref{propo 2.9-v2-cas 3} we have the inequality $\left |S_{i}^\star(v) \cap S_{i-1}^{\ast}(w) \right |= \left |S_{i}^\star(v) \right |\le d^{i-1}$ (where $d=2$), as it should be, because $\left |S_{i}^\star(v) \cap S_{i-1}^{\ast}(w) \right |\le \left |S_{i-1}^{\ast}(w) \right |\le\left |S_{i-1}(w) \right |=d^{i-1}$.
 \end{remark}

\subsection{Application to deflection routing}
\label{application}
The authors proposed in \cite{fm} an analytical model for evaluating the performance of deflection routing schemes under different deflection criteria. In that model, a Markov chain is defined with states $0, 1, \ldots, D,$ corresponding to the possible distances that a packet may be to its destination ($D$ stands for the diameter of the network), and such that the transition probabilities depend on the deflection criteria and the network topology. 

In this paper, we determine for the case of $B(d,D)$ and $K(d,D)$ the following two probabilities that appear in the formulation \cite{fm}:
\begin{itemize}
\item \emph{Input probability} $\P_{\mathsf{in}}(i)$: 
Given a vertex $v$ selected uniformly at random, let $\P_{\mathsf{in}}(i)$ be the probability that another distinct vertex $v'$, also selected uniformly at random, be at distance $i$ from $v$.

\item \emph{Transition probability} $\P_{\mathsf{t}}(i,j)$: Suppose that a packet with destination vertex $z$ is deflected when visiting an intermediate vertex at a distance $i$ to $z$. We denote by $\P_{\mathsf{t}}(i,j)$ the probability that the new distance to $z$ (after the deflection has occurred) be $j$.
\end{itemize}

This subsection applies our results on the distance-layer structure of $B(d, D)$ and $K(d, D)$ to obtain explicit rational expressions, in terms of the degree $d$, for these probabilities. We provide the proofs of the corresponding propositions and theorems in Section~4. 

To calculate $\P_{\mathsf{in}}(i)$ and $\P_{\mathsf{t}}(i,j)$ we need to introduce a suitable partition of the vertex set of the digraph, classifying the vertices according to their sequence representation. In this way, we consider in $V$ an equivalence relation  $\sim$ defined by $v=v_1 v_2\ldots v_D\sim v'=v'_1 v'_2\ldots v'_D$  if and only if there exists a permutation $\sigma$ of the symbol alphabet $A$ such that  $\sigma(v_k)=v'_k$, $1\le k\le D$. Notice that  two equivalent vertices have a sequence representation with the same number $s$ of distinct symbols, where $1\le s\le\min{(d,D)}$ if $G=B(d,D)$  and $2\le s\le\min{(d+1,D)}$ if $G=K(d,D)$. Moreover, let $n_s$ be the number of equivalence classes in which  the number of distinct symbols in the sequence representation of the  vertices is $s$. Therefore, the partition of $V$ induced by the relation $\sim$ is
\begin{equation}
\label{part}
V=\bigcup_s \left(V_{s,1}\cup\cdots\cup V_{s,n_s}\right),
\end{equation}
where $V_{s,1}\cup\cdots\cup V_{s,n_s}$ is the set of vertices having $s$ distinct symbols in their sequence representation. Observe that $|V_{s,j}|$, $1\le j\le n_s$, has the polynomial expression
$$|
V_{s,j}|= \left\{ \begin{array}{ll}
d(d-1)\cdots (d-s+1) & \textrm{if } G=B(d,D), \\[2mm]
(d+1)d(d-1)\cdots (d-s+2) & \textrm{if } G=K(d,D).
\end{array}\right .
$$
Since $|V|=\sum_s\sum_j |V_{s,j}|$, we get
$$
\left\{\begin{array}{ll}
\displaystyle\sum_{s=1}^{\min{(d,D)}} n_s\, d(d-1)\cdots (d-s+1)=d^{D},& \textrm{if } G=B(d,D);\\[6mm]
\displaystyle\sum_{s=2}^{\min{(d+1,D)}} n_s\, (d+1)d(d-1)\cdots (d-s+2)=d^{D}+d^{D-1}, & \textrm{if } G=K(d,D). 
\end{array}\right.
$$
The parameter $n_s$ is independent of the cardinality of the symbol alphabet $A$, or, equivalently, $n_s$ is independent of the degree $d$ of the digraph. Taking into account this fact and  evaluating the above identities for $d=1,2,\ldots$, the values of $n_s$ can be recursively computed. For instance, if $G=B(d,D)$, then the first nonzero values of $n_s$ are $n_1=1$, $n_2=(2^D-1)/2$, $n_3=(3^{D-1}-2^D+1)/2$, $\ldots$; and if $G=K(d,D)$, then $n_2=1$, $n_3=2^{D-2}-1$, $n_4=(3^{D-2}-2^{D-1}+1)/2$, $\ldots$. Therefore, the total number of classes in the partition (\ref{part}), $l=\sum_s n_s$, is independent of the degree $d$ of $G$ and  can be calculated in terms of the diameter $D$. From now on, for simplicity, we will also denote the partition (\ref{part}) as
\begin{equation}
\label{part-2}
V=\mathcal{V}_1\cup\cdots\cup\mathcal{V}_l,
\end{equation}
and we will use both (\ref{part}) and (\ref{part-2}).

At this point, using the partition and the layer structure of the digraph, we present our results on input and transition probabilities. These results, Theorem~\ref{p_in} and Theorem~\ref{p_ij}, will be proved in  Subsection~\ref{proof p_in} and Subsection~\ref{proof p_ij}, respectively.

Expressing the input probability as $\sum_r \P_{\mathsf{in}}(i\mid v \in \mathcal{V}_r)\, \P(v \in \mathcal{V}_r)$ we obtain the following result that provides a rational description of $\P_{\mathsf{in}}(i)$ in terms of the degree $d$ of the digraph.
\begin{theorem}
\label{p_in}
For any choice of the vertices $v^{(1)},\ldots, v^{(l)}$, where $v^{(r)}\in\mathcal{V}_r$, the input probability $\P_{\mathsf{in}}(i)$ is given by
$$
\P_{\mathsf{in}}(i)=\sum_{r=1}^l \frac{|S_{i}^\star\left(v^{(r)}\right)|}{(|V|-1)}\cdot\frac{|\mathcal{V}_r|}{|V|},
$$
and has the following rational expression:
$$
\P_{\mathsf{in}}(i)=\sum_{r=1}^l\frac{ |\mathcal{V}_r|}{|V|(|V|-1)}
\left ( d^i-a^{(r,i)}_{i-1}d^{i-1}-\cdots -a^{(r,i)}_{1} d-a^{(r,i)}_{0} \right),
$$
where  $a^{(r,i)}_{k}\in\{0,1\}$. More precisely, $a^{(r,i)}_{k}=1$ if and only if $S_{k,i}\left(v^{(r)}\right)\neq\emptyset$; if and only if $v^{(r)}_{[k+1,D-(i-k)]}=v^{(r)}_{[i+1,D]}$ and $v^{(r)}_{[k+1,D-(j-k)]}\ne v^{(r)}_{[j+1,D]}$ for all $j$, $k<j< i$.
\end{theorem}

 \begin{example}
 \label{Pin K(d,10)}
 Let $G=K(d,10)$ and let $\mathcal{V}_r$ be the equivalence class of vertices which sequence representation is $v=\alpha\beta\gamma\alpha\beta\gamma\alpha\beta\alpha\beta$, where $\alpha$, $\beta$ and $\gamma$ are distinct symbols. As seen in Examples~\ref{example K(d,10)} and \ref{both}, if $v\in\mathcal{V}_r$, then $|S_8^\star(v)|=d^8-d^6-d^3-1$. Since the number of vertices of $K(d,10)$ is $|V|=d^{10}+d^9$, the probability that a  vertex selected uniformly at random  from $V\setminus\{v\}$ is at distance $8$ from $v$ is
 $$
 \P_{\mathsf{in}}(8\mid v\in\mathcal{V}_r)=\frac{|S_{8}^\star(v_r)|}{|V|-1}=\frac{d^8-d^6-d^3-1}{d^{10}+d^9-1}.
 $$
 \end{example}

 \begin{example}
 \label{ex-K(d,4)}
 As a second example let us calculate all the input probabilities $\P_{\mathsf{in}}(i)$, $1\le i\le 4$, in the Kautz digraph $K(d,4)$ of degree $d$ and diameter $D=4$. In $K(d,4)$ the numbers $n_s$ of equivalence classes for which the the sequence representation of its vertices contains $s$ distinct symbols is $n_2=1$, $n_3=2^{D-2}-1=3$, and $n_4=(3^{D-2}-2^{D-1}+1)/2=1$. Hence (as mentioned in Example~\ref{example K(d,4)}) we have $l=5$ different vertex classes, namely
 \begin{equation*}
 \label{cl-K(d,4)}
 \begin{array}{l} 
 \mathcal{V}_1=V_{2,1}=\{v\in V:\ v=\alpha\beta\alpha\beta\},\\[1mm] 
 \mathcal{V}_2=V_{3,1}=\{v\in V:\ v=\alpha\beta\alpha\gamma\}, \\[1mm] 
 \mathcal{V}_3=V_{3,2}=\{v\in V:\ v=\alpha\beta\gamma\alpha\},\\ [1mm] 
 \mathcal{V}_4=V_{3,3}=\{v\in V:\ v=\alpha\beta\gamma\beta\},\\[1mm]  
 \mathcal{V}_5=V_{4,1}=\{v\in V:\ v=\alpha\beta\gamma\delta\},
 \end{array}
 \end{equation*}
 with respective cardinalities $|\mathcal{V}_1|=(d+1)d$, $|\mathcal{V}_2|=|\mathcal{V}_3|=|\mathcal{V}_4|=(d+1)d(d-1)$, and $|\mathcal{V}_5|=(d+1)d(d-1)(d-2)$. Using the values $a_k^{(r,i)}$ detailed in Table~\ref{table 1} and applying Theorem~\ref{p_in} we obtain the input probabilities summarized in Table~\ref{table 2}.

 \begin{table}[ht]
 \caption{Input probabilities in $K(d,4)$}
 \label{table 2}
 \renewcommand{\arraystretch}{1.5}
 \begin{center}
 \begin{tabular}{|c|c|}
 \hline
 $i$&$\P_{\mathsf{in}}(i)$\\[2.5mm] \hline
 $1$&$\displaystyle\frac{d}{d^4+d^3-1}$\\[3mm]  \hline
 $2$&$\displaystyle\frac{d^4-1}{d^6+d^5-d^2}$ \\[3mm]  \hline
 $3$&$\displaystyle\frac{d^5-d^2-d+1}{d^6+d^5-d^2}$ \\[3mm]  \hline
 $4$&$\displaystyle\frac{d^5-d^3-d^2+1}{d^5+d^4-d}$ \\[3mm]  \hline
 \end{tabular}
 \end{center}
 \end{table}
\end{example}

The transition probability $\P_{\mathsf{t}}(i,j)$ can also be calculated as $\P_{\mathsf{t}}(i,j)=\sum_r \P_{\mathsf{t}}(i,j \mid v\in \mathcal{V}_r)\, \P(v\in \mathcal{V}_r)$. In this sum, $\P_{\mathsf{t}}(i,j \mid  v\in \mathcal{V}_r)$ is the conditional probability that the new distance to destination be $j$, given that a deflection occurs when  visiting a vertex $v$ at distance $i$ to the destination and belonging to the class $\mathcal{V}_r$; whereas $\P(v\in \mathcal{V}_r)$ is the probability that the vertex at which deflection occurs be in $\mathcal{V}_r$. Moreover, we will prove that
\begin{equation}
\label{dfl-1}
\P_{\mathsf{t}}(i,j\mid v\in \mathcal{V}_r)=\frac{1}{d-1}\cdot\frac{\left|S_i^\star\left(v\right)\cap S_{j}^\star\left(w_v\right)\right|}{\left|S_i^\star\left(v\right)\right|}\left(1-\frac{|S_i^\star(v)\cap S_{i-1}^\star(w_v)|}{|S_i^\star(v)|}\right),
\end{equation} 
where $w_v$ is a precise vertex, adjacent from $v$, uniquely determined by $v$ and the distances $i$ and $j$. In this way, we obtain the following result.

\begin{theorem} 
\label{p_ij}
The transition probabilities $\P_{\mathsf{t}}(i,j)$, $1\le i\le j< D$, are given by 
\begin{equation*}
\P_{\mathsf{t}}(i,j) = \frac{1}{(d-1)|V|}\sum_{r} |\mathcal{V}_r|\, p^{(r,i,j)}\left(1-q^{(r,i)}\right),
\end{equation*}
where $p^{(r,i,j)}$ and $q^{(r,i)}$ are rational fractions in the degree $d$ of the form
$$
p^{(r,i,j)} =k^{(r,i,j)}\cdot \frac{\, d^i-\alpha^{(r,i)}_{i-1}d^{i-1}-\cdots -\alpha^{(r,i)}_{1} d-\alpha^{(r,i)}_{0}  }{d^i-a^{(r,i)}_{i-1}d^{i-1}-\cdots -a^{(r,i)}_{1} d-a^{(r,i)}_{0}}
$$
and
$$
q^{(r,i)}  =\kappa^{(r,i)}\cdot \frac{d^{i-1}-b^{(r,i)}_{i-2}d^{i-2}-\ldots -b^{(r,i)}_{1} d-b^{(r,i)}_{0}}{d^i-a^{(r,i)}_{i-1}d^{i-1}-\cdots -a^{(r,i)}_{1} d-a^{(r,i)}_{0}},
$$
 and the coefficients of these rational fractions are $0$, $1$ or $2$. Namely, $k^{(r,i,j)}, \kappa^{(r,i)} \in\{0,1\}$; $\alpha^{(r,i)}_{i-1}\in\{0,1,2\}$;  $a^{(r,i)}_{i-1}\in\{0,1\}$; and $\alpha^{(r,i)}_{l}, a^{(r,i)}_{l}, b^{(r,i)}_{l} \in\{0,1\}$ for $ 0\le l\le i-2$.
\end{theorem}

\begin{remark}
In the proof of this theorem it will be shown how to determine the coefficients $k^{(r,i,j)}$, $\kappa^{(r,i)}$, $a^{(r,i)}_{k}$, $b^{(r,i)}_{k}$, and $\alpha^{(r,i)}_{k}$. More precisely, we will show that if $v^{(r)}$ is any representative vertex in the class $\mathcal{V}_r$, and if $w^{(r)}$ is the vertex adjacent from $v^{(r)}$ given by $w^{(r)}=v^{(r)}_2\cdots v^{(r)}_{D}v^{(r)}_{i+(D-j)}$, then 
\begin{itemize}
\item[(a)] $k^{(r,i,j)}=1$ if and only if $S_i^\star\left(v^{(r)}\right)\cap S_{j}^\star\left(w^{(r)}\right)\ne\emptyset$, as determined by statement (2) of Propositions~\ref{v2-second} and \ref{v2-second-cas 2};
\item[(b)] $\kappa^{(r,i)}=1$ if and only if $S_i^\star\left(v^{(r)}\right)\cap S_{i-1}^\star\left(w^{(r)}\right)\ne\emptyset$, as determined by statement (3) of Propositions~\ref{v2-second} and \ref{v2-second-cas 2};
\item[(c)] the coefficients $a^{(r,i)}_{k}\in\{0,1\}$ are determined from $v^{(r)}$ as in Proposition~\ref{propo 2.3};
\item[(d)] the coefficients  $\alpha^{(r,i)}_{k}$ and  $b^{(r,i)}_{k}$ are determined from $v^{(r)}$ and $w^{(r)}$ as in Theorems~\ref{propo 2.9-v2-cas 1} and \ref{propo 2.9-v2-cas 2}.
\end{itemize}
We stress that the values of all these coefficients are independent of the choice of $v^{(r)}$  in the class $\mathcal{V}_r$.
\end{remark}

The previous theorem provides a description of the probabilities $\P_{\mathsf{t}}(i,j)$ in the case $j<D$. Next we discuss the case $j=D$. Clearly we have $\P_{\mathsf{t}}(D,D)=1$, because $D$ is the maximum possible distance between the vertices of the digraph. Moreover, the transition probabilities $\P_{\mathsf{t}}(i,D)$, $1\le i< D$, can be obtained from Theorem~\ref{p_ij} because, for each $i$, we have $\P_{\mathsf{t}}(i,D)=1-\sum_{j=i}^{D-1}\P_{\mathsf{t}}(i,j)$. However, for the sake of completeness, we present in the following theorem a description of the transition probabilities $\P_{\mathsf{t}}(i,D)$, analogous to those provided in Theorem~\ref{p_ij} for $\P_{\mathsf{t}}(i,j)$.
\begin{theorem} 
\label{p_iD}
The transition probabilities $\P_{\mathsf{t}}(i,D)$, $1\le i < D$, are given by 
\begin{equation*}
\P_{\mathsf{t}}(i,D) = \frac{1}{(d-1) |V|}\sum_{r}\sum_{s=1}^{m_r} |\mathcal{V}_r|\, p^{(r,s,i)}\left(1-q^{(r,s,i)}\right),
\end{equation*}
where $m_r\in\{d-1,d\}$ if $G=K(d,D)$ and $m_r=d$ if $G=B(d,D)$, and where $p^{(r,s,i)}$ and $q^{(r,s,i)}$ are rational fractions in the degree $d$ of the form
$$
p^{(r,s,i)}=k^{(r,s,i)}\cdot \frac{\, d^i-\alpha^{(r,s,i)}_{i-1}d^{i-1}-\cdots -\alpha^{(r,s,i)}_{1} d-\alpha^{(r,s,i)}_{0}  }{d^i-a^{(r,s,i)}_{i-1}d^{i-1}-\cdots -a^{(r,s,i)}_{1} d-a^{(r,s,i)}_{0}}
$$
and
$$
q^{(r,s,i)}=\kappa^{(r,s,i)}\cdot \frac{d^{i-1}-b^{(r,s,i)}_{i-2}d^{i-2}-\ldots -b^{(r,s,i)}_{1} d-b^{(r,s,i)}_{0}}{d^i-a^{(r,s,i)}_{i-1}d^{i-1}-\cdots -a^{(r,s,i)}_{1} d-a^{(r,s,i)}_{0}},
$$
 where all the coefficients are $0$, $1$ or $2$. Namely, $k^{(r,s,i)}$, $\kappa^{(r,s,i)} \in\{0,1\}$; $\alpha^{(r,s,i)}_{i-1}\in\{0,1,2\}$;  $a^{(r,s,i)}_{i-1}\in\{0,1\}$; and $\alpha^{(r,s,i)}_{l}, a^{(r,s,i)}_{l}, b^{(r,s,i)}_{l} \in\{0,1\}$ for $ 0\le l\le i-2$. 
\end{theorem}

 The next two examples illustrate the rational fractions $p^{(r,i,j)}$ and $q^{(r,i)}$, as well as the expressions of the transition probabilities $\P_{\mathsf{t}}(i,j)$ and $\P_{\mathsf{t}}(i,D)$, formulated in Theorems~\ref{p_ij} and \ref{p_iD}.

 \begin{example}
 \label{def-Kautz}
 Let $G=K(d,12)$ and consider the class of vertices $\mathcal{V}_r$ which sequence representation is $\alpha\beta\gamma\alpha\beta\gamma\alpha\beta\gamma\alpha\beta\gamma$, where $\alpha$, $\beta$ and $\gamma$ stand for  different symbols of the alphabet $A$. Suppose that $v\in\mathcal{V}_r$ is the vertex at which the deflection occurs and let $w=\beta\gamma\alpha\beta\gamma\alpha\beta\gamma\alpha\beta\gamma w_{12}$ be the vertex adjacent from $v$ through which this deflection takes place. Let us calculate, for instance, the transition probabilities $\P_{\mathsf{t}}(4,6\mid v\in\mathcal{V}_r)$ and $\P_{\mathsf{t}}(1,6\mid v\in\mathcal{V}_r)$. 

 Firstly, let us consider  $\P_{\mathsf{t}}(4,6\mid v\in\mathcal{V}_r)$. Observe that if this probability is not zero, then the destination vertex $z$ must belong to $S^\star_4(v)$ and also to $S^\star_6(w)$. Therefore, since $S^\star_4(v)\subseteq S_4(v)$ and $S^\star_6(v)\subseteq S_6(v)$, we conclude that $S_4(v)\cap S_6(w)\ne \emptyset$ is a necessary condition for $\P_{\mathsf{t}}(4,6\mid v\in\mathcal{V}_r)\ne 0$. Using the notation of Examples~\ref{example B(d,7)} and \ref{example K(d,10)}, we have $S_4(v)=\{u\in V:\, u=\beta\gamma\alpha\beta\gamma\alpha\beta\gamma\ast\ast\ast\ast\}$ and $S_6(w)=\{u\in V:\, u=\beta\gamma\alpha\beta\gamma w_{12} \ast\ast\ast\ast\ast\ast\}$.  From these sequence representations we can check that $S_4(v)\cap S_6(w)\ne \emptyset$ if and only if $S_4(v)\subseteq S_6(w)$; if and only if $w_{12}=\alpha$. We conclude that if $\P_{\mathsf{t}}(4,6\mid v\in\mathcal{V}_r)\ne 0$, then there is only one precise vertex $w$ adjacent from $v$ such that $d(w,z)=6$, namely $w=\beta\gamma\alpha\beta\gamma\alpha\beta\gamma\alpha\beta\gamma\alpha$.

 Now let us calculate $|S^\ast_4(v)|$, $|S^\ast_4(v)\cap S^\ast_6(w)|$, and $|S_4^\star(v)\cap S_{3}^\star(w)|$ in order to obtain $\P_{\mathsf{t}}(4,6\mid v\in\mathcal{V}_r)$ by applying (\ref{dfl-1}). We can verify from the sequence representations of $v$ and $w$ that the following chain of inclusions hold:
 \begin{equation}
 \label{chain}
 S_0(w)\subseteq S_1(v)\subseteq S_3(w)\subseteq S_4(v)\subseteq S_6(w).
 \end{equation}
 Furthermore, we can check that $S_k(v)\cap S_4(v)=\emptyset$ if $k=0,2,3$, and that $S_k(w)\cap S_6(w)=\emptyset$ if $k=1,2,4,5$. Then $S^\ast_4(v)=S_4(v)\setminus S_1(v)$, $S^\star_3(w)=S_3(w)\setminus S_0(w)$, and $S_6^\star(w)=S_6(w)\setminus S_3(w)$. Moreover, we have
 \begin{align*}
 S_4^\star(v)\cap S_6^\star(w) &=(S_4(v)\setminus S_1(v))\cap (S_6(w)\setminus S_3(w))\\
 &=(S_4(v)\cap S_6(w))\setminus (S_1(v)\cup S_3(w))=S_4(v)\setminus S_3(w)),
 \end{align*}
 and
 \begin{align*}
 S^\ast_4(v)\cap S^\ast_3(w) &=\left(S_4(v)\setminus S_1(v)\right)\cap \left(S_3(w)\setminus S_0(w)\right)\\
 &=(S_4(v)\cap S_3(w))\setminus (S_1(v)\cup S_0(w))=S_3(w)\setminus S_1(v).
 \end{align*}
 We conclude that $|S^\ast_4(v)|=d^4-d$, $|S^\ast_4(v)\cap S^\ast_6(w)|=d^4-d^3$, and $|S^\ast_4(v)\cap S^\ast_3(w)|=d^3-d$. Therefore we get from (\ref{dfl-1}) that the value of the transition probability $\P_{\mathsf{t}}(4,6\mid v\in\mathcal{V}_r)$ is
 \begin{align}
 \label{p_46}
 \nonumber
 \P_{\mathsf{t}}(4,6\mid v\in\mathcal{V}_r)&=\frac{1}{d-1}\cdot\frac{|S^\ast_4(v)\cap S^\ast_6(w)|}{|S^\ast_4(v)|}\left(1-\frac{|S_4^\star(v)\cap S_{3}^\star(w)|}{|S_4^\star(v)|}\right)\\[2mm]
 &=\frac{1}{d-1}\cdot\frac{d^4-d^3}{d^4-d}\left(1-\frac{d^3-d}{d^4-d}\right)=\frac{d^4}{d^5+d^4+d^3-d^2-d-1}.
 \end{align}
 Observe that in the formula (\ref{p_46}) the rational expressions of $|S^\ast_4(v)\cap S^\ast_6(w)|/|S^\ast_4(v)|$ and $|S_4^\star(v)\cap S_{3}^\star(w)|/|S^\ast_4(v)|$ correspond in Theorem~\ref{p_ij} to $p^{(r,4,6)}=(d^4-d^3)/(d^4-d)$ and $q^{(r,4)}=(d^3-d)/(d^4-d)$, respectively.

 Secondly, let us determine $\P_{\mathsf{t}}(1,6\mid v\in\mathcal{V}_r)$. Reasoning as before we deduce that if $\P_{\mathsf{t}}(1,6\mid v\in\mathcal{V}_r)\ne 0$, then we must have $S_1(v)\cap  S_6(w)\ne  \emptyset$, being $w$ the vertex adjacent from $v$ through which the deflection takes place. We can check that this necessary condition  holds if and only if $w_{12}=\alpha$; that is, $w$ must be again the vertex $w=\beta\gamma\alpha\beta\gamma\alpha\beta\gamma\alpha\beta\gamma\alpha$. We have $S^\ast_1(v)=S_1(v)$, because $S_0(v)\not\subseteq S_1(v)$; and,  as in the previous case, $S_6^\star(w)=S_6(w)\setminus S_3(w)$. But now, since $S_1(v)\subseteq S_3(w)$ (see (\ref{chain})), we deduce that
 $$
 S_1^\star(v)\cap S_6^\star(w)= (S_1(v)\cap S_6(w))\setminus S_3(w)=\emptyset.
 $$ 
 Therefore, although the necessary condition $S_1(v)\cap  S_6(w)\ne  \emptyset$ for having $\P_{\mathsf{t}}(1,6\mid v)\ne 0$ holds, we have in this case $\P_{\mathsf{t}}(1,6\mid v)=0$, because $S^\ast_1(v)\cap  S^\ast_6(w)=\emptyset$. This fact is captured in Theorem~\ref{p_ij} by setting $k^{(r,1,6)}=0$.
 \end{example}

 \begin{example}
 Let us consider the Kautz digraph $K(d,4)$, which vertex classes $\mathcal{V}_r$ are detailed  in Example~\ref{ex-K(d,4)}, and let us determine all the transition probabilities $\P_{\mathsf{t}}(i,j)$, $1\le i\le j\le 4$. 

 As in the previous example, let  $w$ be the vertex adjacent from $v$ through which the deflection takes place. If $\P_{\mathsf{t}}(i,j)\ne 0$, then $v$ and $w$ must clearly satisfy the necessary condition $S_i(v)\cap S_j(w)\ne\emptyset$. As we will prove in Lemma~\ref{tech-2}, for $i<D=4$ the intersection $S_i(v)\cap S_j(w)$ is nonempty if and only if $S_i(v)\subseteq S_j(w)$. Hence if $i<4$ and  $\P_{\mathsf{t}}(i,j)\ne 0$, then we must have 
 \begin{equation}
 \label{eq:nec}
 S_i(v)\subseteq S_j(w)\subseteq S_{j+1}(v),
 \end{equation}
 where the last inclusion holds because $w\in S_1(v)$. Besides, since in a Kautz digraph two successive symbols in the sequence representation of the vertices are different, we have $S_i(v)\not\subseteq S_i(w)$. Therefore, if $i<4$, then $S_i(v)\cap S_i(w)=\emptyset$,  and so $\P_{\mathsf{t}}(1,1)=\P_{\mathsf{t}}(2,2)=\P_{\mathsf{t}}(3,3)=0$. Furthermore, since the diameter of $K(d,4)$ is $4$, the transition probabilities $\P_{\mathsf{t}}(3,4)$ and $\P_{\mathsf{t}}(4,4)$ must be equal to $1$. So we only need to compute the probabilities $\P_{\mathsf{t}}(1,2)$, $\P_{\mathsf{t}}(1,3)$, $\P_{\mathsf{t}}(1,4)$, $\P_{\mathsf{t}}(2,3)$, and $\P_{\mathsf{t}}(2,4)$.

 First let us consider $\P_{\mathsf{t}}(1,2)$. According to (\ref{eq:nec}), the possible vertices $v$ at which deflection to distance 2 is possible must satisfy $S_1(v)\subseteq S_3(v)$. Thus either $v=\alpha\beta\alpha\beta\in\mathcal{V}_1$ or $v=\alpha\beta\gamma\beta\in \mathcal{V}_4$. Moreover, also by the necessary condition (\ref{eq:nec}), a vertex $w\in S_1(v)$ through which deflection from distance $1$ to distance $2$ is possible has to be such that $S_1(v)\subseteq S_2(w)$. It is easily checked that if $v\in \mathcal{V}_1$, then there is only one possible $w$, namely $w=\beta\alpha\beta\alpha$; whereas if $v\in \mathcal{V}_4$, then  $w=\beta\gamma\beta\gamma$. We can verify that, in both cases, $|S^\star_1(v)|=d$, $|S^\star_1(v)\cap S^\star_2(w)|=d-1$ and $|S^\star_1(v)\cap S^\star_0(w)|=1$. Therefore, by applying (\ref{dfl-1}), we obtain
 $$
 \P_{\mathsf{t}}(1,2\mid v\in \mathcal{V}_1)=\P_{\mathsf{t}}(1,2\mid v\in \mathcal{V}_4)=\frac{1}{d-1}\cdot\frac{d-1}{d}\left(1-\frac{1}{d}\right)=\frac{d-1}{d^2},
 $$ 
 and hence
 \begin{align*}
 \P_{\mathsf{t}}(1,2)&=\sum_{r=1,4} \P_{\mathsf{t}}(1,2 \mid v\in \mathcal{V}_r)\, \P(v\in \mathcal{V}_r)=\frac{d-1}{d^2}\left(\frac{|\mathcal{V}_1|}{|V|}+\frac{|\mathcal{V}_4|}{|V|}\right)\\[2mm]
 &=\frac{d-1}{d^2}\left(\frac{(d+1)d}{(d+1)d^3}+\frac{(d+1)d(d-1)}{(d+1)d^3}\right)=\frac{d-1}{d^3},
 \end{align*}
 where the probabilities $\P(v\in \mathcal{V}_r)$ have been calculated as $|\mathcal{V}_r|/|V|$. Observe that this expression of $\P_{\mathsf{t}}(1,2)$ corresponds to the one given in Theorem~\ref{p_ij} with nonnull rational expressions $p^{(1,1,2)}=p^{(4,1,2)}=(d-1)/d$ and $q^{(1,1)}=q^{(4,1)}=1/d$.

 Next let us calculate $\P_{\mathsf{t}}(1,3)$. By condition (\ref{eq:nec}) we have to consider those vertex classes $\mathcal{V}_r$ such that if $v\in\mathcal{V}_r$, then $S_1(v)\subseteq S_4(v)$. We can check that this inclusion holds if $v$ satisfies $v_2\ne v_4$; that is, if $v$ belongs to $\mathcal{V}_2$, $\mathcal{V}_3$, or $\mathcal{V}_5$. Moreover, we can verify  that, again, for each $v$ there is only one possible choice of $w$, and that, in any case, we have $|S^\star_1(v)|=d$, $|S^\star_1(v)\cap S^\star_3(w)|=d-1$ and $|S^\star_1(v)\cap S^\star_0(w)|=1$. Hence we conclude from (\ref{dfl-1}) that $\P_{\mathsf{t}}(1,3\mid v\in \mathcal{V}_2)=\P_{\mathsf{t}}(1,3\mid v\in \mathcal{V}_3)=\P_{\mathsf{t}}(1,3\mid v\in \mathcal{V}_5)={(d-1)}/{d^2}$. Therefore
 \begin{align*}
 \P_{\mathsf{t}}(1,3)&=\sum_{r=2,3,5} \P_{\mathsf{t}}(1,3 \mid v\in \mathcal{V}_r)\, \P(v\in \mathcal{V}_r)\\
 &=\frac{d-1}{d^2}\left(2 \frac{(d+1)d(d-1)}{(d+1)d^3}+
 \frac{(d+1)d(d-1)(d-2)}{(d+1)d^3}\right)=\frac{d^2-2d+1}{d^3}.
 \end{align*}
 This expression of $\P_{\mathsf{t}}(1,3)$ corresponds to the one given in Theorem~\ref{p_ij} by taking $p^{(2,1,3)}=p^{(3,1,3)}=p^{(5,1,3)}=(d-1)/d$ and $q^{(2,1)}=q^{(3,1)}=q^{(5,1)}=1/d$ as nonnull rational expressions.

 Once we know $\P_{\mathsf{t}}(1,2)$ and $\P_{\mathsf{t}}(1,3)$  we can obtain the remaining transition probability $\P_{\mathsf{t}}(1,4)$ as
 $$
 \P_{\mathsf{t}}(1,4)=1-\P_{\mathsf{t}}(1,2)-\P_{\mathsf{t}}(1,3)=\frac{d^2-d+1}{d^2}.
 $$

 Now let us calculate $\P_{\mathsf{t}}(2,3)$. So, by (\ref{eq:nec}), we have to consider those vertices $v$ verifying $S_2(v)\subseteq S_4(v)$; that is, those vertices $v$ satisfying $v_3\ne v_4$. But this necessary condition holds for all $v\in V$, because two consecutive symbols in the sequence representation of $v$ are different. Once more, for each $v$ there is only one possible choice of $w\in S_1(v)$ satisfying the necessary condition (\ref{eq:nec}). Moreover, we can check that if $v\in \mathcal{V}_1$, then $|S^\star_2(v)|=d^2-1$, $|S^\star_2(v)\cap S^\star_3(w)|=d^2-d$, and $|S^\star_2(v)\cap S^\star_1(w)|=d-1$; whereas if $v\in \mathcal{V}_s$, $2\le s \le 5$, then  $|S^\star_2(v)|=d^2$, $|S^\star_2(v)\cap S^\star_3(w)|=d^2-d$, and $|S^\star_2(v)\cap S^\star_1(w)|=d$. Therefore, $\P_{\mathsf{t}}(2,3\mid v\in \mathcal{V}_1)={d^2}/{((d^2-1)(d+1))}$ and $\P_{\mathsf{t}}(2,3\mid v\in \mathcal{V}_s)={(d-1)}/{d^2}$ for $2 \le s\le 5$. So we have
 $$
 \P_{\mathsf{t}}(2,3)=\sum_{1\le r\le 5} \P_{\mathsf{t}}(1,3 \mid v\in \mathcal{V}_r)\, \P(v\in \mathcal{V}_r)=\frac{d^6-2d^4+3d^2-1}{d^7+d^6-d^5-d^4}.
 $$
 In this case, the nonnull rational expressions in Theorem~\ref{p_ij} to formulate $\P_{\mathsf{t}}(2,3)$ are $p^{(1,2,3)}=(d^2-d)/(d^2-1)$, $p^{(2,2,3)}=p^{(3,2,3)}=p^{(4,2,3)}=p^{(5,2,3)}=(d^2-d)/d^2$, $q^{(1,2)}=(d-1)=(d^2-1)$, and $q^{(2,2)}=q^{(3,2)}=q^{(4,2)}=p^{(5,2)}=1/d$.

 Finally, 
 $$
 \P_{\mathsf{t}}(2,4)=1-\P_{\mathsf{t}}(2,3)=\frac{d^7-d^5+d^4-3d^2+1}{d^7+d^6-d^5-d^4}.
 $$
 \end{example}

Using the Markov model  \cite{fm} mentioned in Section~1, we can apply the probabilities given in Theorems~\ref{p_in}, \ref{p_ij}, and  \ref{p_iD} to measure the efficiency of deflection routing in De Bruijn and Kautz networks.

We conclude this subsection with two corollaries that are straightforward consequences of our previous results. The first one 
deals with the asymptotic behaviour of the input and transition probabilities. The second one is about the computation of the mean distance in the De Bruijn and Kautz digraphs (some related results can be found in \cite{bls,PaSe,SaTv}).

\begin{corollary}  
	\mbox{ }
	\begin{enumerate}
	  \item $\P_{\mathsf{in}}(i)\sim 1/d^{D-i}$ as $d\to\infty$.
		\item If  $j<D$ and $d$ is large enough, then $\P_{\mathsf{t}}(i,j)\le 1/d$.
		\item If $j=D$, then $\P_{\mathsf{t}}(i,j)\sim 1$ as $d\to\infty$.
	\end{enumerate}
\end{corollary}

\begin{corollary}  If $G$ is the De Bruijn digraph $B(d,D)$ or the Kautz digraph $K(d,D)$, then the mean distance of $G$ is given by $\sum_{i=1}^D i\cdot \P_{\mathsf{in}}(i)$, and hence it can be expressed as a rational fraction in the degree $d$.

\end{corollary}

\setcounter{subsection}{0}
\section{Proofs of our results on the polynomial description of the distance-layer structure}
\label{proofs}
In this section we prove the results presented in Subsection~\ref{sub: layer}. To do this we will use several technical lemmas gathered in Subsection~\ref{technical-lemmas}. We remind that $V$ denotes the vertex set of the digraph, either $G=B(d,D)$ or $G=K(d,D)$, and that, for $v=v_1v_2\cdots v_D\in V$, we denote the subsequence $v_iv_{i+1}\cdots v_j$ by $v_{[i,j]}$.

\subsection{Technical lemmas}
\label{technical-lemmas}

\begin{lemma} 
\label{tech-1}
Let $v\in V$. Then $|S_i(v)|=d^i$ for  $0\le i\le D$. Moreover, if $G=B(d,D)$ and  $i\ge D$, then   $S_i(v)=V$; while  if $G=K(d,D)$ and  $i\ge D+1$, then   $S_i(v)=V$.
\end{lemma}

\begin{proof}
The result follows directly from the sequence representation of the vertices and the adjacency rules. Moreover, we must consider that in $K(d,D)$ there exists a walk of length $D+1$ from a given vertex $v$ to any other one.
\end{proof}

The main part of the next lemma essentially states that, given any two (no necessarily different) vertices $v,v'$, if $k\le i<D$ or $k<i=D$, then either $S_k(v)\subseteq S_i(v')$ or $S_k(v)\cap S_i(v')=\emptyset$. The precise formulation in terms of the sequence representation of $v$ and $v'$ is as follows.

\begin{lemma} 
\label{tech-2}
Let $v,v'\in V$ be two vertices  with sequence representation $v=v_1v_2\cdots v_D$ and $v'=v'_1v'_2\cdots v'_D$. Let $0\le k\leq i$ and assume that $S_i(v')\neq V$. Then $i\le D$, $S_k(v)\neq V$, and  the following statements hold:
\begin{enumerate}
\item If $ k\le i<D$, then either $S_k(v)\subseteq S_i(v')$ or $S_k(v)\cap S_i(v')=\emptyset$. Moreover, $S_k(v)\subseteq S_i(v')$ if and only if $v_{[k+1,D-(i-k)]}=v'_{[i+1,D]}$.
\item If $k<i= D$, then $G=K(d,D)$ and either $S_k(v)\subseteq S_D(v')$ or $S_k(v)\cap S_D(v')=\emptyset$. Moreover, $S_k(v)\subseteq S_D(v')$ if and only if $v_{k+1}\neq v'_{D}$.
\item 
If $k=i= D$, then $G=K(d,D)$ and $S_D(v)\cap S_D(v')\ne\emptyset$. Moreover, if $v_D=v'_D$ then $S_D(v)=S_D(v')$, whereas if $v_D\ne v'_D$, then $S_D(v)\ne S_D(v')$ and $|S_D(v)\cap S_D(v')|=d^D-d^{D-1}$.
\end{enumerate}
\end{lemma}

\begin{proof}
It follows from Lemma~\ref{tech-1} that if $S_i(v')\neq V$, then $i\le D-1$ if $G=B(d,D)$,  $i\le D$ if $G=K(d,D)$, and $S_k(v)\ne V$ for $k\le i$.

Suppose $i<D$. The sequences corresponding to vertices $w\in S_k(v)$ and $w'\in S_i(v')$ are of the form $w=v_{k+1}\cdots v_D\ast\cdots\ast$ and  $w'=v'_{i+1}\cdots v'_D\ast\cdots\ast$, respectively. Since $k\le i$, we get  from these sequence representations that $ S_k(v)\cap S_i(v')\ne\emptyset$ if and only if 
\begin{equation}
\label{contain}
v_{k+1}=v'_{i+1},\ldots, v_{k+(D-i)}=v'_{D},
\end{equation}
that is, the subsequences $v_{[k+1,D-(i-k)]}$ and $v'_{[i+1,D]}$ coincide. Furthermore, $S_k(v)\subseteq S_i(v')$ if and only if condition (\ref{contain}) holds. (Notice that for $k=i$ we have $S_k(v)= S_i(v')$.) Hence if $k\le i< D$, then either $S_k(v)\subseteq S_i(v')$ or $S_k(v)\cap S_i(v')=\emptyset$ and statement (1) is proved.

Now assume that $k<i=D$ and $G=K(d,D)$. Then $S_k(v)=\{w\in V: w=v_{k+1}\cdots v_D \ast\cdots\ast\}$ and $S_D(v')=\{w\in V: w=w_1w_2\cdots w_D,\ w_1\ne v'_D\}$. Therefore, if  $v_{k+1}\ne v'_D$ then $S_k(v)\subseteq S_D(v')$, whereas if $v_{k+1}=v'_D$ then $S_k(v)\cap S_D(v')=\emptyset$. 

Finally assume that $k=i=D$ and $G=K(d,D)$. Then $S_D(v)=\{w\in V: w=w_1w_2\cdots w_D,\ w_1\ne v_D\}$ and $S_D(v')=\{w\in V: w=w_1w_2\cdots w_D,\ w_1\ne v'_D\}$. Hence $S_D(v)\cap S_D(v')\ne\emptyset$ because the alphabet $A$ has $d+1$ symbols and $d\ge 2$. Moreover, $S_D(v)\subseteq S_{D}(v')$ if and only if $S_D(v)=S_{D}(v')$, if and only if $v_D=v'_D$. If $v_D\ne v'_D$, then $|S_D(v)\cap S_D(v')|=(d-1)d^{D-1}=d^D-d^{D-1}$, because $w_1\in A\setminus \{v_D,v'_D\}$.
\end{proof}

\begin{remark}
\label{rem_S_kj-bis}
Let $0\le k<j<i\le D$. By statement (1) of Lemma~\ref{tech-2},  $S_j(v)\not\subseteq S_i(v)$ if and only if $S_j(v)\cap S_i(v)=\emptyset$. Therefore, if $S_{k}(v)\subseteq S_{i}(v)$ and $S_j(v)\not\subseteq S_i(v)$, then $S_k(v)\not\subseteq S_j(v)$. This observation allows us to reformulate Definition~\ref{def_S_kj} in the following way:
 Let $v\in V$ and let $k,i$ be two integers such that $0\le i\le D$ and $0\leq k\leq i$. Then
$$
S_{k,i}(v)=\left\{\begin{array}{ll}
S_{k}(v), & \textrm{if } S_{k}(v)\subseteq S_{i}(v) \textrm{ and } S_{k}(v)\cap S_{j}(v)=\emptyset \textrm{ for all }j,\ k<j<i;\\[2mm]
\emptyset,&   \textrm{otherwise. }
\end{array}\right.
$$
\end{remark}

\begin{remark}
\label{rem_S_kj-seq}
By the previous remark and Lemma~\ref{tech-2} we have $S_{k,i}(v)=S_k(v)$ if and only if $v_{[k+1,D-(i-k)]}=v_{[i+1,D]}$ and $v_{[k+1,D-(j-k)]}\ne v_{[j+1,D]}$ for all $j$, $k<j< i$. In particular, if $k=i-1$, then $S_{i-1,i}(v)=S_{i-1}(v)$ if and only if $S_{i-1}(v)\subseteq S_i(v)$; if and only if $v_{[i,D-1]}=v_{[i+1,D]}$; if and only if $v_i=v_{i+1}=\cdots=v_D$.
\end{remark}

\begin{remark}
\label{void}
Let $0\le k_1,k_2<i\le D$. We claim that if $k_1\ne k_2$, then $S_{k_1,i}(v) \cap S_{k_2,i}(v)=\emptyset$. Indeed,  if $S_{k_1,i}(v) \cap S_{k_2,i}(v)\ne\emptyset$, then $S_{k_1,i}(v)\ne\emptyset$ and $S_{k_2,i}(v)\ne\emptyset$. So, $S_{k_1,i}(v)=S_{k_1}(v)\subseteq S_i(v)$ and $S_{k_2,i}(v)=S_{k_2}(v)\subseteq S_i(v)$. Hence $S_{k_1}(v) \cap S_{k_2}(v)=S_{k_1,i}(v) \cap S_{k_2,i}(v)\ne\emptyset$. Now, by applying Lemma~\ref{tech-2}, either $S_{k_1}(v) \subseteq S_{k_2}(v)$ or $S_{k_2}(v) \subseteq S_{k_1}(v)$. Hence  either $S_{k_1}(v) \subseteq S_{k_2}(v)\subseteq S_i(v)$ or $S_{k_2}(v) \subseteq S_{k_1}(v)\subseteq S_i(v)$. In any case, this leads us to a contradiction with the definition of $S_{k,i}(v)$. This completes the proof of our claim.
\end{remark}


The following result describes the structure of the distance-layer set $S_{i}^\star(v)$.

\begin{lemma}
\label{tech-4}
Let $v\in V$ and let $0\leq i\le D$. Then $\displaystyle S_{i}^\star(v)=S_i(v)\setminus \Big (\bigcup_{k=0}^{i-1} S_{k,i}(v)  \Big ).$
\end{lemma}

\begin{proof}
From the definitions it is clear that 
$$
S_{i}^\star(v)=S_i(v)\setminus \Big (\bigcup_{k=0}^{i-1}S_{k}(v) \Big )= S_i(v)\setminus \Big (\bigcup_{k=0}^{i-1} (S_{k}(v) \cap S_i(v)) \Big ).
$$ 
By Lemma~\ref{tech-2}, either $S_{k}(v) \cap S_i(v)=\emptyset$ or $S_{k}(v) \subseteq S_i(v)$. Therefore, $S_{i}^\star(v)=S_i(v)\setminus \Big (\bigcup_{k=0}^{i-1} S_{k,i}(v)  \Big )$. 
\end{proof}

In the following two lemmas, we consider some valuable properties of the sets $S_{k,j}(w)$ in the case that $w$ is a vertex adjacent from $v$.

\begin{lemma}
\label{tech-5}
Let $v\in V$ and $w\in S_1(v)$.   Let $0\le i,j \le D$ and let $k$, $0\le k\le j$. Assume that $S_{k,j}(w)\cap S_i(v)\neq \emptyset$. Then 
\begin{enumerate}
\item If $k=D$, then $S_{k,j}(w)=S_D(w)$. Moreover,
\begin{enumerate}
\item If $i=D$, then $S_{k,j}(w)= S_i(v)$  if  $G=B(d,D)$, while $S_{k,j}(w)\ne S_i(v)$  if  $G=K(d,D)$.
\item If $i<D$, then $S_i(v)\subseteq S_{k,j}(w)$.
\end{enumerate}

\item If $k\ne D$, then $S_i(v) \subseteq S_{k,j}(w)=S_{k}(w)$ if $k\ge i$, while $S_k(w)=S_{k,j}(w) \subseteq S_{i}(v)$ if $k<i$. Moreover,
\begin{enumerate}
\item If $k=i$, then $G=B(d,D)$,  $S_i(v)=S_i(w)$, and either $j=i$ or $j=i+1$.
\item If $k=i-1=j$, then $S_{k,j}(w)\cap S_i(v)=S_{i-1}(w)$.
\item If $k=i-1<j<D$,  then either $S_{i-1,i}(v)=\emptyset$ or $G=B(d,D)$ and $j=i$. Moreover,  if $S_{i-1,i}(v)\ne \emptyset$, then $S_{i-1,i}(v)=S_{i-1,i}(w)=S_{i-1}(v)=S_{i-1}(w)$.
\item If $k=i-1<j=D$,  then either $S_{i-1,i}(v)=\emptyset$ or $G=B(d,D)$. Moreover,  if $S_{i-1,i}(v)\ne \emptyset$ and $w_D=v_D$, then $i=j=D$ and $S_{D-1,D}(v)=S_{D-1,D}(w)=S_{D-1}(v)=S_{D-1}(w)$.
\item If $k<i-1$, then  there exists $k'\leq i-1$ such that $S_{k,j}(w)\subseteq S_{k',i}(v)$.
\end{enumerate}
\end{enumerate}
\end{lemma}

\begin{proof}
Let us prove statement (1). If $k=D$, then  $j=D$ and clearly $S_{k,j}(w)=S_D(w)$. If $G=B(d,D)$ we have $S_D(v)=S_D(w)=V$. If $G=K(d,D)$, then $S_D(v)\ne V$ and $S_D(w)\ne V$. Furthermore, since $w\in S_1(v)$, we have  $v_D=w_{D-1}\ne w_D$. Hence,  by applying Lemma~\ref{tech-2}, $S_D(v)\ne S_D(w)$.

Next we are going to prove statement (2). From now on assume that $k\ne D$. 

If $S_{k,j}(w)\cap S_i(v)\neq \emptyset$, then by the definition of $S_{k,j}(w)$ we get $S_{k,j}(w)=S_k(w)$, and hence, by applying Lemma~\ref{tech-2}, $S_k(w) \subseteq S_{i}(v)$ if $k<i$, while $S_i(v) \subseteq S_{k}(w)$ if $i \le k$, because $k<D$.

First let us consider statement (2.a). If $k=i$, then from our assumptions we get that $S_i(v)\cap S_i(w)\ne\emptyset$. Thus  the set $\Gamma^+_{i,i}(v)$ defined in Lemma~\ref{tech-3} is nonempty. Therefore, if $i<D$, then from Lemma~\ref{tech-3} we have $G=B(d,D)$ and $S_i(v)=S_i(w)$. To conclude the proof of statement (2.a) we must demonstrate that either $i=j$  or $i+1=j$. On one hand we have $k\le j$. On the other hand  we are assuming $k=i$. So, $i\le j$. Thus it only remains to prove that $j\le i+1$. Assume on the contrary that  $i+1<j$. Since $\Gamma^+_{i,i}(v)\ne\emptyset$, by applying again Lemma~\ref{tech-3} we have $S_{i-1}(w)\subseteq S_i(v)=S_i(w)\subseteq S_{i+1}(v)=S_{i+1}(w)\subseteq\cdots \subseteq S_j(v)=S_j(w)$. Therefore, $S_{i}(w)\subseteq S_{i+1}(w) \subseteq S_j(w)$ and thus, by Definition~\ref{def_S_kj}, we have $S_{i,j}(w)=\emptyset$. This contradicts the assumption $S_{k,j}(w)\cap S_i(v)\neq \emptyset$, because $k=i$.
 
Now let us prove (2.b). From Remark~\ref{rem_def_S_kj} we have $S_{i-1,i-1}(w)=S_{i-1}(w)$. Moreover, $S_{i-1}(w)\subseteq S_i(v)$ because $w\in S_1(v)$. So, statement (2.b) follows.

Next we demonstrate statements (2.c) and (2.d). Notice that $i\le j$ because $k=i-1$ and $k< j$. By the assumptions of the lemma we have  $w\in S_1(v)$ and $S_{i-1,j}(w)\cap S_i(v)\neq \emptyset$. Hence $S_{i-1,j}(w)=S_{i-1}(w)\subseteq S_{i}(v)$. Suppose that $S_{i-1,i}(v)\neq \emptyset$. Recall that this assumption implies $G=B(d,D)$ (see Remark~\ref{rem_def_S_kj}) and, moreover, from the definition of $S_{i-1,i}(v)$ it follows that $S_{i-1,i}(v)=S_{i-1}(v)\subseteq S_i(v)$. 

Let us consider first statement (2.c). So now we are assuming $j<D$. We have to prove that $j=i$ and that $S_{i-1,i}(v)=S_{i-1,i}(w)$ (because this last equality and the assumption $S_{i-1,i}(v)\neq \emptyset$ imply $S_{i-1,i}(w)\neq \emptyset$, and hence $S_{i-1,i}(v)=S_{i-1}(v)$ and $S_{i-1,i}(w)=S_{i-1}(w)$). Since $S_{i-1}(v)\subseteq S_i(v)$ and $i-1\le i<D$, we can apply  statement (1) of Lemma~\ref{tech-2} and we get that $v_{[i,D-1]}=v_{[i+1,D]}$. Moreover, $v_{[2,D]}=w_{[1,D-1]}$ because $w\in S_1(v)$. Hence on one hand we have $v_{i}=w_i=v_{i+1}= \dots = w_{D-1}=v_D$. On the other hand, from the definition of $S_{i-1,j}(w)$, in any case we have $S_{i-1,j}(w)\subseteq S_j(w)$. So, $S_{i-1,j}(w)\cap S_i(v)\subseteq S_i(v)\cap S_j(w)$. Therefore, $S_i(v)\cap S_j(w)\ne\emptyset$ because we are assuming $S_{i-1,j}(w)\cap S_i(v)\neq \emptyset$. Thus, since $i\le j<D$, we can apply once more Lemma~\ref{tech-2},  and now we get that $S_i(v) \subseteq S_j(w)$ and that $v_{[i+1,D-(j-i)]}=w_{[j+1,D]}$. In particular, $w_D=v_{D-(j-i)}=w_{D-(j-i)-1}$ and therefore
\begin{equation}
\label{equ-2}
v_i=w_{i}= v_{i+1}=\dots = w_{D-1}=v_D=w_D.
\end{equation} 
Observe that for $i\le l\le j< D$, equality (\ref{equ-2}) implies $w_{[i,D-l+i-1]}=w_{[l+1,D]}$ and $w_{[j+1,D-(j-l)]}=w_{[l+1,D]}$. Thus (again by Lemma~\ref{tech-2}) we have $S_{i-1}(w)\subseteq S_l(w)\subseteq S_j(w)$ for $i\le l\le j< D$. But if $j>i$ the definition of $S_{i-1,j}(w)$ would imply $S_{i-1,j}(w)=\emptyset$, a contradiction. Therefore, it must be $j=i$, as we wanted to show. To complete the proof of statement (2.c) in the case $j<D$, it only remains to show that $S_{i-1,i}(v)=S_{i-1,i}(w)$. But this is straightforward because by our assumptions we know that $S_{i-1}(w)=S_{i-1,j}(w)=S_{i-1,i}(w)$ and $S_{i-1}(v)=S_{i-1,i}(v)$, and, moreover, by (\ref{equ-2}) we have $v_{[i,D]}=w_{[i,D]}$, and so $S_{i-1}(v)=S_{i-1}(w)$. This completes the proof of (2.c).

Now we are going to prove (2.d). So, let us assume $j=D$. In this case we want to prove  that if $w_D=v_D$, then $i=D$  and   $S_{D-1}(v)=S_{D-1}(w)$ (as in the proof of (2.c) this last equality implies $S_{D-1,D}(v)=S_{D-1,D}(w)=S_{D-1}(v)=S_{D-1}(w)$). First let us show that $i=D$. On the contrary, assume that $i<D$. In that case, since $S_{i-1}(v)\subseteq S_i(v)$ and $i-1\le i<D$, we can apply again statement (1) of Lemma~\ref{tech-2} to get $v_{[i,D-1]}=v_{[i+1,D]}$ and thus $v_{[2,D]}=w_{[1,D-1]}$, because $w\in S_1(v)$. Thus, since $w_D=v_D$, equality (\ref{equ-2}) also holds. Therefore, we have $w_{[i,D-l+i-1]}=w_{[l+1,D]}$ and $w_{[j+1,D-(j-l)]}=w_{[l+1,D]}$ for $i\le l<D$. Thus (again by Lemma~\ref{tech-2}) we have $S_{i-1}(w)\subseteq S_l(w)\subseteq S_D(w)=V$ for $i\le l < D$. Since $i <D$ the definition of $S_{i-1,D}(w)$ implies $S_{i-1,D}(w)=\emptyset$, a contradiction. Therefore, it must be $i=D$, as we wanted to prove. It remains to show that $S_{D-1}(v)=S_{D-1}(w)$. But this is straightforward because $w_D=v_D$ and so $S_{D-1}(v)=S_{D-1}(w)$. 

Finally, let us prove statement (2.e). Let $k<i-1$ and assume that $S_{k,j}(w)\cap S_i(v)\neq \emptyset$. From  $S_{k,j}(w)\cap S_i(v)\neq \emptyset$ it follows that $S_{k,j}(w)=S_k(w)$ and $S_k(w)\cap S_i(v)\neq \emptyset$. Notice that  $S_k(w)\subseteq S_{k+1}(v)$ because $w\in S_1(v)$. Hence $S_{k+1}(v)\cap S_i(v)\neq \emptyset$ and so, by applying Lemma~\ref{tech-2}, it follows that $S_{k+1}(v)\subseteq S_i(v)$ because $k+1<i$. Let $k'=\max \{l \, : k+1\leq l <i \textrm{ and } S_{k+1}(v)\subseteq S_{l}(v)\subseteq S_i(v)\}$. From the definition we get $k'\leq i-1$ and   $S_{k+1}(v)\subseteq S_{k'}(v)= S_{k',i}(v)$. Therefore, $S_{k,j}(w)=S_k(w)\subseteq S_{k+1}(v)\subseteq S_{k',i}(v)$.
\end{proof}

\begin{lemma}
\label{tech-6}
Let $v\in V$ and $w\in S_1(v)$.  If $0\le k<i-1<D$ and $S_{k,i-1}(w)\neq \emptyset$, then there exists $k'\leq i-1$ such that $S_{k,i-1}(w)\subseteq S_{k',i}(v)$.
\end{lemma}

\begin{proof}
Let $k< i-1$ and assume that $S_{k,i-1}(w)\ne\emptyset$. It follows that $S_{k,i-1}(w)=S_k(w)\subseteq S_{i-1}(w)$. Hence $S_{k+1}(v)\cap S_{i-1}(w)\neq \emptyset$ because $S_k(w)\subseteq S_{k+1}(v)$. So, since $k+1\le i-1$, by applying Lemma~\ref{tech-2} we get $S_{k+1}(v)\subseteq S_{i-1}(w)$. Therefore, $S_{k+1}(v)\subseteq S_{i}(v)$ because $S_{i-1}(w)\subseteq S_{i}(v)$. Let $k'=\max \{l \, : k+1\leq l <i \textrm{ and } S_{k+1}(v)\subseteq S_{l}(v)\subseteq S_i(v)\}$. From the definition we get $k'\leq i-1$ and   $S_{k+1}(v)\subseteq S_{k'}(v)= S_{k',i}(v)$. Thus $S_{k,i-1}(w) \subseteq S_{k+1}(v)\subseteq S_{k',i}(v)$.
\end{proof}

In the following two lemmas, Lemma~\ref{tech-7} and Lemma~\ref{tech-7bis}, we provide a detailed description of the intersection sets $S_{i}^\star(v)\cap S_j^{\ast}(w)$ when $w$ is a vertex adjacent from $v$ and for $i,j\le D$ with $i-1\le j$. In fact,  Lemma~\ref{tech-7bis} is a refinement of Lemma~\ref{tech-7} in the case $j\ge i$ and  $i\ne D$.
\begin{lemma}
\label{tech-7}
Let $v\in V$ and $w\in S_1(v)$. Let $i,j\le D$ with $i-1\le j$. If $S_i(v)\cap S_j(w)\neq \emptyset$, then the intersection set $S_{i}^\star(v)\cap S_j^{\ast}(w)$ can be described as follows:
\begin{enumerate}
\item $\displaystyle S_{i-1}(w) \setminus\bigcup_{k=0}^{i-1}  S_{k,i}(v)$ if $j=i-1$.
\item $\displaystyle \Big(S_D(v)\cap S_D(w)\Big)\setminus \bigcup_{k=0}^{D-2} S_{k,D}(v)$ if $i=j=D$ and $G=K(d,D)$.
\item $\displaystyle V\setminus \Big( S_{D-1}(v)\cup S_{D-1}(w)\cup  \bigcup_{k=0}^{D-2} S_{k,D}(v)\Big )$ if $i=j=D$ and $G=B(d,D)$.
\item $\displaystyle S_{i}(v) \setminus \bigg(\Big(\bigcup_{k=0}^{i-1} S_{k,i}(v)\Big) \cup \Big(\bigcup_{k=i-1}^{j-1} S_{k,j}(w) \Big)\bigg)$ if  $j\ge i$ and $i\ne D$. 
\end{enumerate}
\end{lemma}

Before proving the lemma let us highlight the following two remarks.

\begin{remark}  Observe that in the above expressions we have $S_{k,i}(v)\subseteq S_i(v)$ for $0\le k\le i-1$, because of the definition of $S_{k,i}(v)$. Furthermore, by Lemma~\ref{tech-2}, either $S_{k,i}(v)\cap S_{i-1}(w)=\emptyset$ or $S_{k,i}(v)\subseteq S_{i-1}(w)$, $0\le k\le i-1$.
\end{remark}

 \begin{remark}In some cases the expressions of $S_{i}^\star(v) \cap S_j^{\ast}(w)$ given in Lemma~\ref{tech-7} can be further simplified or can be equal to $\emptyset$. For instance, as in Example~\ref{def-Kautz} let $G=K(d,12)$, $v=\alpha\beta\gamma\alpha\beta\gamma\alpha\beta\gamma\alpha\beta\gamma$,  $w=\beta\gamma\alpha\beta\gamma\alpha\beta\gamma\alpha\beta\gamma\alpha\in S_1(v)$, and let us consider the intersection sets $S_{4}^\star(v) \cap S_6^{\ast}(w)$ and $S_{1}^\star(v) \cap S_6^{\ast}(w)$.

 First let us consider $S_{4}^\star(v) \cap S_6^{\ast}(w)$. By Lemma~\ref{tech-2} we have $S_4(v)\cap S_6(w)=S_4(v)\ne\emptyset$ because $v_{[5,10]}=w_{[7,12]}$. Besides, we can verify that for $0\le k\le 3$ the only nonempty set $S_{k,4}(v)$ is $S_{1,4}(v)=S_1(v)$; and that for $3\le k\le 5$ the only nonempty set $S_{k,6}(w)$ is $S_{3,6}(w)=S_3(w)$. So, statement (4) of Lemma~\ref{tech-7} gives $S_{4}^\star(v) \cap S_6^{\ast}(w)=S_4(v)\setminus\left(S_1(v)\cup S_3(w)\right)$. However, this expression can be further simplified to $S_{4}^\star(v) \cap S_6^{\ast}(w)=S_4(v)\setminus S_3(w)$, because in this example we have $S_1(v)\subseteq S_3(w)$ (see (\ref{chain})). 

 Now let us consider $S_1^\star(v)\cap S_6^\star(w)$. By Lemma~\ref{tech-2} we have $S_1(v)\cap S_6(w)=S_1(v)\ne\emptyset$ because $v_{[2,7]}=w_{[7,12]}$, and by statement (4) of Lemma~\ref{tech-7}, the intersection $S_1^\star(v)\cap S_6^\star(w)$ can be expressed as $S_1^\star(v)\cap S_6^\star(w)=S_1(v)\setminus S_3(w)$. Therefore $S_1^\star(v)\cap S_6^\star(w)=\emptyset$, because $S_1(v)\subseteq S_3(w)$. What happens in this case is that $S_1(v)\subseteq S_3(w)=S_{3,6}(w)$, and hence statement (2) of Proposition~\ref{v2-second} or \ref{v2-second-cas 2} implies $S_1^\star(v)\cap S_6^\star(w)=\emptyset$.
 \end{remark}

\begin{proof}
By Lemma~\ref{tech-4} we have $ S_{i}^\star(v)=S_i(v)\setminus \Big ( \bigcup_{k=0}^{i-1} S_{k,i}(v)  \Big )$ and $ S_{j}^\star(w)=S_j(w)\setminus \Big ( \bigcup_{k=0}^{j-1} S_{k,j}(w)  \Big )$. Therefore,
$$
S_{i}^\star(v)\cap S_j^{\ast}(w)=\big(S_{i}(v)\cap S_j(w)\big) \setminus \bigg(\Big(\bigcup_{k=0}^{i-1} S_{k,i}(v)\Big) \cup \Big(\bigcup_{k=0}^{j-1} S_{k,j}(w)\Big) \bigg).
$$
By applying Lemma~\ref{tech-2}, it follows that $S_i(v)\subseteq S_j(w)$ if $i\le j\le D$ and $i<D$, while $S_j(w)\subseteq S_i(v)$ if $j=i-1$. Thus
$$
S_{i}^\star(v)\cap S_j^{\ast}(w)=\left \{
\begin{array}{l}
\displaystyle S_{i-1}(w) \setminus
\bigg(\Big(\bigcup_{k=0}^{i-1} S_{k,i}(v)\Big) 
\cup \Big(\bigcup_{k=0}^{i-2} S_{k,i-1}(w)\Big) \bigg) \textrm{ if }  j=i-1;\\[5mm]
\displaystyle \Big(S_{D}(v)\cap S_D(w)\Big) \setminus
\bigg(\Big(\bigcup_{k=0}^{D-1} S_{k,D}(v)\Big) 
\cup \Big(\bigcup_{k=0}^{D-1} S_{k,D}(w)\Big) \bigg) \\\hfill\textrm{ if } j=i=D; \\[3mm]
\displaystyle S_{i}(v) \setminus
\bigg(\Big(\bigcup_{k=0}^{i-1} S_{k,i}(v)\Big) 
\cup \Big(\bigcup_{k=0}^{j-1} S_{k,j}(w)\Big) \bigg) \textrm{ if } j\ge i \textrm{ and  } i\ne D. 
\end{array}
\right .
$$
By Lemma~\ref{tech-6} we have $\bigcup_{k=0}^{i-2} S_{k,i-1}(w)\subseteq \bigcup_{k=0}^{i-1} S_{k,i}(v)$, and hence the description of $S_{i}^\star(v) \cap S_{j}^{\ast}(w)$ formulated in statement (1) follows.

Next we are going to prove  the expression of $S_{i}^\star(v) \cap S_{j}^{\ast}(w)$ given in statements (2) and (3). So assume that $i=j=D$. Hence  
$$
S_{D}^\star(v)\cap S_D^{\ast}(w)=\Big(S_{D}(v)\cap S_D(w)\Big) \setminus
\bigg(\Big(\bigcup_{k=0}^{D-1} S_{k,D}(v)\Big) 
\cup \Big(\bigcup_{k=0}^{D-1} \big(S_{k,D}(w)\cap S_D(v)\big)\Big) \bigg).
$$
By statement (2.e) of Lemma~\ref{tech-5}, if $0\le k\le D-2$ and $S_{k,D}(w)\cap S_D(v)\neq \emptyset$, then there exists $k'\leq D-1$ such that $S_{k,D}(w)\subseteq S_{k',D}(v)$. Hence $S_{k,D}(w)\cap S_D(v)\subseteq S_{k',D}(v)\cap S_D(v)\subseteq S_{k',D}(v)$, and so 
$$
S_{D}^\star(v)\cap S_D^{\ast}(w)
=\Big(S_{D}(v)\cap S_D(w)\Big) \setminus
\bigg(\Big(\bigcup_{k=0}^{D-1} S_{k,D}(v)\Big) 
\cup \big(S_{D-1,D}(w)\cap S_D(v)\big) \bigg).
$$
If $G=K(d,D)$, then from Remark~\ref{rem_def_S_kj} we have $S_{D-1,D}(v)=S_{D-1,D}(w)=\emptyset$, and so the description given in statement (2) follows.  If $G=B(d,D)$, then from Remark~\ref{rem_def_S_kj} we have $S_{D-1,D}(v)=S_{D-1}(v)$, $S_{D-1,D}(w)=S_{D-1}(w)$, and from Lemma~\ref{tech-1} we have $S_D(v)=S_D(w)=V$. Hence statement (3) follows.

Now we are going to prove that if $j\ge i\ne D$, then  the  set $S_{i}^\star(v)\cap S_j^{\ast}(w)$ can be expressed as stated in statement (4). So, assume that $j\ge i\ne D$. Hence 
\begin{eqnarray*}
\lefteqn{S_{i}^\star(v)\cap S_j^{\ast}(w)}\\
&&=S_{i}(v) \setminus
\bigg(\Big(\bigcup_{k=0}^{i-1} S_{k,i}(v)\Big) 
\cup \Big(\bigcup_{k=0}^{i-2} S_{k,j}(w)\Big)
\cup \Big(\bigcup_{k=i-1}^{j-1} S_{k,j}(w)\Big)\bigg)\\
&&
=S_{i}(v) \setminus\bigg(\Big(\bigcup_{k=0}^{i-1} S_{k,i}(v)\Big) 
\cup \Big(\bigcup_{k=0}^{i-2} \big(S_{k,j}(w)\cap S_i(v)\big) \Big)
\cup \Big(\bigcup_{k=i-1}^{j-1} S_{k,j}(w)\Big)\bigg).
\end{eqnarray*}
By statement (2.e) of Lemma~\ref{tech-5}, if $0\le k\le i-2$ and $S_{k,j}(w)\cap S_i(v)\neq \emptyset$, then there exists $k'\leq i-1$ such that $S_{k,j}(w)\subseteq S_{k',i}(v)$. Hence $S_{k,j}(w)\cap S_i(v)\subseteq S_{k',i}(v)\cap S_i(v)\subseteq S_{k',i}(v)$, and so 
\begin{equation*}
\label{descri}
S_{i}^\star(v)\cap S_j^{\ast}(w)=S_{i}(v) \setminus\bigg(\Big(\bigcup_{k=0}^{i-1} S_{k,i}(v)\Big) \cup \Big(\bigcup_{k=i-1}^{j-1} S_{k,j}(w)\Big)\bigg),
\end{equation*}
as claimed in statement (4).
\end{proof}

\begin{lemma}
\label{tech-7bis}
Let $v\in V$ and $w\in S_1(v)$. Let $j\ge i$ and $i\ne D$. Let $S_i(v)\cap S_j(w)\neq \emptyset$ and $S_i(v)\not\subseteq S_{k,j}(w)$ for $i\leq  k \le j-1$.  Then $S_{i}^\star(v)\cap S_j^{\ast}(w)$ can be described as
\begin{equation*}
S_{i}^\star(v)\cap S_j^{\ast}(w)= S_{i}(v) \setminus\bigg(\Big(\bigcup_{k=0}^{i-2} S_{k,i}(v)\Big)\cup S'  \bigg),
\end{equation*}
where $S'\subseteq S_i(v)$ is either $\emptyset$, or  $S_{i-1}(v)$, or  $S_{i-1}(w)$, or $S_{i-1}(v) \cup S_{i-1}(w)$. Namely 
\begin{enumerate}
\item If $S_{i-1,j}(w)=\emptyset$ and $v_{[i,D-1]}\ne v_{[i+1,D]}$, then $S'=\emptyset$.
 \item If $S_{i-1,j}(w)=\emptyset$ and $v_{[i,D-1]}=v_{[i+1,D]}$, then $S'=S_{i-1}(v)$.
\item If $S_{i-1,j}(w)\ne \emptyset$ and $v_{[i,D-1]}\ne v_{[i+1,D]}$, then $S'=S_{i-1}(w)$.
\item If $S_{i-1,j}(w)\ne \emptyset$, $v_{[i,D-1]}=v_{[i+1,D]}$, and $j<D$, then $S'=S_{i-1}(v)=S_{i-1}(w)$.
\item If $S_{i-1,j}(w)\ne \emptyset$, $v_{[i,D-1]}=v_{[i+1,D]}$, and $j=D$, then $S'=S_{i-1}(v) \cup S_{i-1}(w)$ and $S_{i-1}(v) \cap S_{i-1}(w)=\emptyset$.
\end{enumerate}
\end{lemma}


\begin{proof}
From statement (4) of Lemma~\ref{tech-7}, the intersection set $S_{i}^\star(v) \cap S_j^{\ast}(w)$ can be written as 
\begin{eqnarray}
\label{eq1-proof}
\nonumber \lefteqn{S_{i}^\star( v) \cap S_j^{\ast}(w)=}\\
&&=S_{i}(v) \setminus
\bigg(\Big(\bigcup_{k=0}^{i-1} S_{k,i}(v)\Big) 
\cup \Big(S_{i-1,j}(w) \cap S_i(v)\Big)
\cup \Big(\bigcup_{k=i}^{j-1} \big (S_{k,j}(w) \cap S_i(v) \big)\Big)\bigg)
\end{eqnarray}
if $i<j$, and
\begin{equation}
\label{eq2-proof}
S_{i}^\star( v) \cap S_j^{\ast}(w)=
S_{i}(v) \setminus
\bigg(\Big(\bigcup_{k=0}^{i-1} S_{k,i}(v)\Big) 
\cup \Big(S_{i-1,j}(w) \cap S_i(v) \Big)\bigg)
\end{equation}
if $i=j$. If $S_i(v)\not \subseteq S_{k,j}(w)$ for  $i\leq  k \le j-1$, then from statement (2.e) of Lemma~\ref{tech-5} we get that $S_{k,j}(w) \cap S_i(v)=\emptyset$ for $i\le k \le j-1$. Therefore, our assumptions imply that the following equality holds both for $i<j$ and for $i=j$:
\begin{align*}
S_{i}^\star( v) \cap S_j^{\ast}(w) &= S_{i}(v) \setminus \bigg(\Big(\bigcup_{k=0}^{i-1} S_{k,i}(v)\Big) \cup \Big(S_{i-1,j}(w) \cap S_i(v) \Big)\bigg)=S_{i}(v) \setminus \bigg(\Big(\bigcup_{k=0}^{i-2} S_{k,i}(v)\Big)\cup S'\bigg),
\end{align*}
where $S'=S_{i-1,i}(v)\cup \left(S_{i-1,j}(w) \cap S_i(v) \right)$.

Observe that if  $S_{i-1,j}(w)=\emptyset$, then $S'=S_{i-1,i}(v)$. Therefore, statements (1) and (2) follow from Remark~\ref{rem_S_kj-seq}. Moreover, by the definition of $S_{i-1,i}(v)$, we have $S'=S_{i-1,i}(v)\subseteq S_i(v)$.

From now on we assume  that $S_{i-1,j}(w)\ne\emptyset$. In this case we have $S'=S_{i-1,i}(v)\cup \left(S_{i-1}(w) \cap S_i(v) \right)$, and, so, $S'=S_{i-1,i}(v)\cup S_{i-1}(w)\subseteq S_i(v)$ because $S_{i-1}(w) \subseteq S_i(v)$ (recall that $w\in S_1(v)$). 

By Remark~\ref{rem_S_kj-seq}, if  $v_{[i,D-1]}\ne v_{[i+1,D]}$, then $S'= S_{i-1}(w)$, proving statement (3). Therefore, the proof of the lemma will be completed by demonstrating statements  (4) and (5).

Thus assume that $S_{i-1,j}(w)\ne \emptyset$ and $v_{[i,D-1]}\ne v_{[i+1,D]}$, or, equivalently, assume that $S_{i-1,j}(w)\ne \emptyset$ and $S_{i-1,i}(v)\ne \emptyset$. Hence we have $S'=S_{i-1,i}(v)\cup S_{i-1}(w)=S_{i-1}(v)\cup S_{i-1}(w)$. To prove (4) and (5) we are going to apply  Lemma~\ref{tech-5} with $k=i-1$. Observe that we are under the assumptions of this lemma because, since $S_{i-1,j}(w)\ne \emptyset$ and $w\in S_1(v)$, then $S_{i-1,j}(w)\cap S_i(v)=S_{i-1}(w)\cap S_i(v)=S_{i-1}(w)\ne \emptyset$.

Let us prove (4). If $j<D$, since we are assuming $S_{i-1,i}(v)\ne \emptyset$, from statement (2.c) of Lemma~\ref{tech-5}, we conclude that $j=i$ and $S_{i-1,i}(v)=S_{i-1,i}(w)=S_{i-1}(v)=S_{i-1}(w)$. Thus $S'=S_{i-1}(v)=S_{i-1}(w)$, as we wanted to prove.

Finally, let us prove (5). If $j=D$, since we are assuming $S_{i-1,i}(v)\ne \emptyset$, now from statement (2.d) of Lemma~\ref{tech-5}, we get $w_D\ne v_D$. In particular we have $v_{[i, D]}\ne w_{[i,D]}$. So, by  statement (1) of Lemma~\ref{tech-2}, we have $S_{i-1}(v) \cap S_{i-1}(w)=\emptyset$. This completes the proof of the lemma.
\end{proof}

\subsection{Proof of Proposition~\ref{propo 2.3}}
\label{proof-2.3}
On one hand, from Lemma~\ref{tech-4} we have $S_{i}^\star(v)=S_i(v)\setminus \Big (\bigcup_{k=0}^{i-1} S_{k,i}(v)  \Big )$. On the other hand, from the definition of $S_{k,i}(v)$ we have $S_{k,i}(v) \subseteq S_i(v)$. Therefore, 
\begin{equation*}
\label{card-disj}
|S_{i}^\star(v)|= \Big | S_i(v)\setminus \Big (\bigcup_{k=0}^{i-1} S_{k,i}(v) \Big ) \Big |=|S_i(v)|-\Big |\bigcup_{k=0}^{i-1} S_{k,i}(v) \Big|.
\end{equation*}
By Remark~\ref{void}, if $k_1\ne k_2$, then $S_{k_1,i}(v) \cap S_{k_2,i}(v)=\emptyset$. It follows that $\Big |\bigcup_{k=0}^{i-1} S_{k,i}(v) \Big|=\sum _{k=0}^{i-1} |S_{k,i}(v)|$. Therefore, from the definition of $S_{k,i}(v)$ and by Lemma~\ref{tech-1}, we have 
$$
|S_{i}^\star(v)|=|S_i(v)|-\sum _{k=0}^{i-1} |S_{k,i}(v)|=d^i-\sum _{k=0}^{i-1} a_{k}d^{k},
$$
where the coefficients $a_{k}$ are $0$ or $1$, and $a_{k}=1$ if and only if $S_{k,i}(v)\neq\emptyset$. Clearly, if $i=D$, then $S_{i-1,i}(v)\ne\emptyset$ if and only if $G=B(d,D)$. To conclude, assume that $i<D$. In such a case we have $a_{i-1}=1$ if and only if $S_{i-1,i}(v)\neq\emptyset$; if and only if $S_{i-1}(v)\subseteq S_i(v)$; if and only if  $v_{[i,D-1]}=v_{[i+1,D]}$ (the last equivalence follows from statement (1) of Lemma~\ref{tech-2} which can be applied because $i<D$). Therefore we conclude that if $i<D$, then $a_{i-1}=1$ if and only if $G=B(d,D)$ and $v_i=v_{i+1}=\cdots=v_D$. This completes the proof of Proposition~\ref{propo 2.3}.

\subsection{Proof of Proposition~\ref{v2-first}}
We have to prove that there exists at most one integer $j_0$, $i\le j_0\le D$, such that $S_{i}^\star(v) \cap S_{j_0}^{\star}(w)\ne \emptyset$, because if so, statements (1) and (2) follow. To prove that, let us demonstrate that if $S_{i}^\star(v) \cap S_{j}^{\star}(w)\ne \emptyset$ for some $j$, $i\le j<D$, then $S^\ast_i(v)\cap S^\ast_{j'}(w)= \emptyset$ for all integer $j'$ such that $j<j'\le D$. Thus assume $S_{i}^\star(v) \cap S_{j}^{\star}(w)\ne \emptyset$ and let $j<j'\le D$. On one hand, if $S_{i}^\star(v) \cap S_{j}^{\star}(w)\ne \emptyset$, then $S_i(v)\cap S_{j}(w)\neq \emptyset$ and, since $i<D$, we conclude from Lemma~\ref{tech-2} that $S_i(v)\subseteq S_{j}(w)$. On the other hand, by definition we have $S_{j'}^\star(w)=S_{j'}(w)\setminus \left (\bigcup_{k=0}^{j'-1}S_{k}(w) \right )$ and, since  $j<j'$, we get that $S_{j}(w)\cap S_{j'}^\star(w)=\emptyset$. Thus we have $S_{i}(v)\cap S_{j'}^\star(w)=\emptyset$, because $S_i(v)\subseteq S_j(w)$, and therefore we conclude that $S_{i}^\star(v)\cap S_{j'}^\star(w)=\emptyset$, as we wanted to prove.

\subsection{Proof of statement (1) of Propositions~\ref{v2-second} and \ref{v2-second-cas 2}}  
Let $i=j=D$. We have to prove that if $d\ge 3$, then $S_{i}^\star(v) \cap S_{j}^{\star}(w)\ne\emptyset$; whereas if $d=2$, then $S_{i}^\star(v) \cap S_{j}^{\star}(w)\ne\emptyset$ if and only if $G=K(d,D)$ or $v_D=w_D$. Equivalently, we must demonstrate that if $G=K(d,D)$, then $S_{i}^\star(v) \cap S_{j}^{\star}(w)\ne\emptyset$; while if $G=B(d,D)$, then $S_{i}^\star(v) \cap S_{j}^{\star}(w)\ne\emptyset$ if and only if $d\ge 3$ or $v_D=w_D$.
 
If $G$ is the Kautz digraph $K(d,D)$, then, by statement (3) of Lemma~\ref{tech-2}, we have $S_{D}(v)\cap S_D(w)\ne\emptyset$ and $|S_D(v)\cap S_D(w)|= d^D-d^{D-1}$. Since $S_{D}(v)\cap S_D(w)\ne\emptyset$, by statement (2) of Lemma~\ref{tech-7}, the intersection $S_{D}^\star(v)\cap S_D^{\ast}(w)$ can be expressed as
$$
S_{D}^\star(v)\cap S_D^{\ast}(w)=\Big(S_D(v)\cap S_D(w)\Big)\setminus \bigcup_{k=0}^{D-2} S_{k,D}(v).
$$
Moreover, by Definition~\ref{def_S_kj} and Lemma~\ref{tech-1}, we have either $S_{k,D}(v)=\emptyset$ or $|S_{k,D}(v)|=|S_k(v)|=d^k$. Therefore the cardinality of the union $\bigcup_{k=0}^{D-2} S_{k,D}(v)$ can be bounded as follows:
$$
\left |\bigcup_{k=0}^{D-2} S_{k,D}(v)\right | \le \sum_{k=0}^{D-2} d^k=\frac{d^{D-1}-1}{d-1},
$$
and hence
\begin{equation*}
|S_{D}^\star(v)\cap S_D^{\ast}(w)|\ge \left(d^D-d^{D-1}\right)-\frac{d^{D-1}-1}{d-1}=\frac{(d-2) d^D+1}{d-1}>0.
\end{equation*}
In particular, we get that if $G=K(d,D)$, then $S_{D}^\star(v)\cap S_D^{\ast}(w)\ne\emptyset$, as we wanted to prove.

Now let us assume $G=B(d,D)$. We must demonstrate that, in this case, we have $S_{D}^\star(v) \cap S_{D}^{\star}(w)\ne\emptyset$ if and only if $d\ge 3$ or $v_D=w_D$.

If $G=B(d,D)$,  then $S_D(v)\cap S_D(w)=V$ and, by statement (3) of Lemma~\ref{tech-7}, we can write $S_{D}^\star(v)\cap S_D^{\ast}(w)$ as
\begin{equation}
\label{d2-empty}
S_{D}^\star(v)\cap S_D^{\ast}(w)=V\setminus \Big( S_{D-1}(v)\cup S_{D-1}(w)\cup  \bigcup_{k=0}^{D-2} S_{k,D}(v) \Big ).
\end{equation}

By taking into account again  that either $S_{k,D}(v)=\emptyset$ or $|S_{k,D}(v)|=|S_k(v)|=d^k$, we have 
\begin{eqnarray*}
\lefteqn{\left |S_{D-1}(v)\cup S_{D-1}(w)\cup  \bigcup_{k=0}^{D-2} S_{k,D}(v)\right|}\\
&&  \le \left |S_{D-1}(v)\cup S_{D-1}(w)\right |+\sum_{k=0}^{D-2} \left|S_{k,D}(v)\right | 
\le \left |S_{D-1}(v)\cup S_{D-1}(w)\right |+\frac{d^{D-1}-1}{d-1},
\end{eqnarray*}
and thus
\begin{equation}
\label{S-prime-2-2}
|S_{D}^\star(v)\cap S_D^{\ast}(w)|\ge d^D-\left |S_{D-1}(v)\cup S_{D-1}(w)\right |-\frac{d^{D-1}-1}{d-1}.
\end{equation}

At this point we distinguish two cases: $d\ge 3$ and $d=2$.

First assume $d\ge 3$. Since $|S_{D-1}(v)|=|S_{D-1}(w)|=d^{D-1}$, we have the bound $\left |S_{D-1}(v)\cup S_{D-1}(w)\right |\le 2 d^{D-1}$. 
Hence it follows from (\ref{S-prime-2-2}) that
\begin{equation*}
|S_{D}^\star(v)\cap S_D^{\ast}(w)|\ge d^D-2 d^{D-1}-\frac{d^{D-1}-1}{d-1}=\frac{(d-3) d^D+d^{D-1}+1}{d-1}>0,
\end{equation*}
because $d\ge 3$. Therefore we have proved that if $G=B(d,D)$ and $d\ge 3$, then $S_{D}^\star(v)\cap S_D^{\ast}(w)\ne\emptyset$. 

Finally, assume $d=2$. In this case we must demonstrate that if $v_D=w_D$, then $S_{D}^\star(v) \cap S_{D}^{\star}(w)\ne\emptyset$; while if $v_D\ne w_D$, then $S_{D}^\star(v)\cap S_D^{\ast}(w)=\emptyset$.

If $v_D=w_D$, then, by statement (1) of Lemma~\ref{tech-2}, we have $S_{D-1}(v)=S_{D-1}(w)$. So if $v_D=w_D$,  from (\ref{S-prime-2-2}) and by taking into account that $d=2$, $|V|=2^D$, and $|S_{D-1}(v)|=2^{D-1}$, we get 
$$|S_{D}^\star(v)\cap S_D^{\ast}(w)|\ge 2^D-2^{D-1}-\left(2^{D-1}-1\right)=1,$$
which demonstrates that $S_{D}^\star(v)\cap S_D^{\ast}(w)\ne\emptyset$. 

To end suppose $v_D\ne w_D$. Since we are assuming $d=2$, we conclude that $v_D$ and $w_D$ are the two different symbols of the base alphabet $A$ for the sequence representation of the vertices. Moreover, by using this sequence representation, it is easy to check that in this case we have $S_{D-1}(v)\cup S_{D-1}(w)=V$. Therefore, if $v_D\ne w_D$  we conclude from (\ref{d2-empty}) that $S_{D}^\star(v)\cap S_D^{\ast}(w)=\emptyset$.  This completes the proof.

\subsection{Proof of statement (2) of Propositions~\ref{v2-second} and \ref{v2-second-cas 2}}  
\label{sec:3.5}
Let $j\ge i\ne D$. We have to prove that if $d\ge 3$, then $S_{i}^\star(v) \cap S_{j}^{\star}(w)\ne\emptyset$ if and only if $S_i(v)\cap S_j(w)\neq \emptyset$ and  $S_i(v)\not\subseteq S_{k,j}(w)$ for  $i \le  k < j$; whereas if $d=2$, then  $S_{i}^\star(v) \cap S_{j}^{\star}(w)\ne\emptyset$ if and only if $S_i(v)\cap S_j(w)\neq \emptyset$,  $S_i(v)\not\subseteq S_{k,j}(w)$ for  $i \le  k < j$, and one of the following conditions holds:
\begin{enumerate}
\item $j<D$;
\item $j=D$, and $v_{[i,D-1]}\ne v_{[i+1,D]}$ or $S_{i-1,j}(w)=\emptyset$.
\end{enumerate} 

Firstly we claim that, for any $d\ge 2$, if $j\ge i$ and  $S_{i}^\star(v)\cap S_j^{\ast}(w)\neq \emptyset$, then $S_i(v)\cap S_j(w)\neq \emptyset$ and $S_i(v)\not\subseteq S_{k,j}(w)$ for $i \leq  k <j$. Clearly, if $S_{i}^\star(v)\cap S_j^{\ast}(w)\neq \emptyset$, then  $S_i(v)\cap S_j(w)\neq \emptyset$. If $j=i$ we are done, because in this case the condition $S_i(v)\not\subseteq S_{k,j}(w)$ for $i \leq  k< j$ is empty. So, let us assume $j>i$. Since $S_i(v)\cap S_j(w)\neq \emptyset$, by statement (4) of Lemma~\ref{tech-7}, the intersection set $S_{i}^\star(v) \cap S_j^{\ast}(w)$ can be written as 
$$
S_{i}^\star( v) \cap S_j^{\ast}(w)=S_{i}(v) \setminus
\bigg(\Big(\bigcup_{k=0}^{i-1} S_{k,i}(v)\Big) 
\cup \Big(\bigcup_{k=i-1}^{j-1} S_{k,j}(w) \Big)\bigg).
$$
So if $S_i(v) \subseteq S_{k,j}(w)$ for some $k$, $i \leq  k < j$, then we conclude from the above expression that $S_{i}^\star(v) \cap S_j^{ \ast}(w)=\emptyset$. This finishes the proof of our claim.

Now we are going to demonstrate that if $d=2$, $j=D$, and  $S_{i}^\star(v)\cap S_j^{\ast}(w)\neq \emptyset$, then $v_{[i,D-1]}\ne v_{[i+1,D]}$ or $S_{i-1,j}(w)=\emptyset$. 

We claim that if $d=2$, $j=D$, $v_{[i,D-1]}=v_{[i+1,D]}$, and $S_{i-1,j}(w)\ne \emptyset$, then $S_{i-1}(v) \cup S_{i-1}(w)=S_i(v)$. Indeed, on one hand, since $v_{[i,D-1]}=v_{[i+1,D]}$, from statement (1) of  Lemma~\ref{tech-2} we conclude that $S_{i-1}(v)\subseteq S_i(v)$. So, by Definition~\ref{def_S_kj}, we have  $ S_{i-1,i}(v)\ne\emptyset$. On the other hand, since  $S_{i-1,D}(w)\ne\emptyset$, then $S_{i-1,D}(w)=S_{i-1}(w)$, and hence $S_{i-1,D}(w)\cap S_i(v)=S_{i-1}(w)\cap S_i(v)=S_{i-1}(w)$, because $w\in S_1(v)$. In particular, $S_{i-1,D}(w)\cap S_i(v)\ne\emptyset$. Therefore, if $i\ne D=j$, $S_{i-1,j}(w)\ne \emptyset$, and $v_{[i,D-1]}=v_{[i+1,D]}$, then we can apply statement (2.d) of Lemma~\ref{tech-5} to conclude that $w_D\ne v_D$. 

To finish the proof of our claim, let us use the sequence representation of the vertices to check that if $d=2$, $i<D$, $v_{[i,D-1]}=v_{[i+1,D]}$, and $w_D\ne v_D$, then $S_{i-1}(v) \cup S_{i-1}(w)=S_i(v)$. 

In fact, if $i<D$ and $v_{[i,D-1]}=v_{[i+1,D]}$, then $v_i=v_{i+1}=\cdots=v_D=\alpha$ for some symbol $\alpha$ of the base alphabet $A$. In particular we have $G=B(d,D)$. Hence a vertex $z$ belongs to $S_{i-1}(v)$ if and only if $z_{[1,D-i+1]}=v_{[i,D]}=\alpha\cdots\alpha$. Analogously, since $w=v_2\cdots v_D w_D$ (because $w\in S_1(v)$), a vertex $z'$ is in $S_{i-1}(w)$ if and only if $z'_{[1,D-i+1]}=w_{[i,D]}=\alpha\cdots\alpha w_D$. Moreover, we have $z''\in S_{i}(v)$ if and only if $z''_{[1,D-i]}=w_{[i+1,D]}=\alpha\cdots\alpha$. If we assume $d=2$ and $v_D\ne w_D$, then $v_D=\alpha$ and $w_D$ are the two symbols of the base alphabet $A$, that is, we have $A=\{\alpha,w_D\}$. Now it is clear that a vertex $z$ belongs to $S_{i-1}(v)\cup S_{i-1}(w)$ if and only if $z_{[1,D-i]}=\alpha\cdots\alpha$; if and only if $z\in S_i(v)$. Hence we have $S_{i-1}(v) \cup S_{i-1}(w)=S_i(v)$, as we wanted to check. 

This completes the proof of our claim.

From our claim and by statement (5) of Lemma~ \ref{tech-7bis}, we conclude that if $d=2$, $j=D$, $S_{i}^\star(v)\cap S_j^{\ast}(w)\neq \emptyset$, $S_{i-1,D}(w)\ne \emptyset$, and $v_{[i,D-1]}=v_{[i+1,D]}$, then 
\begin{equation*}
S_{i-1}(v) \cup S_{i-1}(w)=S_i(v)\textrm{ and } S_{i}^\star(v)\cap S_D^{\ast}(w)= S_{i}(v) \setminus\bigg(\Big(\bigcup_{k=0}^{i-2} S_{k,i}(v)\Big)\cup S'  \bigg),
\end{equation*}
where $S'=S_{i-1}(v) \cup S_{i-1}(w)$ and $S_{i-1}(v) \cap S_{i-1}(w)=\emptyset$. Since $S_{i-1}(v) \cup S_{i-1}(w)=S_i(v)$, we get $S_{i}^\star(v)\cap S_D^{\ast}(w)=\emptyset$, which is a contradiction.

At this point we have proved the direct implication of statement (2) of Propositions~\ref{v2-second} and \ref{v2-second-cas 2}. 

To complete the proof we are going to show that if $S_i(v)\cap S_j(w)\neq \emptyset$,   $S_i(v)\not\subseteq S_{k,j}(w)$ for  $i \le  k < j$, and if one of the following conditions hold:
\begin{enumerate}
	\item[(i)] $d\ge 3$;
	\item[(ii)] $j<D$;
	\item[(iii)] $j=D$, and $v_{[i,D-1]}\ne v_{[i+1,D]}$ or $S_{i-1,j}(w)=\emptyset$,
\end{enumerate}
then $S_{i}^\star(v)\cap S_j^{\ast}(w)\neq \emptyset$. 

Let us  assume  $S_i(v)\cap S_j(w)\neq \emptyset$,   $S_i(v)\not\subseteq S_{k,j}(w)$ for  $i \le  k < j$, and that either condition (i), or (ii), or (iii) is fulfilled.

Since $S_i(v)\cap S_j(w)\neq \emptyset$ and  $S_i(v)\not\subseteq S_{k,j}(w)$ for  $i \le  k < j$, by Lemma~\ref{tech-7bis} we deduce that the intersection set $S_{i}^\star(v)\cap S_j^{\ast}(w)$ can be described as
\begin{equation}
\label{S-prime}
S_{i}^\star(v)\cap S_j^{\ast}(w)=
S_{i}(v) \setminus\bigg(\Big(\bigcup_{k=0}^{i-2} S_{k,i}(v)\Big)\cup S'  \bigg),
\end{equation}
where $S'\subseteq S_{i-1}(v) \cup S_{i-1}(w)$. 

First assume that we are under condition (i); that is, $d\ge 3$. Since either $S_{k,i}(v)=\emptyset$ or $|S_{k,i}(v)|=|S_k(v)|=d^k$, and $|S_{i-1}(v)|=|S_{i-1}(v)|=d^{i-1}$, we have
\begin{align}
\label{S-prime-bis}
\left | \bigcup_{k=0}^{i-2} S_{k,i}(v) \cup S' \right | 
& \le \sum_{k=0}^{i-2}\left| S_{k,i}(v)\right |+ \left |S_{i-1}(v)\right|+\left| S_{i-1}(w)\right |\le \frac{d^{i-1}-1}{d-1}+ 2d^{i-1}=\frac{2d^i-d^{i-1}-1}{d-1},
\end{align}
and hence 
\begin{equation}
\label{S-prime-2}
|S_{i}^\star(v)\cap S_j^{\ast}(w)|\ge d^i-\frac{2d^i-d^{i-1}-1}{d-1}=\frac{(d-3) d^i+d^{i-1}+1}{d-1}>0,
\end{equation}
because $d\ge 3$. Therefore, if condition (i) holds, then $S_{i}^\star(v)\cap S_j^{\ast}(w)\ne\emptyset$, as we wanted to prove.

Now we must demonstrate that $S_{i}^\star(v)\cap S_j^{\ast}(w)\ne\emptyset$ if either condition (ii) or (iii) is satisfied. 

First observe that if either condition (ii) or (iii) is satisfied, then, by statements (1), (2), (3), or (4) of 
Lemma ~\ref{tech-7bis}, the set $S'$ in (\ref{S-prime}) is either $S'=\emptyset$, or $S'=S_{i-1}(v)$, or $S'=S_{i-1}(w)$. Therefore, in any case we have $|S'|\le d^{i-1}$, and so we have the bound
\begin{equation}
\label{S-prime-4}
 \left | \bigcup_{k=0}^{i-2} S_{k,i}(v) \cup S' \right | 
 \le \sum_{k=0}^{i-2}\left| S_{k,i}(v)\right |+ \left |S'\right| 
 \le \frac{d^{i-1}-1}{d-1}+ d^{i-1}=\frac{d^i-1}{d-1},
\end{equation}
and hence 
\begin{equation}
\label{S-prime-5}
|S_{i}^\star(v)\cap S_j^{\ast}(w)|\ge d^i-\frac{d^i-1}{d-1}=\frac{(d-2) d^i+1}{d-1}.
\end{equation}
So $|S_{i}^\star(v)\cap S_j^{\ast}(w)|>0$ for any $d\ge 2$. Therefore we have $S_{i}^\star(v)\cap S_j^{\ast}(w)\ne\emptyset$.

\subsection{Proof of statement (3) of Propositions~\ref{v2-second} and \ref{v2-second-cas 2}}  
Here we assume $d\ge 2$. 
We must prove that: 
\begin{enumerate}
	\item[(a)] the intersection  $S_{i}^\star(v) \cap S_{i-1}^{\star}(w)$ is empty if and only if $G=B(d,D)$ and $v_i=v_{i+1}=\cdots=v_D=w_D$.
	\item[(b)] if $S_{i}^\star(v)\cap S_{i-1}^{\ast}(w)=\emptyset$, then $S_{i}^\star(v)\cap S_{i}^{\ast}(w)\ne\emptyset$.
\end{enumerate}

Let us prove statement (a).

First of all observe that $S_i(v)\cap S_{i-1}(w)\neq \emptyset$, because  $w\in S_1(v)$. So we can apply statement (1) of Lemma~\ref{tech-7} to write  $S_{i}^\star(v)\cap S_{i-1}^{\ast}(w)$ as
\begin{align}
\nonumber S_{i}^\star(v)\cap S_{i-1}^{\ast}(w) &= S_{i-1}(w) \setminus\bigcup_{k=0}^{i-1}  S_{k,i}(v)\\
\nonumber &= S_{i-1}(w) \setminus\Big(\bigcup_{k=0}^{i-2}  S_{k,i}(v) \cup  \Big(S_{i-1,i}(v)\cap S_{i-1}(w)\Big)\Big)\\
\label{empty}
&=S_{i-1}(w) \setminus\Big(\bigcup_{k=0}^{i-2}  S_{k,i}(v) \cup  S'\Big),
\end{align}
where  $S'=S_{i-1,i}(v)\cap S_{i-1}(w)$. Next we are going to prove that either $S'=\emptyset$ or $S'=S_{i-1}(v)=S_{i-1}(w)$. To this end, we only must prove that if $S_{i-1,i}(v)\cap S_{i-1}(w)\ne\emptyset$, then $S_{i-1,i}(v)\cap S_{i-1}(w)=S_{i-1}(v)=S_{i-1}(w)$. So assume $S_{i-1,i}(v)\cap S_{i-1}(w)\ne\emptyset$. In particular we have $S_{i-1,i}(v)\ne\emptyset$ and hence, from Definition~\ref{def_S_kj}, we get that $S_{i-1,i}(v)=S_{i-1}(v)$. So our assumption implies that $S_{i-1}(v)\cap S_{i-1}(w)\ne\emptyset$. Now, by applying statement (1) of Lemma~\ref{tech-2}, we have  $S_{i-1}(v)\subseteq S_{i-1}(w)$ and $S_{i-1}(w)\subseteq S_{i-1}(v)$. Hence $S_{i-1,i}(v)\cap S_{i-1}(w)=S_{i-1}(v)\cap S_{i-1}(w)=S_{i-1}(v)=S_{i-1}(w)$.

By Definition~\ref{def_S_kj} and Lemma~\ref{tech-1} we know that $|S_{i-1}(w)|=d^{i-1}$ and that, either $S_{k,i}(v)=\emptyset$ or $|S_{k,i}(v)|=|S_k(v)|=d^k$. Therefore, from (\ref{empty}) we conclude that\begin{equation*}
|S_{i}^\star(v)\cap S_{i-1}^{\ast}(w)| = |S_{i-1}(w)|- \sum_{k=0}^{i-2} \left | S_{k,i}(v)\right |-|S'|\ge d^{i-1}-\sum_{k=0}^{i-2} d^k -|S'|.
\end{equation*}
Thus, since 
$$d^{i-1}-\sum_{k=0}^{i-2}d^k=d^{i-1}-\frac{d^{i-1}-1}{d-1}=\frac{(d-2)d^{i-1}+1}{d-1}>0,$$
we have $|S_{i}^\star(v)\cap S_{i-1}^{\ast}(w)|=0$ if and only if $|S'|\ne 0$. Therefore $S_{i}^\star(v)\cap S_{i-1}^{\ast}(w)=\emptyset$ if and only if $S_{i-1,i}(v)\cap S_{i-1}(w)=S_{i-1}(v)=S_{i-1}(w)$. 

By statement (1) of Lemma~\ref{tech-2}, we have $S_{i-1}(v)= S_{i-1}(w)$ if and only if $v_{[i,D]}=w_{[i,D]}$. Thus, since $w\in S_1(v)$, we conclude that $S_{i-1}(v)= S_{i-1}(w)$ if and only if $v_i=v_{i+1}=\cdots=v_D=w_D$; if and only if $G=B(d,D)$ and $v_i=v_{i+1}=\cdots=v_D=w_D$.

Therefore the proof of (a) will be completed by showing that if $S_{i-1}(v)=S_{i-1}(w)$, then  $S_{i-1,i}(v)\cap S_{i-1}(w)=S_{i-1}(v)=S_{i-1}(w)$. Let us prove this. Assume $S_{i-1}(v)=S_{i-1}(w)$. Then $S_{i-1}(v)=S_{i-1}(w)\subseteq S_i(v)$, because $w\in S_1(v)$. Therefore, by  Definition~\ref{def_S_kj}, we have $S_{i-1,i}(v)=S_{i-1}(v)$. So $S_{i-1,i}(v)\cap S_{i-1}(w)=S_{i-1}(v)\cap S_{i-1}(w)=S_{i-1}(v)=S_{i-1}(w)$, as we wanted to prove.

Now let us demonstrate  (b); that is, we have to prove that if $S_{i}^\star(v) \cap S_{i-1}^{\star}(w)=\emptyset$, then $S_{i}^\star(v) \cap S_{i}^{\star}(w)\ne\emptyset$. 

So let us assume $S_{i}^\star(v) \cap S_{i-1}^{\star}(w)=\emptyset$ and thus, by (a), we have $G=B(d,D)$ and $v_i=v_{i+1}=\cdots=v_D=w_D$. 

If $i=D$ and $d\ge 3$ there is nothing to prove, because in this case we have $S_D^\star(v)\cap S_D^\star(w)\ne\emptyset$ by statement (1) of Propositions~\ref{v2-second}; whereas if $i=D$ and $d=2$, then, since $v_D=w_D$, we also have $S_D^\star(v)\cap S_D^\star(w)\ne\emptyset$ by statement (1) of Propositions~\ref{v2-second-cas 2}.

Hence assume $i<D$. Since $v_i=v_{i+1}=\cdots=v_D=w_D$, we have $v_{[i+1,D-1]}=v_{[i+2,D]}$ and $v_D=w_D$, which implies $v_{[i+1,D]}=w_{[i+1,D]}$, because $w\in S_1(v)$. Thus we conclude from Lemma~\ref{tech-2} that $S_i(v)=S_i(w)$ and so $S_i(v)\cap S_i(w)\ne\emptyset$. Therefore if $d\ge 3$, then it follows from statement (2) of Propositions~\ref{v2-second} that $S_i^\star(v)\cap S_i^\star(w)\ne\emptyset$. Whereas if $d=2$, since $i<D$, we also have $S_i^\star(v)\cap S_i^\star(w)\ne\emptyset$, because of condition (a) of statement (2) of Proposition~\ref{v2-second-cas 2}. This concludes the proof of (b).

\subsection{Proof of statement (4) of Propositions~\ref{v2-second} and \ref{v2-second-cas 2}}
We have to prove the following two statements:
\begin{itemize}
\item[(a)] If $d\ge 3$, then there exists a unique integer $j$, $i\le j\le D$, such that the intersection $S_{i}^\star(v) \cap S_{j}^{\star}(w)$ is nonempty. 
\item[(b)] If $d=2$, then  the intersection $S_{i}^\star(v) \cap S_{j}^{\star}(w)$  is empty for all integer $j$, $i\le j\le D$, if and only if $G=B(d,D)$  and $v_i=v_{i+1}=\cdots=v_D\ne w_D$. 
\end{itemize}

Before proving (a) and (b) let us demonstrate the following claim: for any $v\in V$ and $w\in S_1(v)$, either $S_i(v)\cap S_{D-1}(w)\ne \emptyset$ or  $S_i(v)\cap S_{D}(w)\ne \emptyset$. 

Indeed, the claim clearly holds whenever $G=B(d,D)$, because in this case we have $S_D(w)=V$. If $G$ is the Kautz digraph $K(d,D)$, we conclude from statements (1) and (2) of Lemma~\ref{tech-2} that if $v_{i+1}=w_D$, then $S_i(v)\cap S_{D-1}(w)\ne\emptyset$; while if $v_{i+1}\ne w_D$, then  $S_i(v)\cap S_{D}(w)\ne\emptyset$. This finishes the proof of our claim.

The above claim guarantees that the set of integers $\{\ell:\ i\le\ell\le D \textrm{ and } S_i(v)\cap S_{\ell}(w)\neq \emptyset\}$ is nonempty. Set $\ell_0=\min\{\ell:\ i\le\ell\le D \textrm{ and } S_i(v)\cap S_{\ell}(w)\neq \emptyset\}$. So $\ell_0$ is an integer such that  $i\le \ell_0\le  D$, $S_i(v)\cap S_{\ell_0}(w)\neq \emptyset$, and $S_i(v)\cap S_{k}(w)=\emptyset$ for  $i \le  k < \ell_0$. In particular, for $i \le  k < \ell_0$, we have $S_i(v)\not\subseteq S_{k}(w)$, and hence $S_i(v)\not\subseteq S_{k,\ell_0}(w)$ for  $i \le  k < \ell_0$.

Let us prove (a). 

So we assume $d\ge 3$ and we have to demonstrate that there exists a unique integer $j$, $i\le j\le D$, such that the intersection $S_{i}^\star(v) \cap S_{j}^{\star}(w)$ is nonempty. By Proposition~\ref{v2-first}, it is enough to prove that there exists an integer $j_0$, $i\le j_0\le D$, such that the $S_{i}^\star(v) \cap S_{j_0}^{\star}(w)\ne\emptyset$. 

If $i=D$ the result holds by taking $j_0=D$, because in this case, by statement (1) of Proposition~\ref{v2-second}, we have $S_{D}^\star(v) \cap S_{D}^{\star}(w)\ne\emptyset$. Whereas if $i<D$ the result holds by taking  $j_0=\ell_0$. Indeed, since $S_i(v)\not\subseteq S_{k,\ell_0}(w)$ for  $i \le  k < \ell_0$, by applying statement (2) of Proposition~\ref{v2-second},  we have $S_{i}^\star(v) \cap S_{\ell_0}^{\star}(w)\ne\emptyset$. 

This concludes the proof of (a).

Now let us prove (b). 

Assume $d=2$. First, let us prove that if for all integer $j$, $i\le j\le D$,  the intersection $S_{i}^\star(v) \cap S_{j}^{\star}(w)$  is empty, then $G=B(d,D)$ and $v_i=v_{i+1}=\cdots=v_D\ne w_D$. 

Observe that if $i=D$, then the above implication is a direct consequence of the statement (1) of Proposition~\ref{v2-second-cas 2}. Thus we only must prove the implication in the case $i<D$. 

Hence, assume $i<D$. Let us consider the integer $\ell_0$ defined above. By assumption, $S_{i}^\star(v) \cap S_{\ell_0}^{\star}(w)=\emptyset$. Therefore, by statement (2) of Proposition~\ref{v2-second-cas 2}, we conclude that $\ell_0=D$, $v_{[i,D-1]}= v_{[i+1,D]}$, and $S_{i-1,\ell_0}(w)\ne\emptyset$. Since $i<D$ we have $v_{[i,D-1]}= v_{[i+1,D]}$ if and only if $G=B(d,D)$ and $v_i=v_{i+1}=\cdots=v_D$. To conclude it only remains to show that $v_D\ne w_D$. 

Since $\ell_0=D$ and $S_{i-1,\ell_0}(w)\ne\emptyset$, by Remark~\ref{rem_S_kj-bis} we have $S_{i-1}(w)\cap S_k(w)=\emptyset$ for all $i-1 <k< D$. By statement (1) of Lemma~\ref{tech-2}, we have $S_{i-1}(w)\cap S_k(w)=\emptyset$ for all $i-1 <k< D$ if and only if $w_{[i,D-(k-i)-1)]}\ne w_{[k+1,D]}$ for all $i-1 <k< D$; if and only if $v_{[i+1,D-(k-i)-1)]}\ne v_{[k+2,D]}$ for $v_{D-(k-i)}\ne w_D$ for all $i-1 <k< D$. Since $v_i=v_{i+1}=\cdots=v_D$ we have $v_{[i+1,D-(k-i)-1)]}=v_{[k+2,D]}$ for all $i-1 <k< D-1$. Therefore we have $S_{i-1}(w)\cap S_k(w)=\emptyset$ for all $i-1 <k< D$ if and only if $v_{D-(k-i)}\ne w_D$ for all $i-1 <k< D$; if and only if $v_D\ne w_D$. 

Reciprocally, let us demonstrate that if $G=B(d,D)$ and $v_i=v_{i+1}=\cdots=v_D\ne w_D$, then the intersection $S_{i}^\star(v) \cap S_{j}^{\star}(w)$ is empty for all integer $j$, $i\le j\le D$. If $v_i=v_{i+1}=\cdots=v_D\ne w_D$, then $v_{[i+1,D-(j-i)]}\ne w_{[j+1,D]}$ for all $j$, $i\le j<D$. Therefore, by statement(1) of Lemma~\ref{tech-2}, we have $S_i(v)\cap S_j(v)=\emptyset$ for all $j$, $i\le j<D$, and hence $S_{i}^\star(v) \cap S_{j}^{\star}(w)=\emptyset$ for all $j$, $i\le j<D$. It remains to be proved that we also have $S_{i}^\star(v) \cap S_{D}^{\star}(w)=\emptyset$. Indeed, since $S_{i}^\star(v)=S_i(v)\setminus \left (\bigcup_{k=0}^{i-1}S_{k}(v) \right )$ and $S_{D}^{\star}(w)=V\setminus \left (\bigcup_{k=0}^{D-1}S_{k}(w) \right )$, we conclude that
$$
S_{i}^\star(v) \cap S_{D}^{\star}(w)=S_i(v)\setminus \bigcup_{k=0}^{i-1} \left( S_{k}(v)\cup  S_{k}(w)\right )=\emptyset,
$$
because, since $d=2$ and $v_i=v_{i+1}=\cdots=v_D\ne w_D$, we have $S_{i-1}(v)\cup  S_{i-1}(w)=S_{i}(v)$. (This last equality can be easily checked, as in Subsection~\ref{sec:3.5}, by using the sequence representation of the vertices.)

\subsection{Proof of Theorem~\ref{propo 2.9-v2-cas 1}}
Let $v\in V$, $w\in S_1(v)$, and let $1\le i\le D$. Assume that $S_{i}^\star(v) \cap S_{i-1}^{\ast}(w)\ne\emptyset$ and that $S_{i}^\star(v) \cap S_{j_0}^{\ast}(w)\ne\emptyset$ for some $i\le j_0\le D$. Then, from Proposition~\ref{v2-first} we conclude that
$S_i^\star(v)=
\left(S_{i}^\star(v) \cap S_{i-1}^{\star}(w)\right)\cup \left(S_{i}^\star(v) \cap S_{j_0}^{\star}(w)\right),
$
and so, since  $\left(S_{i}^\star(v) \cap S_{i-1}^{\star}(w)\right)\cap \left(S_{i}^\star(v) \cap S_{j_0}^{\star}(w)\right)=\emptyset$, we get
$
\big |S_i^\star(v)\big |=
\big |S_{i}^\star(v) \cap S_{i-1}^{\star}(w)\big |+ \big |S_{i}^\star(v) \cap S_{j_0}^{\star}(w)\big |.
$
To complete the proof of the theorem we must demonstrate that if $\left |S_{i}^\star(v) \right |=d^i-a_{i-1}d^{i-1}-\ldots -a_{1} d-a_{0}$, then $\left |S_{i}^\star(v) \cap S_{i-1}^{\ast}(w) \right |= d^{i-1}-b_{i-2}d^{i-2}-\ldots -b_{1} d-c_{0},$ where $b_k=1$  if and only if $a_k=1$ and $v_{D-i+k+1}=w_D$.  Let us demonstrate this. 

First of all observe that $S_i(v)\cap S_{i-1}(w)\neq \emptyset$, because  $w\in S_1(v)$. So we can apply statement (1) of Lemma~\ref{tech-7} to write  $S_{i}^\star(v)\cap S_{i-1}^{\ast}(w)$ as
\begin{align}
\label{empty-p}
S_{i}^\star(v)\cap S_{i-1}^{\ast}(w) &= S_{i-1}(w) \setminus\bigcup_{k=0}^{i-1}  S_{k,i}(v)= S_{i-1}(w) \setminus\Big(\bigcup_{k=0}^{i-2}  S_{k,i}(v) \cup  \Big(S_{i-1,i}(v)\cap S_{i-1}(w)\Big)\Big).
\end{align}
Let us see that, in the above expression, the intersection $S_{i-1,i}(v)\cap S_{i-1}(w)$ is empty. Since $S_{i}^\star(v) \cap S_{i-1}^{\ast}(w)\ne\emptyset$, we know by statement (3) of Propositions~\ref{v2-second} and \ref{v2-second-cas 2} that either $G=K(d,D)$ or if $G=B(d,D)$, then $v_i=v_{i+1}=\cdots=v_D=w_D$ does not hold. If $G=K(d,D)$, then $S_{i-1,i}=\emptyset$ (see Remark~\ref{rem_def_S_kj}) an so $S_{i-1,i}(v)\cap S_{i-1}(w)=\emptyset$, as we wanted to show. So assume $G=B(d,D)$ and that condition $v_i=v_{i+1}=\cdots=v_D=w_D$ does not hold. Moreover, we can suppose $\emptyset\ne S_{i-1,i}(v)=S_{i-1}(v)$, because otherwise our claim clearly holds. Thus, by Definition~\ref{def_S_kj}, we have $S_{i-1}(v)\subseteq S_i(v)$ and hence, if $i<D$, then $v_{[i,D-1]}=v_{[i+1,D]}$ by statement (1) of Lemma~\ref{tech-2}. Hence in any case we must have $v_D\ne w_D$ and so $v_{[i,D-1]}\ne w_{[i,D]}$. Again by statement (1) of Lemma~\ref{tech-2} we get $S_{i-1}(v)\cap S_{i-1}(w)=\emptyset$, as we wanted to prove. 

Therefore, since $S_{i-1}(v)\cap S_{i-1}(w)=\emptyset$, we get from (\ref{empty-p}) that
$$
S_{i}^\star(v)\cap S_{i-1}^{\ast}(w) =S_{i-1}(w)  \setminus
\bigg(
\bigcup_{k=0}^{i-2} \big(S_{k,i}(v)\cap S_{i-1}(w)\big)
\bigg),
$$
where, by statement (1) of Lemma~\ref{tech-2}, for $0\le k\le i-2$, we have either $S_{k,i}(v)\cap S_{i-1}(w)=\emptyset$ or $S_{k,i}(v)\subseteq S_{i-1}(w)$. Thus, recalling that if $k_1\ne k_2$, then $S_{k_1,i}(v) \cap S_{k_2,i}(v)=\emptyset$ (see Remark~\ref{void}), we have
\begin{align*}
|S_{i}^\star(v)\cap S_{i-1}^{\ast}(w)| &=
|S_{i-1}(v)|-\sum _{k=0}^{i-2} |S_{k,i}(v)\cap S_{i-1}(w)|=d^{i-1}-b_{i-2}d^{i-2}-\ldots -b_{1} d-b_{0},
\end{align*}
where the coefficients $b_{k}$ are $0$ or $1$. More precisely, for $0\le k\le i-2$, we have $b_k=1$ if and only if $S_{k,i}(v)\cap S_{i-1}(w)\ne\emptyset$. That is, $b_k=1$ if and only if $S_{k,i}(v)\ne\emptyset$ and $S_{k,i}(v)\cap S_{i-1}(w)\ne\emptyset$. Now, by applying Proposition~\ref{propo 2.3} and statement (1) of Lemma~\ref{tech-2} we have $b_k=1$ if and only if $a_k=1$  and $S_{k}(v)\subseteq S_{i-1}(w)$; if and only if  $a_k=1$ and $v_{[k+1,D-(i-k)+1]}=w_{[i,D]}$. To finish the proof  let us demonstrate that we have $a_k=1$  and $v_{[k+1,D-(i-k)+1]}=w_{[i,D]}$ if and only if $a_k=1$ and $v_{D-i+k+1}=w_D$. Clearly, we only must show that if $a_k=1$ and $v_{D-i+k+1}=w_D$, then $v_{[k+1,D-(i-k)+1]}=w_{[i,D]}$. If $i=D$, there is nothing to prove. So let us prove the implication in the case $i<D$. Hence assume $i<D$, $a_k=1$, and $v_{D-i+k+1}=w_D$. By Proposition~\ref{propo 2.3} and the definition of $S_{k,i}(v)$, if  $a_k=1$, then $S_{k,i}(v)=S_k(v)\subseteq S_i(v)$. Hence, again by statement (1) of Lemma~\ref{tech-2} and since $w\in S_1(v)$, we have $v_{[k+1,D-(i-k)]}=v_{[i+1,D]}=w_{[i,D-1]}$. Therefore the equality $v_{[k+1,D-(i-k)+1]}=w_{[i,D]}$ holds, because we are assuming $v_{D-i+k+1}=w_D$. This finishes the proof of the theorem.

\subsection{Proof of Theorem~\ref{propo 2.9-v2-cas 2}}
Let $v\in V$ and $w\in S_1(v)$. Let $1\le i\le D$ and assume that $S_{i}^\star(v) \cap S_{i-1}^{\ast}(w)=\emptyset$. 

If $S_{i}^\star(v) \cap S_{i-1}^{\ast}(w)=\emptyset$, then, by statement (3) of Propositions~\ref{v2-second} and \ref{v2-second-cas 2}, we have $S_{i}^\star(v)\cap S_{i}^{\ast}(w)\ne\emptyset$. Hence the unique integer $j_0$ given in Proposition~\ref{v2-first} is $j_0=i$. Therefore, by statement (1) of this Proposition~\ref{v2-first}, we have $S_{i}^\star(v)=S_{i}^\star(v)\cap S_{i}^{\ast}(w)$, and so, if $\left |S_{i}^\star(v) \right |=d^i-a_{i-1}d^{i-1}-\ldots -a_{1} d-a_{0}$ is the polynomial expression of  $\left |S_{i}^\star(v) \right |$, then 
$$
\big |S_{i}^\star(v) \cap S_{i}^{\ast}(w) \big |=\big |S_{i}^\star(v) \big |=d^i-a_{i-1}d^{i-1}-\ldots -a_{1} d-a_{0}.$$

\subsection{Proof of Theorem~\ref{propo 2.9-v2-cas 3}}
\label{proof propo 2.9-v2-cas 3}
Let $v\in V$ and $w\in S_1(v)$. Let $1\le i\le D$ and assume that $S_{i}^\star(v) \cap S_{j}^{\star}(w)=\emptyset$ for all $i\le j\le D$.

First of all notice that we must have $d=2$, because if $d\ge 3$, then, by statement (4) of Proposition~\ref{v2-second}, there exists a unique integer $j$, $i\le j\le D$, such that the intersection $S_{i}^\star(v) \cap S_{j}^{\star}(w)$  is non-empty, contradicting the assumption that $S_{i}^\star(v) \cap S_{j}^{\star}(w)=\emptyset$ for all $i\le j\le D$.

If $S_{i}^\star(v) \cap S_{j}^{\star}(w)=\emptyset$ for all $i\le j\le D$, then, by statement (2) of Proposition~\ref{v2-first}, we have $S_i^\star(v)= S_{i}^\star(v) \cap S_{i-1}^{\star}(w)$ and so, if $\left |S_{i}^\star(v) \right |=d^i-a_{i-1}d^{i-1}-\ldots -a_{1} d-a_{0}$ is the polynomial expression of  $\left |S_{i}^\star(v) \right |$, then 
\begin{equation}
\label{eq:0}
\big |S_{i}^\star(v) \cap S_{i-1}^{\ast}(w) \big |=\big |S_{i}^\star(v) \big |=d^i-a_{i-1}d^{i-1}-\ldots -a_{1} d-a_{0}.
\end{equation}

It remains to prove that, in this case, we have $a_{i-1}=1$. On one hand, by Proposition~\ref{propo 2.3} we know that $a_{i-1}=1$ if and only if $G=B(d,D)$ and $v_i=v_{i+1}=\cdots=v_D$. On the other hand, by statement (4) of Proposition~\ref{v2-second-cas 2}, if $S_{i}^\star(v) \cap S_{j}^{\star}(w)=\emptyset$ for all $i\le j\le D$, then  $G=B(d,D)$ and $v_i=v_{i+1}=\cdots=v_D\ne w_D$. Therefore $a_{i-1}=1$.

Finally notice that, since $d=2$ and $a_{i-1}=1$, we have $d^i-d^{i-1}=d^{i-1}$. Then the polynomial expression $d^i-d^{i-1}-a_{i-2}d^{i-2}\ldots -a_{1} d-a_{0}$ in (\ref{eq:0}) can be equivalently expressed as $d^{i-1}-a_{i-2}d^{i-2}\ldots -a_{1} d-a_{0}$.

\setcounter{subsection}{0}
\section{Proofs of our results on deflection routing}
\label{proofs-2}
In this section, we prove the results presented in Subsection~\ref{application}. To do this, we will use several additional technical lemmas on the layer structure gathered in Subsection~\ref{technical-lemmas-defl}.

\subsection{Technical lemmas}
\label{technical-lemmas-defl}

The following two lemmas deal with the set of vertices $w\in S_1(v)$ for which the intersection set $S_i(v)\cap S_j(w)$ is nonempty.

\begin{lemma}
\label{tech-3}
Let $v\in V$ and, for $0\leq i\leq j<D$, let $\Gamma_{i,j}^+(v)=\{w\in S_1(v): \, \, S_i(v)\cap S_j(w)\ne\emptyset\}$. Then
\begin{enumerate}
\item The set $\Gamma_{i,j}^+(v)$ is nonempty if and only if one the following conditions is fulfilled:
\begin{enumerate} 
\item $i<j<D-1$ and $v_{[i+1,D+i-j-1]}=v_{[j+2,D]}$.
\item $i<j=D-1$ and either $G=B(d,D)$, or $G=K(d,D)$ and $v_{i+1}\ne v_D$. 
\item $i=j$, $G=B(d,D)$ and $v_{[i+1,D]}=v_D\cdots v_D$.
\end{enumerate}
\item If $\Gamma_{i,j}^+(v)\ne\emptyset$, then $\Gamma_{i,j}^+(v)$ has a unique element $w$ which sequence representation is $w=v_2\cdots v_{D}v_{i+(D-j)}$. Moreover, if $i=j$, then $S_{i-1}(w)\subseteq S_i(v)=S_i(w)\subseteq S_{i+1}(v)=S_{i+1}(w)\subseteq\cdots$
\item Let $\Gamma_{i,j}^\star(v)=\{w\in S_1(v):\, S_i^\star(v)\cap S_j^\star(w)\ne\emptyset\}$. Then $\Gamma_{i,j}^\star(v)\ne\emptyset$ if and only if there exists a vertex $w$ such that  $\Gamma_{i,j}^\star(v)=\{w\}$; if and only if there exists a vertex $w$ such that $\Gamma_{i,j}^\star(v)=\Gamma_{i,j}^+(v)=\{w\}$. Moreover, $w\in\Gamma_{i,j}^\star(v)$ if and only if $w\in\Gamma_{i,j}^+(v)$ and $S_i\left(v\right)\not\subseteq S_{t,j}\left(w\right)$ for $i\le t < j$.
\end{enumerate}
\end{lemma}

\begin{remark}
In statement (1) we provide some necessary and sufficient conditions (stated in terms of the sequence representation of $v$) for $\Gamma_{i,j}^+\left(v\right)\ne\emptyset$, and  in statement (2) it is proved that if $\Gamma_{i,j}^+\left(v\right)\ne\emptyset$, then $\Gamma_{i,j}^+(v)=\{w\}$, where $w$ is  determined from $v$. However, the intersection $S_i^\star\left(v\right)\cap S_{j}^\star\left(w\right)$ can be empty in spite of $v$ being a vertex for which $\Gamma_{i,j}^+\left(v\right)\ne\emptyset$. 
In Example~\ref{def-Kautz} we have an example of this fact. There we considered the digraph $G=K(d,12)$, $v=\alpha\beta\gamma\alpha\beta\gamma\alpha\beta\gamma\alpha\beta\gamma$, $i=1$, and $j=6$. Observe that vertex $v$ satisfies condition (a) in statement (1) of Lemma~\ref{tech-3}, because $v_{[2,6]}=\beta\gamma\alpha\beta\gamma=v_{[8,12]}$. Therefore $\Gamma_{i,j}^+\left(v\right)=\{w=\beta\gamma\alpha\beta\gamma\alpha\beta\gamma\alpha\beta\gamma\alpha\}\ne\emptyset$, where, by statement (2) of  Lemma~\ref{tech-3}, the last symbol in the sequence representation of $w$ is $v_{i+(D-j)}=v_7=\alpha$. However, in this example we have $S_1\left(v\right)\subseteq S_{3,6}\left(w\right)=S_{3}\left(w\right)$ and so, by statement (3) of  Lemma~\ref{tech-3}, the intersection $S_1^\star(v)\cap S_6^\star(w)$ is empty. Indeed, as seen in Example~\ref{def-Kautz}, we have $S_1^\star(v)\cap S_6^\star(w)= (S_1(v)\cap S_6(w))\setminus S_3(w)=\emptyset$. 
\end{remark}

\begin{proof} 
Let $w\in S_1(v)$. Hence $w_{[1,D-1]}=v_{[2,D]}$. Assume first that $j< D-1$. By statement (1) of Lemma~\ref{tech-2}, $S_i(v)\cap S_j(w)\ne\emptyset$  if and only if  $v_{[i+1,D-(j-i)]}=w_{[j+1,D]}$; that is, $S_i(v)\cap S_j(w)\ne\emptyset$ if and only if $v_{[i+1,D-(j-i)-1]}=v_{[j+2,D]}$ and $v_{D-(j-i)}=w_D$. In particular, if $i=j$, then $v_{i+1}=v_{i+2}=\cdots=v_D=w_D$, and hence  $G=B(d,D)$.  

Now suppose $j=D-1$. By Lemma~\ref{tech-2}, $S_i(v)\cap S_{D-1}(w)\ne\emptyset$ if and only if $v_{i+1}=w_D$. Therefore, if $G=B(d,D)$ there is always a vertex $w\in\Gamma_{i,j}^+(v)$, while if $G=K(d,D)$, then there exists $w\in\Gamma_{i,j}^+(v)$ if and only if $v_{i+1}\ne v_D$. In particular, if $G=K(d,D)$ and $\Gamma_{i,j}^+(v)\ne\emptyset$, then $D\ne i+1$. 

Until now we have proved statement (1). Next, to prove statement (2) first let us  assume that $\Gamma_{i,j}^+(v)\ne\emptyset$. Observe from the above that if $w\in \Gamma_{i,j}^+(v)$, then $w_D$ is uniquely determined and it is equal to $v_{i+(D-j)}$ both for $j<D-1$ as for $j=D-1$. Hence $\Gamma_{i,j}^+(v)$ has a unique element $w$ which sequence representation is $w=v_2\cdots v_{D}v_{i+(D-j)}$. It is clear from the previous statements that if $\Gamma_{i,j}^+(v)\ne\emptyset$ and $i=j$, then $G$ cannot be $K(d,D)$. Since $w\in S_1(v)$ we have $S_l(w)\subseteq S_{l+1}(v)$ for all $l\ge 0$. Hence to finish the proof of statement (2) we only need to show that if $i=j$, then $S_k(v)=S_k(w)$ for all $k\ge i$. But this is clear because, from the above discussion,  if $i=j$ and $\Gamma_{i,j}^+(v)\ne\emptyset$, then there exists $\alpha$ such that the sequence representations of $v$ and $w$ are of the form $v=v_{1}\cdots v_i\alpha\cdots\alpha$, $w=w_{1}\cdots w_i\alpha\cdots\alpha$. 

Finally, let us demonstrate statement (3). Since $S_i^\star(v)\cap S_j^\star(w)\subseteq S_i(v)\cap S_j(w)$ we have $\Gamma_{i,j}^\star(v)\subseteq \Gamma_{i,j}^+(v)$. Let us assume $\Gamma_{i,j}^+(v)\ne\emptyset$, in which case $\Gamma_{i,j}^+(v)=\{w\}$, where $w$ is the unique element of $\Gamma_{i,j}^+(v)$ given in statement (2). Therefore we have $\Gamma_{i,j}^\star(v)\ne\emptyset$ if and only if $\Gamma_{i,j}^\star(v)=\Gamma_{i,j}^+(v)=\{w\}$; if and only if  $S_i^\star(v)\cap S_j^\star(w)\ne\emptyset$. Finally, since $S_i\left(v\right)\cap S_{j}\left(w\right)\ne\emptyset$ if $S_i^\star(v)\cap S_j^\star(w)\ne\emptyset$, we conclude from statement (2) of Propositions~\ref{v2-second} and \ref{v2-second-cas 2} that $S_i^\star(v)\cap S_j^\star(w)\ne\emptyset$ if and only if $w\in\Gamma_{i,j}^+(v)$ and $S_i\left(v\right)\not\subseteq S_{t,j}\left(w\right)$ for $i\le t < j$; that is, we have $w\in\Gamma_{i,j}^\star(v)$ if and only if $w\in\Gamma_{i,j}^+(v)$ and $S_i\left(v\right)\not\subseteq S_{t,j}\left(w\right)$ for $i\le t < j$.
\end{proof}

\begin{lemma}
\label{tech-3-bis}
Let $v\in V$ and, for $0\leq i< D$, let $\Gamma_{i,D}^+(v)=\{w\in S_1(v): \, \, S_i(v)\cap S_D(w)\ne\emptyset\}$. The following statements hold:
\begin{enumerate}
\item If $G=B(d,D)$, or $G=K(d,D)$ and $v_{i+1}=v_D$, then $\Gamma_{i,D}^+(v)=S_1(v)$.
\item If $G=K(d,D)$ and $v_{i+1}\ne v_D$, then $w\in \Gamma_{i,D}^+(v)$ if and only if $w\in S_1(v)$ and $v_{i+1}\ne w_{D}$.
Moreover, $|\Gamma_{i,D}^+(v)|=d-1$.
\end{enumerate}
\end{lemma}

\begin{proof} 
If $G=B(d,D)$, then $S_D(w)=V$ and hence $S_i(v)\cap S_D(w)=S_i(v)\ne\emptyset$ for all $w\in S_1(v)$. Therefore if $G=B(d,D)$, then $\Gamma_{i,D}^+(v)=S_1(v)$. Now let $G=K(d,D)$. If $w\in S_1(v)$, then $w_{[1,D-1]}=v_{[2,D]}$ and, in particular, $w_{D-1}=v_D$. Thus we have $w_D\ne v_D$, because two consecutive symbols in the sequence representation of the vertices of $K(d,D)$ are different. By statement (2) of Lemma~\ref{tech-2}, we have $S_i(v)\cap S_D(w)\ne\emptyset$  if and only if $v_{i+1}\ne w_{D}$. Therefore if $v_{i+1}=v_D$, then $v_{i+1}\ne w_{D}$ holds for any $w\in S_1(v)$. We conclude that if $G=K(d,D)$ and $v_{i+1}=v_D$, then $\Gamma_{i,D}^+(v)=S_1(v)$. This completes the proof of statement (1). 

Let us demonstrate statement (2). So assume $G=K(d,D)$ and $v_{i+1}\ne v_D$. By the previous considerations, we have $w\in \Gamma_{i,D}^+(v)$ if and only if $w\in S_1(v)$ and $v_{i+1}\ne w_{D}$; if and only if $w\in S_1(v)$ and $w_D\not\in \{v_{i+1}, v_D\}$. Then, since the symbol alphabet has $d+1\ge 3$ symbols, we have $\Gamma_{i,D}^+(v)\ne\emptyset$ and, moreover, $|\Gamma_{i,D}^+(v)|=d-1$. This completes the proof of the lemma.
\end{proof}




We finish this subsection of technical lemmas with the following one dealing with the partition $\{\mathcal{V}_1,\ldots, \mathcal{V}_l\}$ of the vertex set $V$ introduced in Subsection~\ref{application}. 
Notice that if $\sigma$ is a permutation of the symbol alphabet and $\sigma(v)$ is the vertex whose sequence representation is $\sigma(v_1) \sigma(v_2)\ldots \sigma(v_D)$, then $v$ and $\sigma(v)$ belong to a same vertex class $\mathcal{V}_r$ (that is to say, the sequence representations of $v$ and $\sigma(v)$ have an equivalent structure). The proof of the lemma is an immediate consequence of the definitions and of the fact that $\sigma$ is a bijection.

\begin{lemma}
\label{tech-11}
Let $\sigma$ be a permutation of the symbol alphabet and, given $v\in V$, let $\sigma(v)=\sigma(v_1) \sigma(v_2)\ldots \sigma(v_D)$. Then the following statements hold:
\begin{enumerate}
\item $|S_{i}^\star(v)|=|S_{i}^\star(\sigma(v))|$. 
\item  If $w\in S_1(v)$, then $|S_{i}^\star(v)\cap S_{j}^\star(w)|=|S_{i}^\star(\sigma(v)) \cap S_{j}^\star(\sigma(w))|$. 
\end{enumerate}
\end{lemma}

\subsection{Proof of Theorem~\ref{p_in}}
\label{proof p_in}
Let $v$ be a vertex selected uniformly at random from the vertex set $V$. By definition, the input probability $\P_{\mathsf{in}}(i)$ is the probability of selecting uniformly at random from $V\setminus\{v\}$ a vertex $v'$ which is at distance $i$ from $v$. For a fixed $v$, the probability of selecting such a $v'$ is clearly $\P_{\mathsf{in}}(i\mid v)=|S_{i}^\star\left(v\right)|/{(|V|-1)}$. Moreover, by Lemma~\ref{tech-11}, this probability is the same for any vertex $v\in \mathcal{V}_r$ in a same vertex class $\mathcal{V}_r$, $1\le r\le l$. Moreover, since $v$ is chosen uniformly at random from $V$, we have $\P(v\in \mathcal{V}_r)= {|\mathcal{V}_r|}/{|V|}$. Thus, for any choice of $v^{(r)}\in\mathcal{V}_r$, the input probability $\P_{\mathsf{in}}(i)$ can be expressed as 
\begin{equation*}
	\P_{\mathsf{in}}(i)=\sum_r \P_{\mathsf{in}}\left(i\mid v^{(r)}\in \mathcal{V}_r\right)\, \P\left(v^{(r)}\in \mathcal{V}_r\right)
	= \sum_{r=1}^l \frac{\left|S_{i}^\star\left(v^{(r)}\right)\right|}{(|V|-1)}\cdot\frac{|\mathcal{V}_r|}{|V|}.
\end{equation*}
The proof is completed by using  the polynomial description of $\left|S_{i}^\star\left(v^{(r)}\right)\right|$ given in Proposition~\ref{propo 2.3} and Remark~\ref{rem_S_kj-seq}. Finally, notice that $\P_{\mathsf{\mathsf{in}}}(i)=\Theta\left(1/d^{D-i}\right)$, because $|S_{i}^\star\left(v\right)|=\Theta\left(d^i\right)$ (independently of $v$) and $|V|=\Theta\left(d^D\right)$.

\subsection{Proof of Theorem~\ref{p_ij}}
\label{proof p_ij}
Let $v$ be the vertex at which deflection occurs and suppose that the destination vertex $z$ is at distance $i$ from $v$.  Let $w\in S_1(v)$ be the vertex through which deflection takes place. In other words, we are supposing that a packet circulating within the network (which has to arrive to  $z$) is currently in $v$ and cannot proceed through the shortest path from $v$ to $z$;  and hence it is deflected to vertex $w$. 

Hence the probability that the new distance from $w$ to the destination vertex $z$ is $j$, given that a deflection occurs in $v$ and that the deflection takes place through $w$, is just the probability that, conditional on the event $z\in S_i^\star\left(v\right)$,  the destination vertex $z$ belongs to $S_i^\star\left(v\right)\cap S_{j}^\star\left(w\right)$. In this way, denoting this conditional probability as $\P_{\mathsf{t}}\left(i,j \mid v, w\right)$, we have
\begin{equation*}
\label{defl-1}
\P_{\mathsf{t}}\left(i,j\mid v, w\right)=
\frac{\left|S_i^\star\left(v\right)\cap S_{j}^\star\left(w\right)\right|}{\left|S_i^\star\left(v\right)\right|}.
\end{equation*}
It follows that $\P_{\mathsf{t}}\left(i,j \mid v, w\right)\ne 0$  if and only if $S^\star_i\left(v\right)\cap S^\star_j(w)\ne\emptyset$. 

Let $w'_{v,z}$ be the vertex adjacent from $v$ in the unique shortest path from $v$ to $z$. The vertex $w$ through which deflection takes place is selected uniformly at random from $S_1(v)\setminus \{w'_{v,z}\}$. Hence the probability $\P\left(w\mid v\right)$ that, given that deflection occurs, it takes place through  $w\in S_1(v)$ can be calculated as
$$
\P(w\mid v)=\sum_{z\in S_i^\star(v)} \P(w\mid v, z)\P(z\mid v),
$$
where $\P(w\mid v, z)=0$ if $w=w'_{v,z}$ and $\P(w\mid v, z)=1/(d-1)$ if $w\ne w'_{v,z}$. Moreover, $w=w'_{v,z}$ if and only if $z\in S_i^\star(v)\cap S^\star_{i-1}(w)$.
Therefore, $\P(w\mid v, z)=0$ if and only if $z\in S_i^\star(v)\cap S^\star_{i-1}(w)$. Furthermore, the probability that the destination vertex is a given vertex $z$ belonging to $S_i^\star(v)$ is simply
$$
\P(z\mid v)=\frac{1}{|S_i^\star(v)|}.
$$
Then, since $|S_i^\star(v)|-|S_i^\star(v)\cap S_{i-1}^\star(w)|$ is the number of vertices $z\in S_i^\star(v)$ for which $w\ne w'_{v,z}$, we have
\begin{align*}
\P(w\mid v) &=\frac{1}{|S_i^\star(v)|}\sum_{z\in S_i^\star(v)} \P(w\mid v, z)\\
            &=\frac{1}{|S_i^\star(v)|}\frac{|S_i^\star(v)|-|S_i^\star(v)\cap S_{i-1}^\star(w)|}{d-1}=\frac{1}{d-1}\left(1-\frac{|S_i^\star(v)\cap S_{i-1}^\star(w)|}{|S_i^\star(v)|}\right).
\end{align*}

Let $\Gamma_{i,j}^\star(v)=\{w\in S_1(v): \, S_i^\star(v)\cap S_j^\star(w)\ne\emptyset\}$ be the set defined in statement (3) of Lemma~\ref{tech-3}. Clearly, we have $\P_{\mathsf{t}}\left(i,j \mid v, w\right)\ne 0$ if and only if $w\in\Gamma_{i,j}^\star(v)$. Moreover, it is proved in Lemma~\ref{tech-3} that if $v=v_1v_2\cdots v_D$ and $\Gamma_{i,j}^\star(v)\ne\emptyset$, then $\Gamma_{i,j}^\star(v)$ contains a  single vertex $w_{v}$ which sequence representation is uniquely determined from $v$, $i$ and $j$, namely 
\begin{equation}
\label{determined}
w_v=v_2\cdots v_{D}v_{i+(D-j)}.
\end{equation}

Taking all these considerations into account we can express the transition probability that the new distance to the destination is $j$, conditional on the event that deflection occurs at $v$, as
\begin{align}
\label{cind}
\nonumber 
\P_{\mathsf{t}}\left(i, j \mid v\right)&=\sum_{w\in \Gamma_{i,j}^\star\left(v\right)} \P_{\mathsf{t}}\left(i, j \mid  v, w\right)\cdot \P\left(w\mid v \right)=\P_{\mathsf{t}}\left(i, j \mid v, w_v\right)\P(w_v\mid v)\\
&=\frac{1}{d-1}\cdot\frac{\left|S_i^\star\left(v\right)\cap S_{j}^\star\left(w_v\right)\right|}{\left|S_i^\star\left(v\right)\right|}\left(1-\frac{|S_i^\star(v)\cap S_{i-1}^\star(w_v)|}{|S_i^\star(v)|}\right),
\end{align}
if $\Gamma_{i,j}^\star\left(v\right)\ne\emptyset$; and $\P_{\mathsf{t}}\left(i, j \mid v\right)=0$ otherwise. 

Furthermore, if $\sigma$ is a permutation of the symbol alphabet $A$, then, using the notation introduced in Lemma~\ref{tech-11}, we can check that $\Gamma_{i,j}^\star\left(\sigma(v)\right)\ne\emptyset$ if and only if $\Gamma_{i,j}^\star\left(v\right)\ne\emptyset$, and that if $\Gamma_{i,j}^\star\left(\sigma(v)\right)\ne\emptyset$, then $\Gamma_{i,j}^\star\left(\sigma(v)\right)=\{\sigma\left(w_v\right)\}$. Moreover, $\sigma(w'_{v,z})=w'_{\sigma(v),\sigma(z)}$ is the vertex adjacent from $\sigma(v)$ in the shortest path to $\sigma(z)$. This facts, together with the statements of Lemma~\ref{tech-11}, imply that the probability calculated in (\ref{cind}) is the same for any vertex $v^{(r)}$ in a given vertex class $\mathcal{V}_r$. (Recall that $\mathcal{V}_r$ is the class of vertices to which $v^{(r)}$ belongs according to the structure of its sequence representation.) 

Now, by adding for all the classes $\mathcal{V}_r$ and taking into account that $\P\left(v^{(r)}\in \mathcal{V}_r\right)={|\mathcal{V}_r|}/{|V|}$ we obtain the transition probability $\P_{\mathsf{t}}(i,j)$ that, conditional on the event that the deflection occurs in a vertex which is at distance $i$ to the destination vertex, the new distance to this destination is $j$. In this way, by setting  $w^{(r)}=w_{v^{(r)}}$  we have
\begin{eqnarray*}
\lefteqn{\P_{\mathsf{t}}(i,j) =\sum_{r} \P_{\mathsf{t}}\left(i,j \mid v^{(r)}\in \mathcal{V}_r\right)\, \P\left(v^{(r)}\in \mathcal{V}_r\right)}\\
&&=\frac{1}{(d-1) |V|}\sum_{r} |\mathcal{V}_r|\frac{\left|S_i^\star\left(v^{(r)}\right)\cap S_{j}^\star\left(w^{(r)}\right)\right|}{\left|S_i^\star\left(v^{(r)}\right)\right|}\left(1-\frac{|S_i^\star\left(v^{(r)}\right)\cap S_{i-1}^\star\left(w^{(r)}\right)|}{|S_i^\star\left(v^{(r)}\right)|}\right).
\end{eqnarray*}
Furthermore, if the intersection $S_i^\star\left(v^{(r)}\right)\cap S_{j}^\star\left(w^{(r)}\right)$ is nonempty, we conclude from Theorems~\ref{propo 2.9-v2-cas 1} and \ref{propo 2.9-v2-cas 2} that  $\left|S_i^\star\left(v^{(r)}\right)\cap S_{j}^\star\left(w^{(r)}\right)\right|$ has a polynomial expression given by $d^i-\alpha^{(r,i)}_{i-1}d^{i-1}-\cdots -\alpha^{(r,i)}_{1} d-\alpha^{(r,i)}_{0}$. Moreover, by Theorem~\ref{propo 2.9-v2-cas 1}, if $S_i^\star\left(v^{(r)}\right)\cap S_{i-1}^\star\left(w^{(r)}\right)\ne\emptyset$, then $\left|S_i^\star\left(v^{(r)}\right)\cap S_{i-1}^\star\left(w^{(r)}\right)\right|$ has also a polynomial expression of the form $d^{i-1}-b^{(r,i)}_{i-2}d^{i-2}-\ldots -b^{(r,i)}_{1} d-b^{(r,i)}_{0}$. Therefore, by taking also into account Proposition~\ref{propo 2.3}, we have
\begin{equation*}
\P_{\mathsf{t}}(i,j) = \frac{1}{(d-1)|V|}\sum_{r} |\mathcal{V}_r|\, p^{(r,i,j)}\left(1-q^{(r,i)}\right),
\end{equation*}
where $p^{(r,i,j)}$ and $q^{(r,i)}$ are rational fractions in  $d$ of the form
$$
p^{(r,i,j)}=k^{(r,i,j)}\cdot \frac{\, d^i-\alpha^{(r,i)}_{i-1}d^{i-1}-\cdots -\alpha^{(r,i)}_{1} d-\alpha^{(r,i)}_{0}  }{d^i-a^{(r,i)}_{i-1}d^{i-1}-\cdots -a^{(r,i)}_{1} d-a^{(r,i)}_{0}}
$$
and
$$
q^{(r,i)}=\kappa^{(r,i)}\cdot \frac{d^{i-1}-b^{(r,i)}_{i-2}d^{i-2}-\ldots -b^{(r,i)}_{1} d-b^{(r,i)}_{0}}{d^i-a^{(r,i)}_{i-1}d^{i-1}-\cdots -a^{(r,i)}_{1} d-a^{(r,i)}_{0}},
$$
and $k^{(r,i,j)}, \kappa^{(r,i)}\in\{0,1\}$. Furthermore, we have  $k^{(r,i,j)}=1$ if and only if $S_i^\star\left(v^{(r)}\right)\cap S_{j}^\star\left(w^{(r)}\right)\ne\emptyset$, as determined by statement (2) of Propositions~\ref{v2-second} and \ref{v2-second-cas 2}; and we have $\kappa^{(r,i)}=1$ if and only if $S_i^\star\left(v^{(r)}\right)\cap S_{i-1}^\star\left(w^{(r)}\right)\ne\emptyset$, as determined by statement (3) of Propositions~\ref{v2-second} and \ref{v2-second-cas 2}. 

Finally, observe that the coefficients $a^{(r,i)}_{k}$ are determined from $v^{(r)}$ by using Proposition~\ref{propo 2.3}, and the coefficients   $\alpha^{(r,i)}_{k},  b^{(r,i)}_{k}\in\{0,1\}$ are determined from $v^{(r)}$ and $w^{(r)}$ by using Theorems~\ref{propo 2.9-v2-cas 1} and \ref{propo 2.9-v2-cas 2}. So we conclude that:
\begin{enumerate}
\item  $a^{(r,i)}_{k}\in\{0,1\}$; 
\item  $\alpha^{(r,i)}_{i-1}\in\{0,1,2\}$, and $\alpha^{(r,i)}_{k},  b^{(r,i)}_{k}\in\{0,1\}$ for $0\le k\le i-2$.
\end{enumerate}
This completes the proof of Theorem~\ref{p_ij}.


\subsection{Proof of Theorem~\ref{p_iD}}
\label{proof p_iD}
{
We use the same notation and an analysis similar to that in the proof of Theorem~\ref{p_ij}. The probability that the new distance from $w$ to the destination vertex is $D$, given that a deflection occurs in $v$ (which is at distance $i$ to the destination) and that the deflection takes place through $w\in S_1(v)$ is 
\begin{equation*}
\label{defl-iD}
\P_{\mathsf{t}}\left(i,D\mid v, w\right)=
\frac{\left|S_i^\star\left(v\right)\cap S_{D}^\star\left(w\right)\right|}{\left|S_i^\star\left(v\right)\right|}.
\end{equation*}
Let $\Gamma_{i,D}^+(v)=\{w\in S_1(v) \, : \, \, S_i(v)\cap S_D(w)\ne\emptyset\}$ be the set defined in Lemma~\ref{tech-3-bis}. Clearly, if $w\in S_1\left(v\right)\setminus\Gamma_{i,D}^+\left(v\right)$, then for such a vertex $w$ we have $\P_{\mathsf{t}}\left(i,D \mid v, w\right)=0$. In Lemma~\ref{tech-3-bis} it is proved that $\Gamma_{i,D}^+\left(v\right)$ is always nonempty. Indeed, if $G=B(d,D)$, or $G=K(d,D)$ and $v_{i+1}=v_D$, then $\Gamma_{i,D}^+(v)=S_1(v)$; whereas if $G=K(d,D)$ and $v_{i+1}\ne v_D$, then $\Gamma_{i,D}^+(v)=\{w\in S_1(v) \, :\,  w=v_2\cdots v_D w_D,\ w_D\ne v_{i+1}, v_D\}$, and hence $|\Gamma_{i,D}^+(v)|=d-1$. Therefore, the transition probability that the new distance to the destination is $D$, given the event that deflection occurs at $v$, can be expressed as in (\ref{cind}); that is,
\begin{align}
\label{cind-2}
\nonumber 
\P_{\mathsf{t}}\left(i, D \mid v\right)&=\sum_{w\in \Gamma_{i,D}^+(v)} \P_{\mathsf{t}}\left(i, D \mid  v, w\right)\cdot \P\left(w\mid v \right)\\
&=\frac{1}{d-1}\sum_{s=1}^m \frac{\left|S_i^\star\left(v\right)\cap S_{D}^\star\left(w_{v,s}\right)\right|}{\left|S_i^\star\left(v\right)\right|}\left(1-\frac{|S_i^\star(v)\cap S_{i-1}^\star(w_{v,s})|}{|S_i^\star(v)|}\right),
\end{align}
where $m=d-1$ if $G=K(d,D)$ and $v_{i+1}\ne v_D$ and $m=d$ otherwise; and $w_{v,s}$, $1\le s\le m$, are the vertices belonging to $\Gamma_{i,D}^+\left(v\right)$. 

Furthermore, if $\sigma$ is a permutation of the symbol alphabet $A$, then the elements of $\Gamma_{i,j}^+\left(\sigma(v)\right)$ are $w_{\sigma(v),s}=\sigma\left(w_{v,s}\right)$, $1\le s\le m$. Hence, by taking into account Lemma~\ref{tech-11}, we conclude that the probability (\ref{cind-2}) is the same for any vertex $v^{(r)}$ in a given vertex class $\mathcal{V}_r$. By adding for all the classes $\mathcal{V}_r$ and taking into account that $\P\left(v^{(r)}\in \mathcal{V}_r\right)={|\mathcal{V}_r|}/{|V|}$ we obtain the transition probability $\P_{\mathsf{t}}(i,D)$ that, conditional on the event that the deflection occurs in a vertex which is at distance $i$ to the destination vertex, the new distance to this destination is $D$. In this way, by setting  $w^{(r,s)}=w_{v^{(r)},s}$, $1\le s\le m$, we have
\begin{align*}
\P_{\mathsf{t}}(i,D) &=\sum_{r} \P_{\mathsf{t}}\left(i,D \mid v^{(r)}\in \mathcal{V}_r\right)\, \P\left(v^{(r)}\in \mathcal{V}_r\right)\\
&=\frac{1}{(d-1) |V|}\sum_{r}\sum_{s=1}^m |\mathcal{V}_r|\,\frac{\left|S_i^\star\left(v^{(r)}\right)\cap S_{D}^\star\left(w^{(r,s)}\right)\right|}{\left|S_i^\star\left(v^{(r)}\right)\right|}\cdot\left(1-\frac{|S_i^\star\left(v^{(r)}\right)\cap S_{i-1}^\star\left(w^{(r,s)}\right)|}{|S_i^\star\left(v^{(r)}\right)|}\right),
\end{align*}
Moreover, by considering the polynomial expressions of  $\left|S_i^\star\left(v^{(r)}\right)\cap S_{D}^\star\left(w^{(r,s)}\right)\right|$ and $\left|S_i^\star\left(v^{(r)}\right)\cap S_{i-1}^\star\left(w^{(r,s)}\right)\right|$, we have 
\begin{equation*}
\P_{\mathsf{t}}(i,D) = \frac{1}{(d-1) |V|}\sum_{r}\sum_{s=1}^{m_r}|\mathcal{V}_r|\, p^{(r,s,i)}\left(1-q^{(r,s,i)}\right),
\end{equation*}}
where $m_r\in\{d-1,d\}$ if $G=K(d,D)$ and $m_r=d$ if $G=B(d,D)$, and where $p^{(r,s,i)}$ and $q^{(r,s,i)}$ are rational fractions in the degree $d$ of the form
$$
p^{(r,s,i)}=k^{(r,s,i)}\cdot \frac{\, d^i-\alpha^{(r,s,i)}_{i-1}d^{i-1}-\cdots -\alpha^{(r,s,i)}_{1} d-\alpha^{(r,s,i)}_{0}  }{d^i-a^{(r,s,i)}_{i-1}d^{i-1}-\cdots -a^{(r,s,i)}_{1} d-a^{(r,s,i)}_{0}}
$$
and
$$
q^{(r,s,i)}=\kappa^{(r,s,i)}\cdot \frac{d^{i-1}-b^{(r,s,i)}_{i-2}d^{i-2}-\ldots -b^{(r,s,i)}_{1} d-b^{(r,s,i)}_{0}}{d^i-a^{(r,s,i)}_{i-1}d^{i-1}-\cdots -a^{(r,s,i)}_{1} d-a^{(r,s,i)}_{0}},
$$
and $k^{(r,s,i)}, \kappa^{(r,s,i)}\in\{0,1\}$. More precisely,   $k^{(r,s,i)}=1$ if and only if $S_i^\star\left(v^{(r)}\right)\cap S_{D}^\star\left(w^{(r,s)}\right)\ne\emptyset$, as determined by statement (2) of Propositions~\ref{v2-second} and \ref{v2-second-cas 2},  and  $\kappa^{(r,s,i)}=1$ if and only if $S_i^\star\left(v^{(r)}\right)\cap S_{i-1}^\star\left(w^{(r,s)}\right)\ne\emptyset$, as determined by statement (3) of Propositions~\ref{v2-second} and \ref{v2-second-cas 2}. 
 
As in Theorem~\ref{p_ij}, the coefficients $a^{(r,s,i)}_{k}$, $\alpha^{(r,s,i)}_{k}$, $b^{(r,s,i)}_{k}$ are determined from $v^{(r)}$ and $w^{(r,s)}$ by using Proposition~\ref{propo 2.3} and Theorems~\ref{propo 2.9-v2-cas 1} and \ref{propo 2.9-v2-cas 2}. Furthermore, $a^{(r,s,i)}_{k}\in\{0,1\}$, $\alpha^{(r,s,i)}_{i-1}\in\{0,1,2\}$, and $\alpha^{(r,s,i)}_{k},  b^{(r,s, i)}_{k}\in\{0,1\}$ for $0\le k\le i-2$.

\setcounter{subsection}{0}
\section{Final remarks}


The digraphs $B(d,D)$ and $K(d,D)$ are fundamental examples of digraphs on alphabet \cite{gfy92} as well as iterated line digraphs \cite{Da17,fya-23.1}. Indeed, in the line digraph $L(G_0)$ of a digraph $G_0$ each vertex represents an arc $(x,y)$ of $G_0$; and a vertex $(x,y)$ is adjacent to a vertex $(z,t)$ if and only if $y=z$.  For any $k>1$, the $k$-iterated line digraph, $L^k(G_0)$, is defined recursively by $L^k(G_0)=L(L^{k-1}(G_0))$ (see for instance \cite{fya-23.1}). In particular, if $G_0$ is the complete symmetric digraph on $d$ vertices with a loop in each vertex, then $B(d,D)=L^{D-1}(G_0)$; and if $G_0$ is the  complete symmetric digraph on $d+1$ vertices without loops, then $K(d,D)=L^{D-1}(G_0)$. Other used network topologies correspond to iterated line digraphs as, for instance, the generalized De Bruijn cycles \cite{DeBruinj_gen_cycle}. So, we point out that an analysis of the distance-layer structure (and hence the evaluation of the efficiency of deflection routing in the corresponding network topology), similar to the one presented in this paper, could be done in other families of digraphs on alphabet or of iterated line digraphs.


\begin{thebibliography}{1}

\bibitem{origindeflection} P. Baran, \emph{On Distributed Communications Networks},  
IEEE Trans. Comm. Sys. \textbf{12} (1964), 1--9, 
doi: 10.1109/TCOM.1964.1088883

\bibitem{bls} J-C. Bermond, Z. Liu, and M. Syska,
\emph{Mean eccentricities of de Bruijn networks},
Networks \textbf{30} (3) (1997), 187--203, 
doi: 10.1002/(SICI)1097-0037(199710)30:3<187::AID-NET4>3.0.CO;2-H

\bibitem{CheTaMi} C. Chen, Z. Tao,  and J. S. Miguel,
\emph{Bufferless NoCs with Scheduled Deflection Routing}, 
2020 14th IEEE/ACM International Symposium on Networks-on-Chip (NOCS) (2020), 1--6, 
doi: 10.1109/NOCS50636.2020.9241585

\bibitem{kautz} J-C. Bermond and C. Peyrat,
\emph{De Bruijn and Kautz networks: A competitor for the hypercube?},
In: Andre, F., Verjus, J.P. (eds): Hypercube and Distributed Computers, North-Holland, Amsterdam (1989), 279--294.

\bibitem{BoDaHu} K. B\"{o}hmov\'{a}, C. Dalf\'{o}, and C. Huemer,
\emph{New cyclic {K}autz digraphs with optimal diameter},
Contrib. Discrete Math. \textbf{16} (3) (2021), 111--124

\bibitem{Da17} C. Dalf\'{o},
\emph{Iterated line digraphs are asymptotically dense},
Linear Algebra Appl. \textbf{529} (2017), 391--396, 
doi: 10.1016/j.laa.2017.04.036

\bibitem{debruijn} N. G. De Bruijn,
\emph{A combinatorial problem},
Koninkl. Nederl. Acad. Wetensch. Proc. Ser. A \textbf{49} (1946), 758--764.

\bibitem{DeBruinj_gen_cycle} J. G\'omez,  C. Padr\'o, and S. Perennes,
\emph{Large Generalized Cycles},
Discrete Appl. Math. \textbf{89} (1998), 107--123, 
doi: 10.1016/S0166-218X(98)00120-6

\bibitem{DoShMi} Y. Dong, E. Shan, and X. Min,
\emph{Distance domination of generalized de Bruijn and Kautz digraphs},
Front. Math. China \textbf{12} (2) (2017), 339--357,
doi: 10.1007/s11464-016-0607-y

\bibitem{endam}
J. F{\`a}brega, J. Mart\'{\i}, and X. Mu{\~n}oz,
\emph{Layer structure of De Bruijn and Kautz digraphs. An application to deflection routing}, 
Electron. Notes Discrete Math. \textbf{54} (2016), 157--162,
doi: 10.1016/j.endm.2016.09.028

\bibitem{fm} J. F{\`a}brega and X. Mu{\~n}oz, 
\emph{A Study of Network Capacity under Deflection Routing Schemes}.
In: Kosch H., Böszörményi L., Hellwagner H. (eds) Euro-Par 2003 Parallel Processing. Euro-Par 2003. Lecture Notes in Computer Science \textbf{2790} (2003), 989--994,
doi: 10.1007/978-3-540-45209-6\_135

\bibitem{FaMo} P. Faizian, M. A. Mollah, X. Yuan,  Z. Alzaid, S. Pakin, and M. Lang,
\emph{Random Regular Graph and Generalized De Bruijn Graph with $k$-Shortest Path Routing},
IEEE Trans. Parallel Distrib. Systems \textbf{29} (1) (2018), 144--155, 
doi: 10.1109/TPDS.2017.2741492
 
\bibitem{fya-23.1} M. A. Fiol,  J. L. A. Yebra,  and I. Alegre,
\emph{Line digraph iterations  and  the  $(d,k)$  digraph  problem}, IEEE  Trans. Comput. \textbf{C-33} (1984), 400--403,
doi: 10.1109/TC.1984.1676455
 
\bibitem{gfy92} J. G\'omez, M. A. Fiol, and J. L. A. Yebra, 
\emph{Graphs on alphabets as models for large interconnection networks}, Discrete Appl. Math. \textbf{37/38} (1992), 227--243, 
doi: 10.1016/0166-218X(92)90135-W

\bibitem{GriKaSte} C. Grigorious, T. Kalinowski,  and S. Stephen,
\emph{On the power domination number of de {B}ruijn and {K}autz digraphs},
Lecture Notes in Comput. Sci., \textbf{10765} (2018), 264--272,
doi: 10.1007/978-3-319-78825-8\_22

\bibitem{haeri2015} S. Haeri and L. Trajkovic,
\emph{Intelligent deflection routing in buffer-less networks},
IEEE Transactions on Cybernetics \textbf{45} (2) (2015), 316--327,
doi: 10.1109/TCYB.2014.2360680

\bibitem{kautz-2} W. H. Kautz,
\emph{Bounds on directed $(d,k)$ graphs}, 
Theory of Cellular Logic Networks and Machines, AFCRL-68-0668 Final Rept. (1968), 20--28.

\bibitem{kautz-3} W. H. Kautz,
\emph{Design of optimal interconnection networks for multiprocessors}, 
In: Architecture and Design of Digital Computers, Nato Advanced Summer Institute (1969), 249--272.

\bibitem{KuJaSl} R. G. Kunthara, R. K. James, S. Z. Sleeba, and J. Jose,
\emph{Traffic aware routing in 3D NoC using interleaved asymmetric edge routers},
Nano Communication Networks \textbf{27} (2021),
doi: 10.1016/j.nancom.2020.100334.

\bibitem{Lichi} N. Lichiardopol,
\emph{Quasi-centers and radius related to some iterated line digraphs, proofs of several conjectures on de Bruijn and Kautz graphs},
Discrete Appl. Math. \textbf{202} (2016), 106--110,
doi: 10.1016/j.dam.2015.08.025

\bibitem{MaEtYa} S. Marcovich, T. Etzion, and E. Yaakobi, 
\emph{Balanced de Bruijn Sequences}, 
2021 IEEE International Symposium on Information Theory (ISIT) (2021) 1528--1533,
doi: 10.1109/ISIT45174.2021.9517873

\bibitem{NlMu} B. Nleya and A. Mutsvangwa,
\emph{A Node-Regulated Deflection Routing Framework for Contention Minimization}, Journal of Computer Networks and Communications, \textbf{2020}, Article ID 2708357, 14 pages, 2020, doi: 10.1155/2020/2708357

\bibitem{PaSe} G. Panchapakesan and A. Sengupta,
\emph{On a lightwave network topology using Kautz digraphs},
IEEE Trans. Comput., \textbf{48} (10) (1998), 1131--1137,
doi: 10.1109/12.805162

\bibitem{SaTv} P. Salinger and P. Tvrdfk,
\emph{All-to-All Scatter in Kautz networks},
In: Pritchard D., Reeve J. (eds) Euro-Par'98 Parallel Processing. Euro-Par 1998. Lecture Notes in Computer Science \textbf{1470} (1998), 1057--1061,
doi: 10.1007/BFb0057967

\bibitem{SheLi} H. Shen and Z. Li,
\emph{A Kautz-Based Wireless Sensor and Actuator Network for Real-Time, Fault-Tolerant and Energy-Efficient Transmission},
IEEE Trans. on Mobile Computing, \textbf{15} (1) (2016), 1--16,
doi: 10.1109/TMC.2015.2407391

\bibitem{SheZhaGu} Y. Sheng, Y. Zhang, H. Guo, S. Bose, and G. Shen,
\emph{Benefits of Unidirectional Design Based on Decoupled Transmitters and Receivers in Tackling Traffic Asymmetry for Elastic Optical Networks}, 
J. Opt. Commun. Netw. \textbf{10} (2018),  C1-C14,
doi: 10.1364/JOCN.10.0000C1

\bibitem{zheng2015} X. Zheng, Y. Hu, D. Luo, and X. Wu,
\emph{Study of Deflection Routing from an Information-theoretic Perspective},
 International Journal of Future Generation Communication and Networking,
 \textbf{8} (1) (2015), 227--236,
 doi: 10.14257/ijfgcn.2015.8.1.23

\bibitem{ZeGyBl} J. Zerwas, C. Gy\"orgyi, A. Blenk, S. Schmid, and C. Avin, \emph{Kevin: de Brujin-based topology with demand-aware links and greedy routing}, arXiv preprint (2022),  arXiv:2202.05487.
\end{thebibliography}
\end{document}